\documentclass[leqno]{mnarticle} 

\usepackage{amssymb}
\usepackage{xcolor}
\usepackage[normalem]{ulem}
\usepackage{graphicx} 
\usepackage{xspace}
\usepackage{multirow}
\usepackage{bbm}
\usepackage{Macros/mydef} 
\usepackage{algpseudocode}
\usepackage[numbers]{natbib}
\let\cite=\citet
\usepackage{enumerate}
\usepackage{booktabs}
\usepackage{graphicx}
\usepackage{subcaption}
\usepackage{todonotes}
\usepackage{thmtools}
\usepackage{cleveref}

\usepackage{graphicx}
\usepackage{subcaption}
\usepackage{array}
\usepackage{graphicx}
\usepackage{booktabs}
\usepackage{cite}

\ifpdf
  \DeclareGraphicsExtensions{.eps,.pdf,.png,.jpg}
\else
  \DeclareGraphicsExtensions{.eps}
\fi

\newsiamremark{remark}{Remark}
\newsiamremark{hypothesis}{Hypothesis}
\crefname{hypothesis}{Hypothesis}{Hypotheses}
\newsiamthm{claim}{Claim}

\headers{Graph-Poisson limiting}{Murtazo Nazarov}

\title{High-Order Invariant-Domain Preserving Continuous Finite Elements via Graph-Poisson Convex Limiting
  \thanks{%
\funding{This research is funded by Swedish Research Council (VR) under Grant Number 2021-04620.}}}

\author{
Murtazo Nazarov\thanks{
Division of Scientific Computing, Department of Information Technology,
Uppsala University, Uppsala 751 05, Sweden
(\email{murtazo.nazarov@uu.se}).
}
}

\usepackage{amsopn}

\usepackage{xcolor}
\usepackage[normalem]{ulem}
\usepackage[most]{tcolorbox}
\usepackage{cancel}

\newtcolorbox{delblock}{breakable, enhanced, colback=white,
  colframe=red!75!black, boxrule=0.4pt, arc=0pt, left=2pt, right=2pt,
  fontupper=\color{red},
  overlay={\draw[red!75!black, line width=0.6pt]
    (frame.south west)--(frame.north east);
    \draw[red!75!black, line width=0.6pt]
    (frame.north west)--(frame.south east);}}

\newcommand{\Mone}{M_1}
\newcommand{\Mk}{M_k}

\usepackage{etoc}

\begin{document}

\maketitle

\begin{abstract} 
  We develop a high-order, locally conservative, invariant-domain preserving continuous finite element method for nonlinear scalar conservation laws. A first-order graph-viscosity scheme is built on the fine-$\polP_1$ submesh induced by the high-order nodes, so that it remains compact and has positive lumped masses for every polynomial degree, while the high-order candidate can be any stabilized Galerkin method; here we use residual-based artificial viscosity and continuous interior penalty methods. The two updates conserve different discrete masses, and comparing each in its own conserved measure gives a nodal difference with zero sum, which a graph-Poisson problem represents as antisymmetric edge fluxes. For polynomial degrees with positive $\polP_k$ lumped masses (\eg $\polP_1$, $\polP_3$, $\polP_5$ on triangles), convex limiting of these fluxes conserves the $\polP_k$ mass exactly and in local flux form. For degrees with zero or negative lumped masses (\eg $\polP_2$, $\polP_4$), limiting perturbs the mass only on the elements where it acts, and this perturbation is removed inside each such element without violating the bounds, exactly whenever the element has enough room within its bounds. Numerical experiments demonstrate high-order accuracy for smooth solutions, robustness near discontinuities, and agreement with the entropy solution on standard benchmarks.
\end{abstract}



\pagestyle{myheadings} \thispagestyle{plain} \markboth{M. NAZAROV}{Graph-Poisson limiting}

\section{Introduction}
\label{Sec:introduction}
In this paper, we are interested in the numerical approximation of nonlinear scalar conservation laws of the form
\begin{equation}\label{eq:pde}
  \begin{cases}
    \p_t u + \DIV \bef(u) = 0, \quad (\bx,t) \in \Omega\times(0,T], \\
    u(\bx,0) = u_0(\bx), \quad \bx \in \Omega,
  \end{cases}
\end{equation}
where $\Omega\subset\polR^d$ is a bounded open domain, $d$ is the space dimension, $u_0$ is the prescribed initial data, and $T>0$ is the final time. We assume that \eqref{eq:pde} is equipped with appropriate boundary conditions. The flux $\bef:\calA\to\polR^d$ is assumed to be Lipschitz continuous on an invariant domain $\calA\subset\polR$. By an invariant domain we mean that, if the initial data takes values in $\calA$, then the exact entropy solution remains in $\calA$ for all later times \citep{Guermond_Popov_2016}. For a scalar conservation law the invariant domain is simply the interval $[u_{\min}, u_{\max}]$, where $u_{\min}:=\operatorname*{ess\,inf} u_0$ and $u_{\max}:=\operatorname*{ess\,sup} u_0$, and preserving it is a maximum principle; when Dirichlet conditions are imposed, we assume that the boundary data take values in $[u_{\min}, u_{\max}]$. Following \citep{Guermond_Popov_2016,Guermond_Nazarov_Popov_Tomas_2018}, a numerical method is called invariant-domain preserving (IDP) if every convex invariant set of \eqref{eq:pde} is invariant for one time step of the method, the states being the nodal values. For a scalar conservation law every interval is a convex invariant set, so this means that if all nodal values $\sfU_i^n$ lie in an interval $[a,b]$, then so do all $\sfU_i^{n+1}$; in particular they remain in $[u_{\min}, u_{\max}]$. 

The main objective of this work is to construct high-order finite element approximations of \eqref{eq:pde} that are locally conservative and preserve the invariant-domain property. In particular, we are interested in methods that retain accuracy beyond second order in space. High-order invariant-domain preserving (IDP) schemes for hyperbolic problems are most commonly based on finite volume or discontinuous Galerkin (DG) discretizations. We refer the reader to \citep{Sanders_1988,Liu_Tadmor_1998,Zhang_2010} and the references therein for finite volume methods, and to \citep{Zhang_Xia_Shu_2012,Hajduk_2021,Panzer_2021} and the references therein for DG methods.

In contrast, comparable high-order IDP constructions for continuous finite element methods are less developed. Robust IDP continuous finite element schemes restricted to first- or second-order accuracy have been proposed in several works; see, for instance, \citep{Kuzmin_2002,Kuzmin_Moller_2005,Guermond_Nazarov_2014,Guermond_Nazarov_Popov_Yong_2014,Guermond_Popov_2017} and the references therein. Most high-order limiting approaches rely on three main ingredients: a robust low-order IDP scheme, a high-order scheme that is not necessarily IDP, and a limiting procedure that combines the two. Many high-order finite element stabilizations, including entropy-viscosity and residual-viscosity methods, can be used as the high-order ingredient; see for instance, \citep{Guermond_pasquetti_popov_JCP_2011,Nazarov_Larcher_2017,Dao_Nazarov_2022}. The low-order IDP scheme can be constructed by discrete upwinding, as in \citep{Kuzmin_2002,Kuzmin_Moller_2005}, or by the {\em graph-viscosity} method of \citep{Guermond_Nazarov_2014} and its extension to general systems of conservation laws \citep{Guermond_Popov_2016}. Once the low- and high-order schemes are constructed, a wide family of well-established limiting algorithms can be employed, such as flux-corrected transport (FCT) \citep{Zalesak_1979,Kuzmin_2002}, algebraic flux correction (AFC) \citep{Kuzmin_Moller_2005,Kuzmin_2020}, and convex limiting \citep{Guermond_Nazarov_Popov_Tomas_2018,Guermond_Nazarov_Popov_2024}. In these frameworks, the difference between the low- and high-order discretizations is decomposed into conservative antidiffusive graph fluxes, which are then limited.
We note that for continuous and discontinuous high-order finite elements, the localized FCT limiters of \citep{Cotter_Kuzmin_2016,Lohmann_etal_2017} constrain intermediate states by local bounds and write the limited update as a convex combination; the edge-based convex limiting of \citep{Guermond_Nazarov_Popov_Tomas_2018}
preserves invariant sets described by quasiconcave constraints, also for systems, and is closely related to these localized FCT limiters: both write the limited update as a convex combination of states constrained by local bounds. We use the terminology of \citep{Guermond_Nazarov_Popov_Tomas_2018} (graph viscosity for discrete upwinding, convex limiting for localized FCT), keeping in mind this common origin. 

A major difficulty in extending this methodology to high-order continuous Lagrange finite elements lies in the construction of the low-order IDP scheme. The first-order graph-viscosity schemes developed in \citep{Guermond_Nazarov_2014,Guermond_Popov_2016}, as well as many low-order schemes used in FCT and AFC methods, are edge-based and are naturally suited to piecewise-linear spaces. For higher-order Lagrange spaces, the graph associated with the finite element stencil becomes wider, and the resulting low-order method requires smaller time steps and becomes more diffusive as the degree increases, as observed in \citep[Sec.~3.3.2.1]{Fahad_2015}, \citep[Sec.~3.2.2]{Quezada_2016} and \citep[Sec.~4.1, Fig.~4]{Guermond_Nazarov_Popov_2024}. This phenomenon was also observed in \citep{Anderson_etal_2017} for discontinuous Galerkin discretizations, even with positive Bernstein bases and studied in detail in \citep{Lohmann_etal_2017}; it motivated the sparse subcell approximations of \citep{Kuzmin_Quezada_2020,Hajduk_etal_2020a,Hajduk_etal_2020b} and subsequently \citep{Hajduk_2021,Panzer_2021}. A second difficulty concerns mass lumping. The low-order IDP scheme requires a lumped mass matrix with positive entries \citep{Guermond_Popov_Yang_2017}. However, standard high-order Lagrange elements do not always have nonnegative lumped masses. For example, on simplices, the lumped masses are positive for $\polP_1$, nonnegative for $\polP_2$ (zero at the vertices), and positive for $\polP_3$ and $\polP_5$, whereas negative lumped masses generally appear for other polynomial degrees. This creates a significant obstacle to constructing arbitrarily high-order IDP continuous finite element schemes using the standard graph-viscosity machinery directly on the high-order space.

\paragraph{Contributions} The first contribution of this paper is to remove this bottleneck by separating the spaces used for the low- and high-order schemes. The high-order method can be constructed in the desired polynomial space $\polP_k$ by any stabilized Galerkin method that conserves the $\polP_k$ integral $\int_\Omega u_h\,\ud\bx$, such as SUPG \citep{Brooks_Hughes_1982}, continuous interior penalty \citep{Burman_Hansbo_2004}, subgrid viscosity \citep{Guermond_1999}, or entropy and residual viscosity \citep{Guermond_pasquetti_popov_JCP_2011,Nazarov_Larcher_2017}. The low-order IDP scheme, however, is built on a fine-$\polP_1$ submesh whose vertices coincide with the nodal points of the high-order space. Thus the low-order method remains a first-order IDP scheme on a piecewise-linear graph, while the high-order method retains the accuracy of the high-order polynomial approximation.

The second contribution concerns conservation. The low-order scheme conserves the fine-$\polP_1$ lumped mass $\sum_i m_i^{\polP_1}\sfU_i$, whereas the high-order scheme conserves the $\polP_k$ integral $\sum_i m_i^{\polP_k}\sfU_i=\int_\Omega u_h\,\ud\bx$. We show that comparing each update in its own conserved measure,
\[
  \sfA_i=m_i^{\polP_k}\big(\sfU_i^{\rm H}-\sfU_i^n\big)-m_i^{\polP_1}\big(\sfU_i^{\rm L}-\sfU_i^n\big),
\]
gives a nodal vector with zero sum, exactly for periodic or compactly supported problems and up to the difference of the two discrete boundary fluxes in general (Lemma~\ref{lem:zero_sum}). A graph-Poisson problem on the fine-$\polP_1$ graph represents it as antisymmetric edge fluxes, with no global mass shift under the assumption of Lemma~\ref{lem:zero_sum}. The fluxes are limited by symmetric convex limiting. Whenever all $\polP_k$ lumped masses are positive (for instance $\polP_1$, $\polP_3$ and $\polP_5$ on triangles), the resulting update is in local flux form with respect to the $\polP_k$ masses: it preserves the local bounds, conserves $\int_\Omega u_h$ exactly, and returns the high-order update when no flux is limited (Theorem~\ref{t:ho_max_p} and Proposition~\ref{thm:pk_form}). For $\polP_2$ and $\polP_4$, whose lumped masses vanish or are negative at some nodes, this form is not available; the update is then written with fine-$\polP_1$ masses and a nodal lumping source, and the $\polP_k$-mass defect, which is supported on the limited elements, is removed by an element-local exchange (Section~\ref{Sec:local_exchange}).

\paragraph{Related work} Representing a zero-sum nodal difference by antisymmetric fluxes through a (graph) Laplacian is classical; see, \eg, the conservation analysis of residual distribution schemes in \citep[Lemma~4.2]{Abgrall_2018}. Element-local versions were used to generate subcell fluxes from zero-sum element contributions in \citep[eqs.~(45)--(47)]{Kuzmin_Quezada_2020} and \citep[eq.~(4.5)]{Hajduk_2021}, and for a posteriori subcell correction of DG schemes in \citep[Sec.~2.1]{Vilar_Abgrall_2024}. The local reconstructions only require small element matrices, whereas the present reconstruction requires one global graph-Laplacian solve per forward-Euler substep of the Runge--Kutta method, because the consistent mass matrix of the high-order candidate couples all nodes. The advantage of the present approach is that the high-order candidate needs no edge-flux or subcell-residual structure, as required, \eg in \citep{Kuzmin_Quezada_2020,Hajduk_2021,Vilar_Abgrall_2024}; the limiter can therefore be added to an existing conservative high-order code without changing it. 

Bernstein bases are commonly used to avoid nonpositive lumped masses, combined with subcell flux limiting and residual distribution in \citep{Lohmann_etal_2017,Hajduk_etal_2020a,Kuzmin_Quezada_2020}. The present algorithm would apply to Bernstein bases, for which all lumped masses are positive and the $\polP_k$ form is available for every degree; our objective, however, is standard continuous Lagrange finite elements. 

The mismatch between the masses of a low-order scheme and of a high-order Lagrange scheme was addressed in \citep{Guermond_Nazarov_Popov_2024}, in the spirit of the residual distribution schemes of \citep{Abgrall_etal_2017}: the low-order method is restricted to a nearest-neighbour stencil and its fluxes are modified so that it carries the same generalized mass as the high-order method, which allows convex limiting beyond second order. In the present approach the low-order scheme is not modified; the mass difference is instead absorbed into the antidiffusive fluxes by the graph-Poisson reconstruction.

Mass-conservative limiting by local redistribution of nodal states was introduced in \citep{Guermond_Wang_2025}. Concurrently with this work, \citep{Guermond_2026} proposed a cell-based, locally conservative redistribution limiter on the same hierarchical $\polP_1$ subdivision of high-order Lagrange elements. That limiter acts directly on nodal states, assumes nonnegative masses, and guarantees the bounds under a local budget condition, with a global step as fallback; it is formulated for scalar problems and for systems such as the Euler equations. Its conservation is cell-wise in terms of nodal states, whereas the present limiter is in edge-flux form. The present correction guarantees the bounds without any budget condition and addresses nonpositive Lagrange lumped masses explicitly. For $\polP_1$ and $\polP_3$ it conserves the $\polP_k$ mass exactly; for $\polP_2$ and $\polP_4$ exact conservation requires that every limited element has enough capacity for the element-local exchange of Section~\ref{Sec:local_exchange}, which is of the same borrow-and-lend type as in \citep{Guermond_2026}; in our tests this holds for smooth solutions, and for solutions with discontinuities the unrecovered mass stays below about $3\times10^{-3}$ of the total mass.

The rest of the paper is organized as follows. In Section~\ref{sec:prelim}, we introduce the finite element spaces and meshes used throughout the paper. Sections~\ref{Sec:low} and~\ref{Sec:high} describe the low-order and high-order schemes, respectively. The main limiting algorithm, including the graph-Poisson flux reconstruction and convex limiting procedure, is presented in Section~\ref{Sec:limiter}, together with the conservation analysis and the Lax--Wendroff discussion. Finally, in Section~\ref{Sec:numerics}, we solve several benchmark problems to assess the accuracy, robustness, and invariant-domain properties of the proposed method. Conclusions and an outlook to systems are given in Section~\ref{sec:conclusions}.

\section{Preliminaries}\label{sec:prelim}
In this section, we present the triangulation of the domain of interest, the construction of finite element spaces, and some definitions that will be needed in our further analysis.

\subsection{Finite element approximations}
Let us denote by $\calT_h$ a subdivision of $\Omega$ into finite number of disjoint elements $K$ such that $\overline\Omega=\cup_{K\in \calT_h} \overline K$, where $\overline \Omega$ and $\overline K$ denote the closure of $\Omega$ and $K$, respectively. We consider a family of shape-regular and conforming meshes $\{\calT_h\}_{h>0}$, where $h$ denotes the largest diameter of all triangles of $\calT_h$. We denote by $\bg_K: \widehat K \mapsto K$ the affine mapping that maps the reference element $\widehat K$ to $K$. 

We define finite element spaces that we use. For each mesh $\calT_h$ we associate the following continuous approximation space:
\begin{equation}\label{eq:Xh}
  \calX_h := \{v_h: v_h \in \calC^0(\overline \Omega); \, \forall K\in \calT_h, \,
  v_h|_K \circ \bg_K \in \polP_k \},
\end{equation}
where $\polP_k$ is the set of multivariate polynomials of total degree at most $k\ge1$ defined over $\widehat K$. We also denote by $I$ the total number of Lagrange nodes in the mesh $\calT_h$, and $\{ \bN_1, \ldots, \bN_I \}$ is the collection of all the Lagrange nodes and $\{\varphi_1, \cdots, \varphi_I\}$ is the set of corresponding scalar-valued shape functions, and $S_i$ is the support of $\varphi_i$, and $S_{ij} = S_i \bigcap S_j$ is the intersection of the supports of $\varphi_i$ and $\varphi_j$. We often refer to the index set of all degrees of freedom by $\calV:=\{1,2,\ldots, I\}$.  We denote by $\calI(S_i)$ and $\calI(K)$ the set of all indices of the shape functions living at $S_i$ and cell $K$, respectively.

We denote by $m^{\polP_k}_{ij} := \int_\Omega \varphi_j(\bx) \varphi_i(\bx) \ud \bx$ the consistent mass matrix, and by $m^{\polP_k}_i := \sum_{j\in \calV} m^{\polP_k}_{ij} = \int_\Omega \varphi_i(\bx) \ud \bx$ the lumped mass matrix. Note that, we used here the partition of unity property of the test functions, \ie $\sum_{j\in\calV} \varphi_j = 1$. 

Next, we denote by $\calT_h^{\polP_1, \text{fine}}$ the mesh whose vertices coincide with the nodes of the finite element space $\calX_h$. For example, $\calT_h^{\polP_1, \text{fine}} = \calT_h$, when the polynomial space is $k=1$, \ie $\polP_1$. Examples of patches $\calT_h^{\polP_1, \text{fine}}$ and $\calT_h$ for the case of $\polP_2$ and $\polP_3$ are shown in Figure~\ref{fig:meshes}. The fine mesh defines a graph with vertex set $\calV$; we denote by $\calE$ its edge set, i.e., the set of unordered pairs $\{i,j\}$, $i\ne j$, such that $\bN_i$ and $\bN_j$ are vertices of a common triangle of $\calT_h^{\polP_1, \text{fine}}$. All edge fluxes in this paper are defined on $\calE$.

In addition, we need a continuous piecewise linear finite element space on the fine mesh $\calTh^{\polP_1, \text{fine}}$ that will be useful in constructing the first-order viscosity used in this paper:
\begin{equation}\label{eq:Xhfine}
  \calX^{\polP_1, \text{fine}}_h := \{v_h: v_h \in \calC^0(\overline \Omega); \, \forall K\in \calT^{\polP_1, \text{fine}}_h, \,
  v_h|_K \circ \bg_K \in \polP_1 \},
\end{equation}
and let us denote the shape function of $\calX^{\polP_1, \text{fine}}_h$ by $\varphi_i^{\polP_1, \text{fine}}$. Therefore, the lumped mass matrix corresponding to this fine space is 
$	
m^{\polP_1, \textrm{fine}}_i := \int_{S^{\polP_1, \textrm{fine}}_i}\varphi_i^{\polP_1, \textrm{fine}} \ud \bx,
$
where $S_i^{\polP_1, \textrm{fine}} = \text{supp }(\varphi_i^{\polP_1, \text{fine}})$.
In order to make notation simpler, we further use
$
m^{\polP_1}_i
$
instead of
$
m^{\polP_1, \textrm{fine}}_i
$
and
$
S_i^{\polP_1}
$
instead of
$
S_i^{\polP_1, \textrm{fine}}.
$

\subsection{Two discrete masses and their element splits}\label{sec:masses}
The coarse lumped masses $m_i^{\polP_k}=\int_\Omega\varphi_i\,\ud\bx$ are \emph{not} positive in general. On a triangle, the reference values of $\int_{\widehat K}\varphi_i \ud \bx /|\widehat K| $ for equispaced Lagrange nodes are given in Table~\ref{tab:lumped}: they vanish at the vertices for $\polP_2$ and $\polP_4$ and are negative at some edge nodes for $\polP_4$.
Both families of masses split over the coarse elements: for $K\in\calT_h$ and $i\in\calI(K)$ let
\[
  m_{i,K}^{\polP_1}:=\sum_{T\subset K}\int_T\varphi_i^{\polP_1}\,\ud\bx,\qquad
  m_{i,K}^{\polP_k}:=\int_K\varphi_i\,\ud\bx,
\]
where the first sum runs over the fine-$\polP_1$ triangles contained in $K$. Then $m_i^{\polP_1}=\sum_{K\ni i}m_{i,K}^{\polP_1}$, $m_i^{\polP_k}=\sum_{K\ni i}m_{i,K}^{\polP_k}$ and, since both are partitions of unity on $K$,
\begin{equation}\label{eq:zero_sum_K}
  \sum_{i\in\calI(K)}m_{i,K}^{\polP_1}=\sum_{i\in\calI(K)}m_{i,K}^{\polP_k}=|K| .
\end{equation}
We write
\[
  \Mone(\sfW):=\sum_{i\in\calV}m_i^{\polP_1}\sfW_i,\qquad
  \Mk(\sfW):=\sum_{i\in\calV}m_i^{\polP_k}\sfW_i=\int_\Omega \sum_{i\in\calV}\sfW_i\varphi_i\,\ud\bx,
\]
and $E(\sfW):=\Mone(\sfW)-\Mk(\sfW)=\sum_{K\in \calT_h} E_K(\sfW)$ with $E_K(\sfW):=\sum_{i\in\calI(K)}(m_{i,K}^{\polP_1}-m_{i,K}^{\polP_k})\sfW_i$. By \eqref{eq:zero_sum_K}, $E_K$ vanishes on constants, so that
\begin{equation}\label{eq:EK_bound}
  |E_K(\sfW)|\le C\,|K|\operatorname{osc}_K\sfW,\qquad
  \operatorname{osc}_K\sfW:=\max_{i\in\calI(K)}\sfW_i-\min_{i\in\calI(K)}\sfW_i .
\end{equation}

\begin{table}[htbp]
  \caption{Reference lumped masses $\int_{\widehat K}\varphi_i\,\ud\bx/|\widehat K|$ of equispaced Lagrange elements on triangles, by node type, and the fine-$\polP_1$ counterparts $m_{i,K}^{\polP_1}/|K|$. The last column gives $\max_i m_i^{\polP_1}/m_i^{\polP_k}$ over the nodes with $m_i^{\polP_k}>0$; it is mesh independent.}
  \label{tab:lumped}
  \centering
  \begin{tabular}{ccccc}
    \toprule
    & vertex & edge & interior & $\max m^{\polP_1}/m^{\polP_k}$\\
    \midrule
    $\polP_1$ & $1/3$ \ ($1/3$) & -- & -- & 1\\
    $\polP_2$ & $0$ \ ($1/12$) & $1/3$ \ ($1/4$) & -- & -- \\
    $\polP_3$ & $1/30$ \ ($1/27$) & $3/40$ \ ($1/9$) & $9/20$ \ ($2/9$) & $1.48$\\
    $\polP_4$ & $0$ \ ($1/48$) & $-1/45$,\ $4/45$ \ ($1/16$) & $8/45$ \ ($1/8$) & -- \\
    $\polP_5$ & $0.0109$ \ ($1/75$) & $0.0248$ \ ($1/25$) & $0.0248$,\ $0.1984$ \ ($2/25$) & $3.23$\\
    \bottomrule
  \end{tabular}
\end{table}

\begin{figure}[h!]
  \centering
  \begin{subfigure}{0.25\textwidth}
    \centering
    \includegraphics[width=\textwidth]{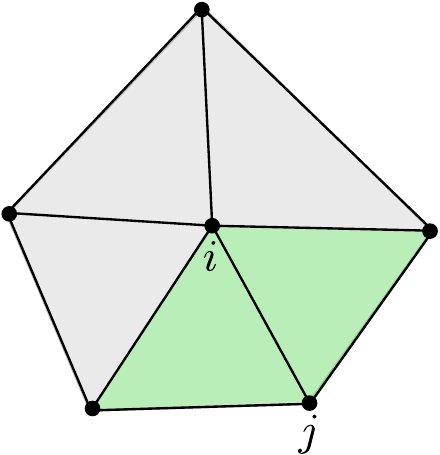}
    \caption*{$\polP_1$}
  \end{subfigure}
  \hspace{0.2in}
  \begin{subfigure}{0.25\textwidth}
    \centering
    \includegraphics[width=\textwidth]{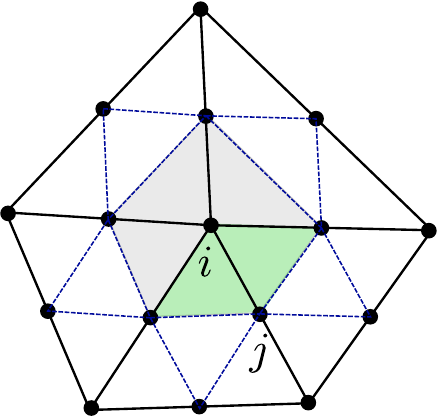}
    \caption*{$\polP_2$}
  \end{subfigure}
  \hspace{0.2in}
  \begin{subfigure}{0.25\textwidth}
    \centering
    \includegraphics[width=\textwidth]{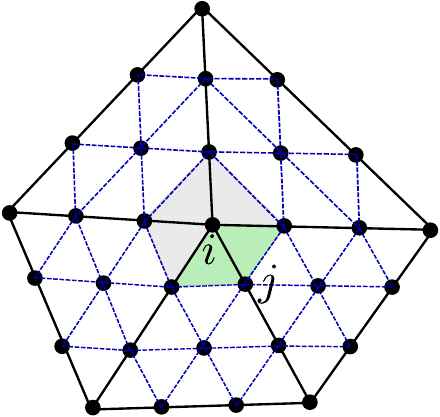}
    \caption*{$\polP_3$}
  \end{subfigure}
  \caption{
    Nodal distribution and sub-meshes for different polynomial spaces.
    The gray patch including the green triangles is $S_i^{\polP_1}$, and the green triangles are $S_{i,j}^{\polP_1}$, which are the support of the edge $(i,j)$.}
  \label{fig:meshes}
\end{figure}

\section{Low order scheme}
\label{Sec:low}
Recall that $m_i^{\polP_1}$ denotes the lumped mass of the fine-$\polP_1$ space at the node $\bN_i$. Let $u_h^n:= \sum_{i\in \calV} \sfU^n_i \varphi_i$ be approximate solution at $t=t_n$. 

Let ${T} \in \calT_h^{\polP_1}$ be a fine-$\polP_1$ triangle with local nodes $\calI({T}) = \{i,j,l\}$. For every edge $(i,j) \subset {T}$, we define the vector
\[
  \bc_{ij}^{T}
  :=
  \int_{T} \varphi_i^{\polP_1}\nabla \varphi_j^{\polP_1}\,{\rm d}\bx .
\]
Since $\nabla\varphi_j^{\polP_1}$ is constant on ${T}$, this can be written as
$
\bc_{ij}^{T}
=
m_{i,{T}}^{\polP_1}\nabla\varphi_j^{\polP_1}\big|_{T} ,
$
where $m_{i,{T}}^{\polP_1}= \int_{T} \varphi_i^{\polP_1} \ud \bx =  \frac{|{T}|}{d+1}$ is the $\polP_1$ mass matrix in the cell ${T}$.

Next, we define local graph-viscosity associated to the edge $(i,j)$ of the cell ${T}$ by
\[
  d_{ij}^{n,{T}}
  =
  \sup_{w\in[\sfU_{T}^{n,\min},\,\sfU_{T}^{n,\max}]}
  \left|
    \bef'(w)\cdot \bc_{ij}^{T}
  \right|,
\]
where $\sfU_{T}^{n,\min}:=\min_{l\in\calI({T})}\sfU_l^n$ and $\sfU_{T}^{n,\max}:=\max_{l\in\calI({T})}\sfU_l^n$. If $\bef'$ is affine, as for Burgers' equation, $w\mapsto|\bef'(w)\cdot\bc_{ij}^{T}|$ is convex, the supremum is attained at a vertex value, and $d_{ij}^{n,{T}}=\max_{l\in\calI({T})}|\bef'(\sfU_l^n)\cdot\bc_{ij}^{T}|$; for a linear flux $\bef(u)=\boldsymbol\beta(\bx)u$ the same formula holds with $\boldsymbol\beta$ at the vertices. For a general flux, $d_{ij}^{n,{T}}\le\|\bc_{ij}^{T}\|\sup_w\|\bef'(w)\|$ can be used; for the KPP flux $\bef(u)=(\sin u,\cos u)$ this gives $d_{ij}^{n,{T}}=\|\bc_{ij}^{T}\|$.
The symmetric low-order edge viscosity is then
\begin{equation}\label{eq:dij}
  d_{ij}^{{\rm L},n, {T}}
  :=
  \max
  \big(
  d_{ij}^{n,{T}},d_{ji}^{n,{T}}
  \big),
  \quad i\ne j,\quad i,j\in\calI({T}),
\end{equation}

For every $i\in \calI({T})$, we define the local graph viscosity contribution by
\[
  b_{T}^{{\rm L}}(u_h^n,\varphi_i^{\polP_1})
  :=
  \sum_{j\in\calI({T}),\, j\not=i}
  d_{ij}^{{\rm L}, n,{T}}
  \big(\sfU_i^n-\sfU_j^n\big).
\]
This edge-difference form is locally conservative: the contribution
associated with the edge $(i,j)$ is equal and opposite in the equations
for nodes $i$ and $j$.

The global artificial viscosity operator is obtained by summing all fine cell contributions containing node $i$:
\begin{equation}\label{eq:b_h}
  b_h^{{\rm L}}(u_h^n, \varphi_i^{\polP_1}) 
  = 
  \sum_{{T}\subset S_i^{\polP_1}} b_{T}^{{\rm L}}(u_h^n,\varphi_i^{\polP_1})
  =
  \sum_{{T}\subset S_i^{\polP_1}} \sum_{j\in\calI({T}),\, j\not=i}
  d_{ij}^{{\rm L}, n,{T}}
  \big(\sfU_i^n-\sfU_j^n\big).
\end{equation}

Finally, we define the first order method in the fine-$\polP_1$ mesh as follows:
\begin{equation}\label{eq:LO}
  m_i^{\polP_1}
  \frac{\sfU_i^{\mathrm L,n+1}-\sfU_i^n}{\dt}
  +
  \int_{S_i^{\polP_1}}
  \DIV \big(I_h^{\polP_1}\bef(u_h^n)\big)
  \varphi_i^{\polP_1}\,\ud\bx
  +
  b_h^{\rm L}\big(u_h^n,\varphi_i^{\polP_1}\big)
  =0 ,
\end{equation}
for every $i\in \calV$, where $I_h^{\polP_1}\bef(u_h^n)=\sum_j\bef(\sfU_j^n)\varphi_j^{\polP_1}$ is the nodal interpolant of the flux on the fine-$\polP_1$ mesh, so that the flux term is $\sum_{{T}\subset S_i^{\polP_1}}\sum_{j\in\calI({T})}\bef(\sfU_j^n)\SCAL\bc_{ij}^{T}$.

\begin{proposition}
  The low-order scheme \eqref{eq:LO} is conservative with respect to the
  fine-$\polP_1$ lumped mass, i.e.,
  \[
    \sum_{i\in\calV} m_i^{\polP_1}\sfU_i^{\mathrm L,n+1}
    =
    \sum_{i\in\calV} m_i^{\polP_1}\sfU_i^{n}.
  \]
\end{proposition}

\begin{proof}
  Multiplying \eqref{eq:LO} by $\dt$ and summing over all $i\in\calV$ and using the partition of unity property of $\varphi_i$, and conservation properties of
  $
  b_h^\textrm{L}(u^n_h, \varphi_i^{\polP_1})
  $
  proves the statement. 
\end{proof}

Let us denote
\begin{equation}\label{eq:aij}
  a_{ij}^{n,{T}}
  :=
  \bc_{ij}^{T}\SCAL\int_0^1\bef'\big(\sfU_i^n+s(\sfU_j^n-\sfU_i^n)\big)\,\ud s,
\end{equation}
so that $\big(\bef(\sfU_j^n)-\bef(\sfU_i^n)\big)\SCAL\bc_{ij}^{T}=a_{ij}^{n,{T}}(\sfU_j^n-\sfU_i^n)$ and, since the intermediate values lie in $[\sfU_{T}^{n,\min},\sfU_{T}^{n,\max}]$, $|a_{ij}^{n,{T}}|\le d_{ij}^{n,{T}}\le d_{ij}^{{\rm L},n,{T}}$,
and
\begin{equation}\label{eq:gij}
  \gamma_{ij}^{n}
  :=
  \frac{\dt}{m_i^{\polP_1}}
  \sum_{{T}\subset S_{ij}^{\polP_1}}
  \left(
    d_{ij}^{{\rm L},n,{T}}
    -
    a_{ij}^{n,{T}}
  \right),
  \quad i\ne j.
\end{equation}

\begin{theorem}[Discrete Maximum Principle]\label{Thm:Min_Max_first_order}
  Let $u_h^n=\sum_{i\in \cal V} \sfU_i^n \varphi_i$ be an approximation at time $t_n$ and $\sfU_i^{{\rm L},n+1}$ be as in \eqref{eq:LO}. Assume that the time-step $\dt$ satisfies the following CFL condition
  \begin{equation}\label{eq:cfl}
    \max_{i\in\calV} \sum_{j \in \calI( S_i^{\polP_1}), j\not=i}
    \gamma^n_{ij}(\dt)
    \le \nu,
  \end{equation}
  where $0 < \nu \le 1$ is a CFL number.  
  Then for any $i\in \calV$ we have
  \begin{equation}\label{eq:UL_max}
    \sfU_i^{n,\min} \le \sfU_i^{{\rm L},n+1}
    \le \sfU_i^{n,\max},
  \end{equation}
  where 
  \begin{equation}\label{eq:Umax-Umin}
    \sfU_i^{n,\min}:=\min_{j\in\calI(S_i^{\polP_1})}\sfU_j^{n},\quad \textrm{and} \quad \sfU_i^{n,\max}:=\max_{j\in\calI(S_i^{\polP_1})}\sfU_j^{n}.
  \end{equation}
\end{theorem}

\begin{proof}
  Since $\sum_{j\in\calI({T})}\bc_{ij}^{T}=\int_{T}\varphi_i^{\polP_1}\nabla\big(\sum_j\varphi_j^{\polP_1}\big)\,\ud\bx=0$, the flux term of \eqref{eq:LO} is
  \[
    \begin{aligned}
      \sum_{{T}\subset S_i^{\polP_1}}\sum_{j\in\calI({T})}\bef(\sfU_j^n)\SCAL\bc_{ij}^{T}
      &=\sum_{{T}\subset S_i^{\polP_1}}\sum_{j\in\calI({T}),\,j\ne i}\big(\bef(\sfU_j^n)-\bef(\sfU_i^n)\big)\SCAL\bc_{ij}^{T}\\
      &=\sum_{{T}\subset S_i^{\polP_1}}\sum_{j\in\calI({T}),\,j\ne i}a_{ij}^{n,{T}}(\sfU_j^n-\sfU_i^n).
    \end{aligned}
  \]
  Hence \eqref{eq:LO} reads
  \[
    \begin{aligned}
      \sfU_i^{\mathrm L,n+1}
      &=\sfU_i^n+\frac{\dt}{m_i^{\polP_1}}\sum_{{T}\subset S_i^{\polP_1}}\sum_{j\in\calI({T}),\,j\ne i}\big(d_{ij}^{{\rm L},n,{T}}-a_{ij}^{n,{T}}\big)(\sfU_j^n-\sfU_i^n)\\
      &=\Big(1-\sum_{j\ne i}\gamma_{ij}^n\Big)\sfU_i^n+\sum_{j\ne i}\gamma_{ij}^n\sfU_j^n .
    \end{aligned}
  \]
  Since $|a_{ij}^{n,{T}}|\le d_{ij}^{{\rm L},n,{T}}$, all $\gamma_{ij}^n\ge0$, and by \eqref{eq:cfl} the coefficient of $\sfU_i^n$ is nonnegative. Thus $\sfU_i^{\mathrm L,n+1}$ is a convex combination of the values $\sfU_j^n$, $j\in\calI(S_i^{\polP_1})$, which gives \eqref{eq:UL_max}.
\end{proof}

\begin{remark}[CFL condition]\label{rem:cfl}
  We define the minimal fine mesh size by
  \[
    h_{\min} := \min_{T\in \calT_h^{\polP_1}}\big( \max_{i \in \calI({T})} \|\GRAD \varphi_i^{\polP_1}\|_{L^\infty({T})} \big)^{-1}.
  \]
  Similarly, we define maximum wave speed by
  $
  \beta:= \sup_{w\in[u_{\min}, u_{\max}]} \|\bef'(w)\|.
  $
  Since $\varphi_i^{\polP_1}\ge0$ and $\|\GRAD\varphi_j^{\polP_1}\|_{L^\infty({T})}\le {h_{\min}}^{-1}$, we have $\|\bc_{ij}^{T}\|\le m_{i,{T}}^{\polP_1}/{h_{\min}}$. Hence
  $|a_{ij}^{n,{T}}|\le \beta\|\bc_{ij}^{T}\|\le \frac{\beta}{{h_{\min}}}m_{i,{T}}^{\polP_1}$ and $d_{ij}^{{\rm L},n,{T}}\le\frac{\beta}{{h_{\min}}}m_{i,{T}}^{\polP_1}$.
  Each ${T}\subset S_i^{\polP_1}$ has $d$ vertices other than $i$, so
  \[
    \sum_{j\in\calI(S_i^{\polP_1}),\,j\ne i}\gamma_{ij}^n
    \le \frac{\dt}{m_i^{\polP_1}}\sum_{{T}\subset S_i^{\polP_1}} d\,\frac{2\beta}{{h_{\min}}}\,m_{i,{T}}^{\polP_1}
    = 2d\,\frac{\beta}{{h_{\min}}}\,\dt .
  \]
  Therefore \eqref{eq:cfl} holds if
  \[
    \dt\le\frac{\nu}{2d}\,\frac{{h_{\min}}}{\beta},
  \]
  which is a sufficient condition of the usual form $\dt\le C\,{h_{\min}}/\beta$. In the computations we use a sharper form of this bound discussed below.

  Let $\beta_{T}\ge\|\bef'(w)\|$ for all $w\in[\min u_0,\max u_0]$ (for a linear flux $\bef(u)=\boldsymbol\beta(\bx)u$, $\beta_{T}:=\max_{\bx\in {T}}\|\boldsymbol\beta(\bx)\|$), where ${T}$ is a fine $\polP_1$ triangle. Since the nodal values stay in $[\min u_0,\max u_0]$, $d_{ij}^{{\rm L},n,{T}}-a_{ij}^{n,{T}}\le2d_{ij}^{{\rm L},n,{T}}\le2\beta_{T}\|\bc_{ij}^{T}\|$ for all $n$, and \eqref{eq:cfl} holds for a forward-Euler step of length $\tau$ if
  \begin{equation}\label{eq:cfl_dt}
    \tau\le \nu \,\min_{i\in\calV}\frac{m_i^{\polP_1}}{\zeta_i},\qquad
    \zeta_i:=\sum_{{T}\subset S_i^{\polP_1}}\ \sum_{j\in\calI({T}),\,j\ne i}2\beta_{T}\|\bc_{ij}^{T}\|.
  \end{equation}
  As in \citep[eq.~(6.1)]{Guermond_Nazarov_Popov_2024}, we set
  \begin{equation}\label{eq:dt_nu}
    \dt=s\,{\rm CFL}\,\min_{i\in\calV}\frac{m_i^{\polP_1}}{\zeta_i},
  \end{equation}
  where $s=5$ is the number of stages of the SSP-RK(5,4) method of \citep{Kraaijevanger_1991} and ${\rm CFL}$ is the CFL number reported in Section~\ref{Sec:numerics}; in the $\polP_k$ form $\dt$ is further divided by $r_{\max}$ (Lemma~\ref{lem:LO_pk}). Every stage of this method is a convex combination of forward-Euler steps of length at most $0.6631\,\dt$, so \eqref{eq:cfl} holds at every stage and every time level with $\nu =0.6631\,s\,{\rm CFL}\approx3.32\,{\rm CFL}$, that is, for ${\rm CFL}\le0.30$.
  
\end{remark}

\begin{lemma}[Low-order scheme in the $\polP_k$ masses]\label{lem:LO_pk}
  Assume $m_i^{\polP_k}>0$ for all $i\in\calV$ and set $r_{\max}:=\max_{i}m_i^{\polP_1}/m_i^{\polP_k}$. Define
  \begin{equation}\label{eq:LOk}
    \widehat\sfU_i^{{\rm L},n+1}:=\sfU_i^n+\frac{m_i^{\polP_1}}{m_i^{\polP_k}}\big(\sfU_i^{{\rm L},n+1}-\sfU_i^n\big),
  \end{equation}
  If the time step satisfies \eqref{eq:cfl} with $\nu\le 1/r_{\max}$, then $\sfU_i^{n,\min}\le\widehat\sfU_i^{{\rm L},n+1}\le\sfU_i^{n,\max}$ and $\Mk(\widehat\sfU^{{\rm L},n+1})=\Mk(\sfU^n)$.
\end{lemma}
\begin{proof}
  By the proof of Theorem~\ref{Thm:Min_Max_first_order},
  $\widehat\sfU_i^{{\rm L},n+1}=\big(1-r_i\sum_{j\ne i}\gamma_{ij}^n\big)\sfU_i^n+\sum_{j\ne i}r_i\gamma_{ij}^n\sfU_j^n$ with $r_i=m_i^{\polP_1}/m_i^{\polP_k}\le r_{\max}$, which is a convex combination when $r_{\max}\sum_{j\ne i}\gamma_{ij}^n\le1$. Conservation follows from \eqref{eq:LOk} and the conservation of \eqref{eq:LO}.
\end{proof}
By Table~\ref{tab:lumped}, the time step is reduced by the factor $1.48$ for $\polP_3$ and $3.23$ for $\polP_5$; no reduction is needed for $\polP_1$.
\begin{remark}[Relation to other low-order schemes]\label{rem:LO_literature}
  The low-order scheme \eqref{eq:LO} is not new: it is the graph-viscosity method of \citep{Guermond_Popov_2016} applied on the fine-$\polP_1$ submesh $\calT_h^{\polP_1}$ whose vertices are the $\polP_k$ Lagrange nodes. The advective term is the group form $\sum_{{T}\in\calT_h^{\polP_1}}\sum_j\bef(\sfU_j)\SCAL\bc_{ij}^{T}$ of \citep{Guermond_Popov_2016}; the only difference is that the viscosity is computed cell by cell, from the range of $u_h$ on the fine cell as in \citep{Guermond_Nazarov_2014}, and then summed over the cells sharing an edge. Since $|\sum_{{T}\in\calT_h^{\polP_1}} a_{ij}^{T}|\le\sum_{{T}\in\calT_h^{\polP_1}}|a_{ij}^{T}|$, this viscosity is not smaller than the one of \citep{Guermond_Popov_2016} computed with the assembled vectors $\bc_{ij}=\sum_{{T}\in\calT_h^{\polP_1}}\bc_{ij}^{T}$. For linear transport the discrete upwinding of \citep{Kuzmin_Moller_2005} adds less viscosity, $\max(0,-k_{ij},-k_{ji})$ instead of an absolute value; both lead to the same convex-combination argument. Low-order schemes on a $\polP_1$ or $\polQ_1$ subcell mesh, whose stencil does not grow with the polynomial degree, were used with Bernstein elements in \citep{Lohmann_etal_2017,Kuzmin_Quezada_2020,Hajduk_etal_2020a} and with discontinuous Galerkin spectral elements in \citep[Sec.~3]{Panzer_2021}, where for linear transport the viscosity is, as here, the larger of the two edge coefficients evaluated with the nodal velocities \citep[eq.~(43)]{Panzer_2021}; the low-order states of \citep[Sec.~6]{Guermond_2026} are computed with the method of \citep{Guermond_Popov_2016} on the same hierarchical $\polP_1$ subdivision as here. What is specific to continuous Lagrange $\polP_k$ elements is that the lumped mass $m_i^{\polP_1}$ of the low-order scheme differs from the lumped mass $m_i^{\polP_k}$ of the high-order scheme (Section~\ref{sec:masses}); in the discontinuous Galerkin setting of \citep{Panzer_2021} both schemes use the same diagonal mass matrix. In \citep{Guermond_Nazarov_Popov_2024} this mismatch is removed by modifying the fluxes of a nearest-neighbour low-order method so that it carries the same generalized mass as the high-order method, in the spirit of \citep{Abgrall_etal_2017}, and \citep{Guermond_2026} avoids it by abandoning the flux-correction paradigm altogether. Here the low-order scheme is kept unchanged and the mass difference is absorbed into the antidiffusive fluxes by the graph-Poisson reconstruction of Section~\ref{Sec:limiter}.
\end{remark}

\section{High-order scheme}
\label{Sec:high}

Since the limiting procedure discussed below does not require the low-order and high-order schemes to use the same stabilization terms, one can use any high-order finite element discretization to compute the high-order nodal values $\sfU_i^{\textrm{H},n+1}$. Examples of popular stabilization techniques for finite element approximations include the entropy viscosity method \citep{Guermond_pasquetti_popov_JCP_2011,Nazarov_Larcher_2017}, the continuous interior penalty (CIP) method \citep{Douglas_Dupont_1976,Burman_Hansbo_2004}, and residual-based artificial viscosity, or residual viscosity (RV), methods \citep{Nazarov_2013,Nazarov_Hoffman_2013}. In this work, we use the RV method.

We already constructed an efficient local fine-$\polP_1$ infrastructure when defining the first-order scheme in the previous section. Therefore, we construct the high-order scheme by defining a higher-order local edge-viscosity coefficient analogous to \eqref{eq:dij}, and subsequently a higher-order viscous bilinear form analogous to \eqref{eq:b_h}. Let $\sfR_i^n$ denote the nodal residual of the high-order approximation at time $t_n$, i.e., the nodal values of the $L^2(\Omega)$ projection onto $\calX_h$ of $|\p_t u_h+\DIV\bef(u_h)|$, where $\p_t u_h$ is approximated by the fourth-order backward difference formula. In the present implementation, the element-local residual magnitude associated with node $i\in\calI({T})$ is
\[
  \sfR_i^{n,{T}}
  :=
  |\sfR_i^n|
  \int_{T} \varphi_i^{\polP_1}\,\ud\bx
  =
  |\sfR_i^n|m_{i,{T}}^{\polP_1}.
\]

For every edge $(i,j)$ contained in an element ${T}\in\calT_h^{\polP_1}$, we define
\begin{equation}
  \label{eq:dij_K}
  d_{ij}^{{\rm R},n,{T}}
  =
  C_{\rm R}
  \max\left\{
    \frac{\sfR_i^{n,{T}}}{n(u_h^n)},
    \frac{\sfR_j^{n,{T}}}{n(u_h^n)}
  \right\},
\end{equation}
where $C_{\rm R}\ge 0$ is a user-defined parameter, and $n(u_h^n)>0$ is a normalization function chosen so that \eqref{eq:dij_K} has the same units as \eqref{eq:dij}. For example, one may use
$
n(u_h^n)
=
\|u_h^n-\overline{u_h^n}\|_{L^\infty}
+
\epsilon\|u_h^n\|_{L^\infty},
$
where $\epsilon>0$ is a small constant and
$
\overline{u_h^n}
=
\frac{1}{|\Omega|}
\int_\Omega u_h^n\,\ud\bx .
$

The symmetric local residual-viscosity coefficient is then defined by
\begin{equation}
  \label{eq:dij_H}
  d_{ij}^{{\rm H},n,{T}}
  =
  \min\left\{
    d_{ij}^{{\rm L},n,{T}},
    d_{ij}^{{\rm R},n,{T}}
  \right\},
  \qquad i\ne j.
\end{equation}

For every $i\in\calV$, we define the element-local residual-based graph-viscosity contribution by
\[
  b_{T}^{{\rm H}}(u_h^n,\varphi_i^{\polP_1})
  :=
  \sum_{\substack{j\in\calI({T}), j\ne i}}
  d_{ij}^{{\rm H},n,{T}}
  \big(\sfU_i^n-\sfU_j^n\big).
\]
This definition is locally conservative.

The global artificial viscosity operator is obtained by summing all fine-cell contributions over the fine-$\polP_1$ cells contained in the support of $\varphi_i^{\polP_1}$:
\begin{equation}
  \label{eq:b_h_H}
  b_h^{{\rm H}}(u_h^n,\varphi_i^{\polP_1})
  =
  \sum_{{T}\subset S_i^{\polP_1}}
  b_{T}^{{\rm H}}(u_h^n,\varphi_i^{\polP_1})
  =
  \sum_{{T}\subset S_i^{\polP_1}}
  \sum_{\substack{j\in\calI({T}), j\ne i}}
  d_{ij}^{{\rm H},n,{T}}
  \big(\sfU_i^n-\sfU_j^n\big).
\end{equation}

Since the residual may become very small in smooth regions, residual viscosity alone may add almost no stabilization there. In addition, residual-viscosity methods may perform suboptimally for even-order polynomial spaces; see, for instance, \citep[Sec.~5]{Dao_Nazarov_2022}. As shown in \citep[Sec.~4.2]{Ern_Guermond_2013}, adding an additional linear stabilization term, such as CIP stabilization, can restore optimal convergence for even-order polynomial spaces.

In this work, we therefore employ the CIP stabilization of \citep{Douglas_Dupont_1976,Burman_Hansbo_2004} in addition to the RV term \eqref{eq:b_h_H}:
\[
  s_h(u_h^n,\varphi_i)
  =
  \gamma
  \sum_{F\in\mathcal F_h^{\mathrm{int}}}
  h_F^2
  \beta_F
  \int_F
  [\![\GRAD u_h^n]\!]
  [\![\GRAD \varphi_i]\!]
  \,\ud s,
\]
where
$
[\![ \nabla w ]\!]_F
:=
\nabla w^+\SCAL \bn^+
+
\nabla w^-\SCAL \bn^-
$
denotes the jump of the normal derivative across an interior facet $F$ of the coarse $\polP_k$ mesh. The facet wave speed is defined by
\[
  \beta_F
  :=
  \frac12
  \left(
    \big|\bef'(u_h^n)^+\SCAL \bn^+\big|
    +
    \big|\bef'(u_h^n)^-\SCAL \bn^-\big|
  \right).
\]
The facet size is computed as
$
h_F := \frac12(h_{K^+}+h_{K^-}),
$
where $K^+\in\calT_h$ and $K^-\in\calT_h$ are the two elements sharing the interior facet $F$, and $h_K=\operatorname{diam}(K)$. The penalty parameter is set to
$
\gamma
=
\gamma_{\mathrm{cip}}
\frac{d^2}{10(1+k)^4}
$
with $\gamma_{\mathrm{cip}}=1$ in all simulations presented in this paper.

Finally, we compute the high-order nodal values $\sfU_i^{\textrm{H},n+1}$ by solving
\begin{equation}
  \label{eq:HO}
  \begin{aligned}
    \sum_{j\in\calI(S_i)}
    m_{ij}^{\polP_k}
    \frac{\sfU_j^{\textrm{H},n+1}-\sfU_j^n}{\dt}
    +
    \int_{S_i}
    \DIV \bef(u_h^n)\varphi_i\,\ud\bx
    +
    b_h^\textrm{H}(u_h^n,\varphi_i^{\polP_1})
    +
    s_h(u_h^n,\varphi_i)
    =
    0,
  \end{aligned}
\end{equation}
for every $i\in\calV$.

\begin{proposition}
  The high-order scheme \eqref{eq:HO} is conservative with respect to the high-order mass matrix, that is,
  \[
    \sum_{i\in\calV} m_i\sfU_i^{\mathrm H,n+1}
    =
    \sum_{i\in\calV} m_i\sfU_i^{n}.
  \]
\end{proposition}

\begin{proof}
  Summing \eqref{eq:HO} over all $i\in\calV$, using the partition-of-unity property of the basis functions, and using the conservation properties of
  $
  b_h^\textrm{H}(u_h^n,\varphi_i^{\polP_1})
  $
  and
  $
  s_h(u_h^n,\varphi_i),
  $
  proves the result.
\end{proof}

\section{Graph-Poisson limiter}
\label{Sec:limiter}
The low-order update \eqref{eq:LO} conserves the fine-$\polP_1$ mass $\Mone$, whereas the high-order update \eqref{eq:HO} conserves the $\polP_k$ integral $\Mk$. Instead of comparing the two updates in one of these measures, we compare each of them in its own conserved measure and define
\begin{equation}\label{eq:A_new}
  \sfA_i^{n+1}:=m_i^{\polP_k}\big(\sfU_i^{{\rm H},n+1}-\sfU_i^n\big)-m_i^{\polP_1}\big(\sfU_i^{{\rm L},n+1}-\sfU_i^n\big),\qquad i\in\calV .
\end{equation}

\begin{lemma}[Compatibility]\label{lem:zero_sum}
  Assume periodic boundary conditions, or $\bef(u_h^n)\SCAL\bn=0$ on $\p\Omega$. Then
  \[
    \sum_{i\in\calV}\sfA_i^{n+1}=\Mk(\sfU^{{\rm H},n+1})-\Mk(\sfU^n)-\big(\Mone(\sfU^{{\rm L},n+1})-\Mone(\sfU^n)\big)=0.
  \]
\end{lemma}
\begin{proof}
  Sum \eqref{eq:LO} and \eqref{eq:HO} over $i\in\calV$. Since $\sum_{i\in\calV}\varphi_i^{\polP_1}=\sum_{i\in\calV}\varphi_i=1$, the viscous terms $b_h^{\rm L}$, $b_h^{\rm H}$ and the stabilization $s_h$ sum to zero, and the flux terms give
  \[
    \Mone(\sfU^{{\rm L},n+1})-\Mone(\sfU^n)=-\dt\int_{\p\Omega}I_h^{\polP_1}\bef(u_h^n)\SCAL\bn\,\ud s,
  \]
  and
  \[      
    \Mk(\sfU^{{\rm H},n+1})-\Mk(\sfU^n)=-\dt\int_{\p\Omega}\bef(u_h^n)\SCAL\bn\,\ud s .
  \]
  Both boundary integrals vanish for periodic boundary conditions. If $\bef(u_h^n)\SCAL\bn=0$ on $\p\Omega$, then $\bef(\sfU_j^n)\SCAL\bn=0$ at every boundary node $j$ and on every boundary facet containing it, so $I_h^{\polP_1}\bef(u_h^n)\SCAL\bn=0$ on $\p\Omega$ as well, since $\bn$ is constant on each facet. Hence both mass changes vanish and $\sum_{i\in\calV}\sfA_i^{n+1}=0$.
\end{proof}

\subsection{Graph-Poisson flux reconstruction}\label{sec:graph_poisson}

The decomposition of $\sfA$ into edge fluxes is not unique: adding the same amount to the three fluxes around a fine triangle, oriented consistently, leaves every nodal sum unchanged. We select the fluxes of minimal $\ell^2$-norm.

\begin{lemma}[{Classical minimum-norm graph property; see \citep[Sec.~IX.1, Thm.~2]{Bollobas_1998}}]\label{lemma:F}
  Let the fine-$\polP_1$ graph with vertex set $\calV$ and edge set $\calE$ be connected,
  and let $\sfA\in\polR^{|\calV|}$ satisfy $\sum_{i\in\calV}\sfA_i=0$. Among all
  antisymmetric edge fluxes, $\sfF_{ij}=-\sfF_{ji}$, satisfying
  $\sum_{j\in\calI(S_i^{\polP_1})}\sfF_{ij}=\sfA_i$ for all $i\in\calV$, the minimizer of
  $\sum_{\{i,j\}\in\calE}\sfF_{ij}^2$ is unique and given by
  $\sfF_{ij}=\sfQ_i-\sfQ_j$, where $\sfQ$ solves
  \begin{equation}\label{eq:LQ}
    L\sfQ=\sfA
  \end{equation}
  where $L$ is the graph Laplacian of the fine-$\polP_1$ graph; $\sfQ$ is unique up to a constant.
\end{lemma}

\begin{proof}
  Orient each edge and let $B\in\polR^{|\calV|\times|\calE|}$ be the incidence matrix,
  $B_{ne}=1$ if edge $e$ leaves node $n$, $B_{ne}=-1$ if it enters $n$, and $B_{ne}=0$ otherwise.
  Then the constraint reads $B\sfF=\sfA$, $(B^{\!\top}\sfQ)_{(i,j)}=\sfQ_i-\sfQ_j$, and $L=BB^{\!\top}$.
  On a connected graph $\ker L=\ker B^{\!\top}=\operatorname{span}\{\mathbf 1\}$, see, \eg \citep{Chung_1997}, so $L\sfQ=\sfA$ is
  solvable since $\sum_{i\in\calV}\sfA_i=0$, and $\sfF^\ast:=B^{\!\top}\sfQ$ is feasible. Any other feasible
  flux is $\sfF=\sfF^\ast+\sfC$ with $B\sfC=0$, and
  $(\sfF^\ast,\sfC)=(\sfQ,B\sfC)=0$; hence $|\sfF|^2=|\sfF^\ast|^2+|\sfC|^2\ge|\sfF^\ast|^2$,
  with equality only for $\sfC=0$.
\end{proof}

The same construction, $\sfF=B^{\!\top}\sfQ$ with $L\sfQ$ equal to a zero-sum right-hand side, is used on the triangulated degrees of freedom of a single element in \citep[Lemma~4.2, eq.~(27)]{Abgrall_2018} to write residual distribution schemes in flux form; here it is applied on the global fine-$\polP_1$ graph. The minimal fluxes are unique and independent of the enumeration of the nodes, and they are the smallest antidiffusive correction, in the $\ell^2$ sense, that reproduces $\sfA$; the price is that they are global.

Next, we construct a relation between the high- and low-order solutions. We consider the following two cases:
\begin{enumerate}
\item $\polP_k$ form, all $m_i^{\polP_k} > 0$;
\item Fine-$\polP_1$ form, any polynomial degree.
\end{enumerate}

\noindent{\bf $\polP_k$ form.}

We write \eqref{eq:A_new} as
\begin{equation}\label{eq:ho1}
  m_i^{\polP_k}\big(\sfU_i^{{\rm H},n+1} 
  -
  \sfU_i^n\big)
  =
  m_i^{\polP_1}\big(\sfU_i^{{\rm L},n+1}-\sfU_i^n\big) + \sfA_i^{n+1},\quad i\in\calV,
\end{equation}
and \eqref{eq:LOk} as
\begin{equation}\label{eq:lo1}
  m_i^{\polP_k}\big(\widehat\sfU_i^{{\rm L},n+1}-\sfU_i^n\big)=m_i^{\polP_1}\big(\sfU_i^{{\rm L},n+1}-\sfU_i^n\big),\quad i\in\calV.
\end{equation}

Subtracting \eqref{eq:lo1} from \eqref{eq:ho1}, dividing by $m_i^{\polP_k}$ and using Lemma~\ref{lemma:F} gives
\begin{equation}\label{eq:HO_LO_pk}
  \sfU_i^{{\rm H},n+1} 
  =
  \widehat\sfU_i^{{\rm L},n+1}
  +
  \frac{1}{m_i^{\polP_k}} \sum_{j\in\calI(S_i^{\polP_1})}\sfF_{ij}^{n+1}.
\end{equation}

\noindent{\bf Fine-$\polP_1$ form.}

Now, let us consider more general case where some entries of $m_i^{\polP_k}$ can be zero or even negative. We add $m_i^{\polP_1} \sfU_i^{{\rm H}, n+1} - m_i^{\polP_1} \sfU_i^{{\rm H}, n+1}$ to the right hand side of \eqref{eq:ho1} and get 
\begin{align*}
  m_i^{\polP_k}\big(\sfU_i^{{\rm H},n+1} 
  -
  \sfU_i^n\big)
  &
    =
    \big(m_i^{\polP_1}\sfU_i^{{\rm L},n+1}- m_i^{\polP_1}\sfU_i^n \big) + \sfA_i^{n+1} + 
    \big(m_i^{\polP_1} \sfU_i^{{\rm H}, n+1} - m_i^{\polP_1} \sfU_i^{{\rm H}, n+1} \big) \\
  &
    =
    \sfA_i^{n+1} - 
    m_i^{\polP_1} \big(\sfU_i^{{\rm H}, n+1} - \sfU_i^{{\rm L},n+1} \big) +
    m_i^{\polP_1} \big(\sfU_i^{{\rm H}, n+1} - {\sfU_i^n} \big).
\end{align*}
This can be written as 
\[
  m_i^{\polP_1}\big(\sfU_i^{{\rm H},n+1}-\sfU_i^{{\rm L},n+1}\big)=\sfA_i^{n+1}+\sfs_i^{n+1},
\]
where
\[
  \sfs_i^{n+1}:=\big(m_i^{\polP_1}-m_i^{\polP_k}\big)\big(\sfU_i^{{\rm H},n+1}-\sfU_i^n\big),
\]
is the {\em lumping source} and is a nodal entity not of flux form; and its sum
$\sum_{i\in\calV}\sfs_i^{n+1}=\sum_{K\in\calT_h}E_K(\sfU^{{\rm H},n+1}-\sfU^n)$ is in general nonzero.
It is limited separately, and the resulting $\polP_k$-mass defect is removed by the
element-local exchange of Section~\ref{Sec:local_exchange}.

Consequently, for this case we obtain the following high- and low order relation:
\begin{equation}\label{eq:HO_LO_p1}
  \sfU_i^{{\rm H},n+1} 
  =
  \sfU_i^{{\rm L},n+1}
  +
  \frac{1}{m_i^{\polP_1}} \Big(\sum_{j\in\calI(S_i^{\polP_1})}\sfF_{ij}^{n+1} + \sfs_i^{n+1} \Big) .
\end{equation}

\begin{remark}[{Solution of the graph-Laplacian problem}]\label{rem:graph_laplacian_solve}
  The matrix $L$ is assembled once on the fine-$\polP_1$ graph: every edge $\{i,j\}\in\calE$ contributes the $2\times2$ block with rows $(1,-1)$ and $(-1,1)$, once, so that $L_{ii}$ is the number of neighbours of $i$ and $L_{ij}=-1$ for $\{i,j\}\in\calE$, i.e., $(L\sfQ)_i=\sum_{j\in\calI(S_i^{\polP_1})}(\sfQ_i-\sfQ_j)$. It is symmetric positive semidefinite with $\ker L=\operatorname{span}\{\mathbf 1\}$ on a connected graph; the constant null space is attached to the matrix. Since $L$ depends only on the fine-$\polP_1$ connectivity, its sparsity does not grow with $k$. In our computations $L$ is factorized once by the sparse direct solver MUMPS and the factors are reused at every Runge--Kutta stage, so that each solve of \eqref{eq:LQ} is one forward and one backward triangular substitution. Alternatively, algebraic multigrid preconditioned conjugate gradients, warm-started with the potential of the previous stage, solves \eqref{eq:LQ} at $\mathcal O(|\calV|)$ cost. The fluxes $\sfQ_i-\sfQ_j$ are antisymmetric for any $\sfQ$, so an inexact solve does not affect conservation, only the consistency with the high-order update. If $\sum_{i\in\calV}\sfA_i\neq0$ (Lemma~\ref{lem:zero_sum}), the right-hand side is made compatible by the projection $\sfA_i\leftarrow\sfA_i-w_i\sum_{j\in\calV}\sfA_j/\sum_{j\in\calV} w_j$, with $w_i=m_i^{\polP_k}$ in the $\polP_k$ form and $w_i=m_i^{\polP_1}$ in the fine-$\polP_1$ form, which amounts to a uniform shift of one candidate state; its size $|\sum_{i\in\calV}\sfA_i|$ is monitored.
\end{remark}

\begin{remark}[{General boundary conditions}]\label{rem:boundary_flux}
  {Without the assumption of Lemma~\ref{lem:zero_sum}, the proof gives
    $\sum_{i\in\calV}\sfA_i^{n+1}=-\dt\int_{\p\Omega}\big(\bef(u_h^n)-I_h^{\polP_1}\bef(u_h^n)\big)\SCAL\bn\,\ud s$:
    the high-order scheme uses $\bef(u_h^n)$ on $\p\Omega$, the low-order scheme its
    fine-$\polP_1$ interpolant, and the difference is an interpolation error of the boundary flux,
    of order $h^2$ for smooth solutions. Since \eqref{eq:LQ} requires a zero sum, $\sfA^{n+1}$ is made
    compatible by the projection of Remark~\ref{rem:graph_laplacian_solve}. In both forms this
    replaces the high-order candidate by the uniform shift $\sfU^{{\rm H},n+1}-c$, with
    $c=\sum_{i\in\calV}\sfA_i^{n+1}/\Mk(\mathbf 1)$ in the $\polP_k$ form ($w_i=m_i^{\polP_k}$) and
    $c=\sum_{i\in\calV}\sfA_i^{n+1}/\Mone(\mathbf 1)$ in the fine-$\polP_1$ form ($w_i=m_i^{\polP_1}$); the
    base state $\overline\sfU^{n+1}$, and hence the bounds, are not affected. The shift is $\mathcal O(\dt\,h^2)$ per step but uniform in $\Omega$, so for a smooth solution with a nonzero flux through $\p\Omega$ it may limit the accuracy of high-degree approximations to second order; in our tests it is negligible, since the solutions either vanish near $\p\Omega$ or contain shocks. A local alternative is to subtract from $\sfA^{n+1}$ the nodal boundary mismatch $b_i:=-\dt\int_{\p\Omega}\big(\bef(u_h^n)-I_h^{\polP_1}\bef(u_h^n)\big)\SCAL\bn\,\varphi_i\,\ud s$, which is supported on the boundary nodes and satisfies $\sum_{i\in\calV}b_i=\sum_{i\in\calV}\sfA_i^{n+1}$. If no Dirichlet values
    are imposed, the limited solution then satisfies
    \[
      \Mk(\sfU^{n+1})-\Mk(\sfU^n)=-\dt\int_{\p\Omega}I_h^{\polP_1}\bef(u_h^n)\SCAL\bn\,\ud s,
    \]
    i.e., it is conservative with the low-order boundary flux: in the $\polP_k$ form by
    Proposition~\ref{thm:pk_form}, and in the fine-$\polP_1$ form after the element-local exchange,
    since the proof of Proposition~\ref{prop:p1_defect} then gives
    $\Mk(\sfV^{n+1})-\Mk(\sfU^n)=-\dt\int_{\p\Omega}I_h^{\polP_1}\bef(u_h^n)\SCAL\bn\,\ud s+\sum_{K\in\calT_h}D_K$.
    At nodes with strongly imposed Dirichlet values the datum $g_i$ is written into both candidates
    before limiting and into the limited solution after it; such a node contributes
    $\sfA_i^{n+1}=(m_i^{\polP_k}-m_i^{\polP_1})(g_i-\sfU_i^n)$, which vanishes for time-independent
    data. There the base state equals $g_i$, which need not lie in the local bounds of $\sfU^n$; the
    bounds of Remark~\ref{remark:relaxed_bounds} are therefore enlarged to
    $\min(\widetilde\sfU_i^{n,\min},\overline\sfU_i^{n+1})$ and $\max(\widetilde\sfU_i^{n,\max},\overline\sfU_i^{n+1})$.
    This changes nothing at the other nodes, where $\overline\sfU_i^{n+1}$ already lies in the bounds,
    keeps the hypothesis of Theorem~\ref{t:ho_max_p}, and preserves the invariant interval when the
    boundary data lie in $[\min u_0,\max u_0]$ (cf.\ \citep[eq.~(15)]{Anderson_etal_2017}).}
\end{remark}

\subsection{Convex limiting}\label{sec:convex_limiting}
\label{Sec:convex_limiting}

Once the high- and low-order solutions have been related by a conservative graph flux, one can apply algebraic flux correction, flux-corrected transport, or convex limiting techniques to obtain an invariant-domain preserving high-order method. We use the convex limiting framework of \citep{Guermond_Nazarov_Popov_Tomas_2018}.

Let us combine the relations \eqref{eq:HO_LO_pk} and \eqref{eq:HO_LO_p1} into the following form
\begin{equation}\label{eq:HO_LO}
  \sfU_i^{{\rm H},n+1} 
  =
  \overline\sfU_i^{n+1}
  +
  \frac{1}{\mu_i} \Big(\sum_{j\in\calI(S_i^{\polP_1})}\sfF_{ij}^{n+1} + \sigma_i^{n+1} \Big),
\end{equation}
where
\begin{itemize}
\item for $\polP_k$ form (all $m_i^{\polP_k}>0$): $\mu_i=m_i^{\polP_k}$, $\overline\sfU_i^{n+1}=\widehat\sfU_i^{{\rm L},n+1}$, $\sigma_i^{n+1}=0$;
\item for fine-$\polP_1$ form (any degree): $\mu_i=m_i^{\polP_1}$, $\overline\sfU_i^{n+1}=\sfU_i^{{\rm L},n+1}$, $\sigma_i^{n+1}=\sfs_i^{n+1}$.
\end{itemize}

Note that in both cases $\overline\sfU_i^{n+1}\in[\sfU_i^{n,\min},\sfU_i^{n,\max}]$ by Theorem~\ref{Thm:Min_Max_first_order} and Lemma~\ref{lem:LO_pk}.

Let $\big(\ell_{ij}\big)_{i,j \in \calV}$ be a symmetric limiter matrix satisfying
\[
  0\le \ell_{ij}=\ell_{ji}\le 1.
\]
and let $0\le{\eta}_i\le1$ be a nodal source limiter. We define the limited state as
\begin{equation}
  \label{eq:lijAij}
  \sfV_i^{n+1}
  =
  \overline\sfU_i^{n+1}
  +
  \frac{1}{\mu_i}
  \Big(\sum_{j\in\calI(S_i^{\polP_1})}
  \ell_{ij}\sfF_{ij}^{n+1}+{\eta}_i\sigma_i^{n+1}\Big).
\end{equation}

If $\ell_{ij}\equiv0$ and ${\eta}_i\equiv0$, then $\sfV^{n+1}=\overline\sfU^{n+1}$ is the low-order state; if $\ell_{ij}\equiv1$ and ${\eta}_i\equiv1$, then $\sfV^{n+1}=\sfU^{{\rm H},n+1}$ exactly by \eqref{eq:HO_LO}.
Let $\sfU_i^{n,\min}$ and $\sfU_i^{n,\max}$ denote the local bounds defined in \eqref{eq:Umax-Umin} (in the computations they are replaced by the relaxed bounds of Remark~\ref{remark:relaxed_bounds}). We construct the limiter so that
\[
  \sfU_i^{n,\min}
  \le
  \sfV_i^{n+1}
  \le
  \sfU_i^{n,\max}.
\]

Next, let us denote the number of nodes in the support of $\varphi_i^{\polP_1}$ by $\textrm{card}(S_i^{\polP_1})$ and define:
$
\lambda_j(i)
:=
\frac{1}{\operatorname{card}(S_i^{\polP_1})-1+\iota}
$
(with $\iota=1$ in the fine-$\polP_1$ form, where the source uses one more slot $\lambda_\sigma(i)=\lambda_j(i)$, and $\iota=0$ otherwise)
for every
$
j\in\calI(S_i^{\polP_1})\setminus\{i\}.
$
For each edge contribution, we define
$
\sfP_{ij}^{n+1}
:=
\frac{\sfF_{ij}^{n+1}}
{\lambda_j(i)\mu_i}
$
and $\sfP_{i}^{\sigma,n+1}:=\frac{\sigma_i^{n+1}}{(\lambda_\sigma(i)\mu_i)}$.
Then the limited update can be written as the convex combination
\begin{equation}\label{eq:convcom}
  \sfV_i^{n+1}
  =
  \sum_{j\in\calI(S_i^{\polP_1})\setminus\{i\}}
  \lambda_j(i)
  \big(
  \overline\sfU_i^{n+1}  +
  \ell_{ij}\sfP_{ij}^{n+1}
  \big)
  +\iota\,\lambda_\sigma(i)\big(\overline\sfU_i^{n+1}+{\eta}_i\sfP_i^{\sigma,n+1}\big).
\end{equation}
For each directed edge $i\to j$, set
\[
  \theta_{ij}
  =
  \begin{cases}
    \displaystyle
    \min\left\{
    1,
    \frac{\overline\sfU_i^{n+1}-\sfU_i^{n,\min}}
    {|\sfP_{ij}^{n+1}|}
    \right\},
    &
      \sfP_{ij}^{n+1}<0,
    \\[4mm]
    \displaystyle
    \min\left\{
    1,
    \frac{\sfU_i^{n,\max}-\overline\sfU_i^{n+1}}
    {|\sfP_{ij}^{n+1}|}
    \right\},
    &
      \sfP_{ij}^{n+1}>0,
    \\[4mm]
    1,
    &
      \sfP_{ij}^{n+1}=0.
  \end{cases}
\]
The symmetric edge limiter is then
\begin{equation}
  \label{eq:lij}
  \ell_{ij}:=\min(\theta_{ij},\theta_{ji}).
\end{equation}
In the implementation a tiny constant $\epsilon=10^{-14}\big(\max_{j\in\calV} \sfU_j^n - \min_{j\in\calV} \sfU_j^n\big)$ is added to $|\sfP_{ij}^{n+1}|$ to avoid round-off in the quotient; since it only decreases $\theta_{ij}$, the bounds are unaffected.
The source limiter ${\eta}_i$ is defined by the same formula with $\sfP_{ij}^{n+1}$ replaced by $\sfP_i^{\sigma, n+1}$; it needs no symmetrization.

\paragraph{Iterated limiting} 
The limiter can be applied repeatedly, as in iterative flux-corrected transport \citep{Schar_Smolarkiewicz_1996}, \citep[Sec.~6.3]{Kuzmin_Moller_2005} and iterated convex limiting \citep[Sec.~4.3, Alg.~1]{Guermond_Nazarov_Popov_Tomas_2018}; in the high-order setting of \citep[Sec.~6.1]{Guermond_Nazarov_Popov_2024} the limiter is applied four times per stage. In pass $p\ge2$ the fluxes rejected so far, $\big(1-\ell_{ij}^{[p-1]}\big)\sfF_{ij}^{n+1}$, are limited by \eqref{eq:lij} around the state $\sfV^{[p-1]}$ of the previous pass, which takes the place of $\overline\sfU_i^{n+1}$, with the same bounds (in \citep[Sec.~6.3]{Kuzmin_Moller_2005} the bounds are recomputed from the intermediate state); the accepted fractions accumulate, $\ell_{ij}^{[p]}=\ell_{ij}^{[p-1]}+\ell_{ij}^{(p)}\big(1-\ell_{ij}^{[p-1]}\big)\in[0,1]$, and remain symmetric. Since $\sfV^{[p-1]}$ satisfies the bounds, every pass is again a convex combination of admissible states, so Theorem~\ref{t:ho_max_p} and the conservation statements below hold after every pass with $\ell_{ij}$ replaced by $\ell_{ij}^{[p]}$. A single pass treats every edge separately: at a node whose base state lies on a bound it rejects every flux that points out of the admissible interval, even when the fluxes on the other edges of the node would compensate it. Later passes recover part of such fluxes. We use three passes; they require no additional graph-Laplacian solve.

\begin{theorem}[Bounds of the limited solution]
  \label{t:ho_max_p}
  Let bounds $\sfU_i^{\min}\le\sfU_i^{\max}$ be given such that
  $\sfU_i^{\min}\le\overline\sfU_i^{n+1}\le\sfU_i^{\max}$ for all $i\in\calV$. Then
  $\sfV^{n+1}$ defined by \eqref{eq:lijAij} with the limiter \eqref{eq:lij} satisfies
  \[
    \sfU_i^{\min}
    \le
    \sfV_i^{n+1}
    \le
    \sfU_i^{\max},
    \qquad
    \forall i\in\calV .
  \]
  The assumption holds for the local bounds \eqref{eq:Umax-Umin} under the CFL condition
  \eqref{eq:cfl}, with $\nu$ replaced by $\nu/r_{\max}$ in the $\polP_k$ form, by
  Theorem~\ref{Thm:Min_Max_first_order} and Lemma~\ref{lem:LO_pk}.
\end{theorem}

\begin{proof}
  Fix $i\in\calV$ and $j\in\calI(S_i^{\polP_1})\setminus\{i\}$. The state $\overline\sfU_i^{n+1}$
  lies in $[\sfU_i^{\min},\sfU_i^{\max}]$ by assumption, and so does
  $\overline\sfU_i^{n+1}+\theta_{ij}\sfP_{ij}^{n+1}$ by the definition of $\theta_{ij}$. Hence every
  point of the segment between them lies in $[\sfU_i^{\min},\sfU_i^{\max}]$, in particular
  $\overline\sfU_i^{n+1}+\ell_{ij}\sfP_{ij}^{n+1}$, since $0\le\ell_{ij}=\min(\theta_{ij},\theta_{ji})\le\theta_{ij}$.
  If $\iota=1$, the same argument applies to $\overline\sfU_i^{n+1}+{\eta}_i\sfP_i^{\sigma,n+1}$.
  The weights in \eqref{eq:convcom} are nonnegative and sum to one,
  $\big(\operatorname{card}(S_i^{\polP_1})-1+\iota\big)\lambda_j(i)=1$, so $\sfV_i^{n+1}$ is a convex
  combination of states in $[\sfU_i^{\min},\sfU_i^{\max}]$.
\end{proof}

\begin{corollary}[Invariant-domain preservation]\label{cor:IDP}
  If $[\sfU_i^{\min},\sfU_i^{\max}]\subset[u_{\min},u_{\max}]$ for all $i\in\calV$, then all
  nodal values of $\sfV^{n+1}$ lie in the invariant interval $[u_{\min},u_{\max}]$. Since every
  stage of the SSP Runge--Kutta method is a forward-Euler step and the method is a convex
  combination of such steps, the same holds after each time step. In the fine-$\polP_1$ form,
  $\sfV^{n+1}$ is the state before the element-local exchange, which preserves the bounds by
  construction (Section~\ref{Sec:local_exchange}).
\end{corollary}

\begin{remark}[Values between the nodes]\label{rem:between_nodes}
  \Cref{cor:IDP} concerns the nodal values. For Lagrange elements of degree $k\ge2$ it does not imply that $v_h^{n+1}(\bx)\in[u_{\min},u_{\max}]$ between the nodes, since the basis functions change sign; this holds for $\polP_1$ and for Bernstein bases, see \citep[Sec.~3.1]{Guermond_Nazarov_Popov_2024}.
\end{remark}

\begin{proposition}[Local conservation in the $\polP_k$ form]\label{thm:pk_form}
  Assume $m_i^{\polP_k}>0$ for all $i\in\calV$ (\eg\ $\polP_1$, $\polP_3$, $\polP_5$ on triangles), and let $\sfV^{n+1}$ be given by \eqref{eq:lijAij} in the $\polP_k$ form. Then the update is in conservation form with respect to the $\polP_k$ masses,
  \[
    m_i^{\polP_k}\big(\sfV_i^{n+1}-\sfU_i^n\big)=m_i^{\polP_1}\big(\sfU_i^{{\rm L},n+1}-\sfU_i^n\big)+\sum_{j\in\calI(S_i^{\polP_1})}\ell_{ij}\sfF_{ij}^{n+1},\qquad i\in\calV,
  \]
  where the first term on the right is the low-order update \eqref{eq:LO} and the second consists of antisymmetric limited fluxes. Under the assumption of Lemma~\ref{lem:zero_sum},
  $\int_\Omega v_h^{n+1}\,\ud\bx=\Mk(\sfV^{n+1})=\Mk(\sfU^n)$; in general, $\Mk(\sfV^{n+1})-\Mk(\sfU^n)$ is the low-order boundary flux (Remark~\ref{rem:boundary_flux}).
\end{proposition}

\begin{proof}
  The identity follows from \eqref{eq:LOk} and \eqref{eq:lijAij} with $\overline\sfU^{n+1}=\widehat\sfU^{{\rm L},n+1}$ and $\mu_i=m_i^{\polP_k}$. Summing over $i\in\calV$, the limited fluxes cancel since $\ell_{ij}\sfF_{ij}^{n+1}=-\ell_{ji}\sfF_{ji}^{n+1}$, and $\Mone(\sfU^{{\rm L},n+1})=\Mone(\sfU^n)$ under the assumption of Lemma~\ref{lem:zero_sum}.
\end{proof}

The bounds of $\sfV^{n+1}$ in the $\polP_k$ form follow from Theorem~\ref{t:ho_max_p} and Lemma~\ref{lem:LO_pk}, and $\sfV^{n+1}=\sfU^{{\rm H},n+1}$ if no flux is limited, by \eqref{eq:HO_LO}.

\begin{proposition}[{$\polP_k$-mass defect of the fine-$\polP_1$ form}]\label{prop:p1_defect}
  Let $\sfV^{n+1}$ be given by \eqref{eq:lijAij} in the fine-$\polP_1$ form, and assume the setting of Lemma~\ref{lem:zero_sum}. Then
    \[
      \Mone(\sfV^{n+1})-\Mone(\sfU^n)=\sum_{i\in\calV}{\eta}_i\sfs_i^{n+1},
      \qquad
      \Mk(\sfV^{n+1})-\Mk(\sfU^n)=\sum_{K\in\calT_h}D_K,
    \]
    with the element defect
    \begin{equation}\label{eq:DK}
      D_K:=E_K(\sfW^{n+1})=\sum_{i\in\calI(K)}\big(m_{i,K}^{\polP_1}-m_{i,K}^{\polP_k}\big)\sfW_i^{n+1},\qquad
      \sfW_i^{n+1}:={\eta}_i\big(\sfU_i^{{\rm H},n+1}-\sfU_i^n\big)-\big(\sfV_i^{n+1}-\sfU_i^n\big).
    \end{equation}
    In particular, $D_K=0$ if $m_{i,K}^{\polP_1}=m_{i,K}^{\polP_k}$ for all $i\in\calI(K)$ (for instance for $k=1$), or if ${\eta}_i=1$ and $\ell_{ij}=1$ for all $i\in\calI(K)$ and $j\in\calI(S_i^{\polP_1})\setminus\{i\}$. Hence $\Mk(\sfV^{n+1})=\Mk(\sfU^n)$ if no flux and no source is limited, and $D_K\neq0$ only on elements that contain a node where the limiter is active.
\end{proposition}
\begin{proof}
  By \eqref{eq:lijAij} with $\mu_i=m_i^{\polP_1}$, $\overline\sfU^{n+1}=\sfU^{{\rm L},n+1}$ and $\sigma_i^{n+1}=\sfs_i^{n+1}$,
    \[
      m_i^{\polP_1}\big(\sfV_i^{n+1}-\sfU_i^n\big)=m_i^{\polP_1}\big(\sfU_i^{{\rm L},n+1}-\sfU_i^n\big)+\sum_{j\in\calI(S_i^{\polP_1})}\ell_{ij}\sfF_{ij}^{n+1}+{\eta}_i\sfs_i^{n+1}.
    \]
    Summing over $i\in\calV$, the limited fluxes cancel by antisymmetry and $\Mone(\sfU^{{\rm L},n+1})=\Mone(\sfU^n)$ under the assumption of Lemma~\ref{lem:zero_sum}, which gives the first identity. Writing $m_i^{\polP_k}=m_i^{\polP_1}-(m_i^{\polP_1}-m_i^{\polP_k})$,
    \[
      \Mk(\sfV^{n+1})-\Mk(\sfU^n)=\Mone(\sfV^{n+1})-\Mone(\sfU^n)-\sum_{i\in\calV}\big(m_i^{\polP_1}-m_i^{\polP_k}\big)\big(\sfV_i^{n+1}-\sfU_i^n\big),
    \]
    and inserting the first identity and the definition of $\sfs_i^{n+1}$ gives $\Mk(\sfV^{n+1})-\Mk(\sfU^n)=\sum_{i\in\calV}\big(m_i^{\polP_1}-m_i^{\polP_k}\big)\sfW_i^{n+1}$. Since $m_i^{\polP_1}=\sum_{K\ni i}m_{i,K}^{\polP_1}$ and $m_i^{\polP_k}=\sum_{K\ni i}m_{i,K}^{\polP_k}$, the sum regroups element by element into $\sum_{K}E_K(\sfW^{n+1})$. If ${\eta}_i=1$ and $\ell_{ij}=1$ for all $j$, then $\sfV_i^{n+1}=\sfU_i^{{\rm H},n+1}$ by \eqref{eq:HO_LO_p1}, so $\sfW_i^{n+1}=0$.
\end{proof}
The defect $\sum_{K\in\calT_h}D_K$ is removed by the element-local exchange of Section~\ref{Sec:local_exchange}.

\begin{remark}[Relaxed bounds]\label{remark:relaxed_bounds}
  Strict local bounds clip smooth extrema and reduce the accuracy there; this holds for any
    limiter that confines the update to the range of the data on the local stencil
    \citep{Lohmann_etal_2017,Kuzmin_2020,Kuzmin_Quezada_2020}, and is usually avoided by relaxing
    the bounds \citep{Guermond_Nazarov_Popov_Tomas_2018} or by a smoothness indicator
    \citep{Lohmann_etal_2017,Hajduk_etal_2020b}. We use
    \begin{equation}\label{eq:ri}
      \widetilde\sfU_i^{n,\min}:=\max\big(\sfU_i^{n,\min}-C_{\rm rel}r_i\,\delta u,\ \min u_0\big),\quad
      \widetilde\sfU_i^{n,\max}:=\min\big(\sfU_i^{n,\max}+C_{\rm rel}r_i\,\delta u,\ \max u_0\big),
    \end{equation}
    where 
    $
    r_i=\Big(\frac{m_i^{\polP_1}}{|\Omega|}\Big)^{\frac{3}{2d}},
    $
    $\delta u:=\max u_0-\min u_0$ and $C_{\rm rel}\ge0$ is dimensionless ($C_{\rm rel}=1$ in all
    computations), so that $r_i\sim h^{3/2}$ and the relaxation is invariant under shifts and
    rescalings of $u$. The same dimensionless mesh size and exponent are used in
    \citep[Sec.~5.4]{Kronbichler_Maier_Tomas_2025}, there as a multiplicative relaxation
    $\sfU_i^{n,\min}(1-c_r r_i)$, $\sfU_i^{n,\max}(1+c_r r_i)$; the exponent goes back to
    \citep[Sec.~4.7]{Guermond_Nazarov_Popov_Tomas_2018}. That form suits positive quantities such as the density; for a scalar of arbitrary sign it tightens a negative upper bound, depends on the origin of $u$, and does not relax a bound equal to zero, e.g.\ on a plateau of the solution. We therefore relax additively, scaled by $\delta u$.
    The relaxed bounds are recomputed from $\sfU^n$ at every time level (and every Runge--Kutta stage) and replace $\sfU_i^{n,\min}$, $\sfU_i^{n,\max}$ in the definition of $\theta_{ij}$. They contain the base state $\overline\sfU_i^{n+1}$, so Theorem~\ref{t:ho_max_p} with $\sfU_i^{\min}=\widetilde\sfU_i^{n,\min}$, $\sfU_i^{\max}=\widetilde\sfU_i^{n,\max}$ and
    Corollary~\ref{cor:IDP} apply: the global bounds hold exactly, and the local bounds are recovered
    as $h\to0$. For $C_{\rm rel}>0$ the method is therefore IDP for the invariant interval $[u_{\min},u_{\max}]$, while an interval $[a,b]\subset[u_{\min},u_{\max}]$ containing $\sfU^n$ is preserved only up to $C_{\rm rel}r_i\,\delta u=\mathcal O(h^{3/2})$; for $C_{\rm rel}=0$ every interval is preserved. The effect of $C_{\rm rel}$, and of plateaus where the solution attains a global
    bound, is reported in Section~\ref{Sec:numerics}.
\end{remark}

\begin{remark}[Post-step correction]\label{rem:post_step}
  For comparison in Section~\ref{Sec:numerics} we also use a variant that limits the stages with the relaxed local bounds only, without the intersection with $[\min u_0,\max u_0]$. After each full time step, every nodal value is clipped to $[\min u_0,\max u_0]$; the removed $\polP_k$ mass is returned inside each element by the exchange of Section~\ref{Sec:local_exchange}, and a remainder that does not fit is moved by antisymmetric fine-$\polP_1$ edge fluxes to the nearest nodes with room. This conserves the $\polP_k$ mass and enforces the global bounds once per step, as in the final-stage limiting of \citep{Kuzmin_etal_2022}, but the transport of the remainder is not local.
\end{remark}

\subsection{Degrees with nonpositive lumped masses: element-local exchange}
\label{Sec:local_exchange}
For $\polP_2$ and $\polP_4$ we use the fine-$\polP_1$ form and remove the defect \eqref{eq:DK} element by element. For every $K\in\calT_h$ we seek increments $c^K_i$, $i\in\calI(K)$, supported on the nodes with $m_i^{\polP_k}>0$, such that
\begin{equation}\label{eq:exchange}
  \sum_{i\in\calI(K)}m_i^{\polP_k}c^K_i=-D_K .
\end{equation}
Let $n_i$ be the number of coarse elements containing $i$ and let
$R_i^+:=m_i^{\polP_k}(\widetilde\sfU_i^{n,\max}-\sfV_i^{n+1})$, $R_i^-:=m_i^{\polP_k}(\sfV_i^{n+1}-\widetilde\sfU_i^{n,\min})$ be the nodal capacities, with the bounds \eqref{eq:ri} used by the limiter, so that $R_i^\pm\ge0$ by Theorem~\ref{t:ho_max_p}; element $K$ may use the share $R_i^\pm/n_i$. If $-D_K>0$ we set
\[
  c^K_i:=\alpha_K\frac{R_i^+/n_i}{m_i^{\polP_k}},\qquad
  \alpha_K:=\min\Big\{1,\frac{-D_K}{\sum_{j\in\calI(K),\,m_j^{\polP_k}>0}R_j^+/n_j}\Big\},
\]
with $\alpha_K:=0$ if the denominator vanishes, and symmetrically with $R^-$ and a minus sign if $-D_K<0$. The final state is 
\begin{equation}\label{eq:mass_fix}
  \sfU_i^{n+1}=\sfV_i^{n+1}+\sum_{K\ni i}c_i^K, \quad i\in\calV.
\end{equation}

\begin{proposition}[Element-local exchange]\label{prop:exchange}
  The state $\sfU^{n+1}$ satisfies the local bounds \eqref{eq:ri}. If $\alpha_K$ is not truncated for any $K$, i.e.\ every element has enough capacity, then $\Mk(\sfU^{n+1})=\Mk(\sfU^n)$ exactly under the assumption of Lemma~\ref{lem:zero_sum}. Every increment is confined to the element whose defect it compensates, and $\sum_{i\in\calI(K)} m_i^{\polP_1}|c_i^K|\le C|D_K|$.
\end{proposition}

\begin{proof}
  The shares of a node sum to its capacity, so $\sfV_i+\sum_{K \ni i}c_i^K\in[\widetilde\sfU_i^{n,\min},\widetilde\sfU_i^{n,\max}]$; increments of opposite sign from different elements stay within their own side of the interval. If no $\alpha_K$ is truncated, \eqref{eq:exchange} holds for every $K$, hence $\Mk(\sfU^{n+1})=\Mk(\sfV^{n+1})-\sum_{K\in\calT_h}D_K=\Mk(\sfU^n)$ by Proposition~\ref{prop:p1_defect}. Finally, $\sum_{i\in\calI(K)}m_i^{\polP_k}|c_i^K|=\alpha_K\sum_{j\in\calI(K)}R_j^\pm/n_j\le|D_K|$, and on the nodes with $m_i^{\polP_k}>0$ the ratio $m_i^{\polP_1}/m_i^{\polP_k}$ is bounded by the last column of Table~\ref{tab:lumped}, independently of the mesh; this gives the last bound with $C=\max\{m_i^{\polP_1}/m_i^{\polP_k}:\ i\in\calV,\ m_i^{\polP_k}>0\}$.
\end{proof}

\begin{lemma}[The two masses have the same weak limit]\label{lem:mismatch}
  Let $t^n=n\tau$, $n=0,\dots,N$, with $N\tau=T$, let $\phi\in C^\infty_c(\Omega\times[0,T))$, $\phi_i^n:=\phi(\bN_i,t^n)$, and assume $\|\sfU^n\|_{\ell^\infty}\le C_\infty$ and $\sum_{K\in\calT_h}|K|\operatorname{osc}_K\sfU^n\le C_{\rm BV}\,h$ for all $n$. Then
    \[
      \Big|\sum_{n=0}^{N-1}\sum_{i\in\calV}\phi_i^n\big(m_i^{\polP_1}-m_i^{\polP_k}\big)\big(\sfU_i^{n+1}-\sfU_i^n\big)\Big|\le C h,
    \]
    with $C$ depending on $\phi$, $T$, $C_\infty$ and $C_{\rm BV}$.
\end{lemma}
\begin{proof}
  Write $w_i:=m_i^{\polP_1}-m_i^{\polP_k}$ and denote the sum by $S$. \\
    \emph{Step 1 (summation by parts in time).} For sequences $(a^n)$, $(b^n)$, $\sum_{n=0}^{N-1}a^n(b^{n+1}-b^n)=a^{N-1}b^N-a^0b^0-\sum_{n=1}^{N-1}(a^n-a^{n-1})b^n$. We apply it at every node with $a^n=\phi_i^nw_i$, $b^n=\sfU_i^n$. Since $\phi$ has compact support in $[0,T)$, $\phi_i^{N-1}=0$ for $\tau$ small enough, hence
    \[
      S=-\sum_{i\in\calV}\phi_i^0w_i\sfU_i^0-\sum_{n=1}^{N-1}\sum_{i\in\calV}\big(\phi_i^n-\phi_i^{n-1}\big)w_i\sfU_i^n .
    \]
    The time difference now acts on $\phi$, which is $\mathcal O(\tau)$, and not on $\sfU^n$, which is not small near shocks.
    \emph{Step 2 (element splitting).} Every term has the form $\sum_i\psi_iw_i\sfU_i^n$ with nodal values $\psi_i$. Since $w_i=\sum_{K\ni i}(m_{i,K}^{\polP_1}-m_{i,K}^{\polP_k})$, we have $\sum_i\psi_iw_i\sfU_i^n=\sum_{K\in\calT_h} E_K(\psi\sfU^n)$, where $(\psi\sfU^n)_i=\psi_i\sfU_i^n$. \\
    \emph{Step 3 ($E_K$ vanishes on constants).} By \eqref{eq:EK_bound} and $\operatorname{osc}_K(\psi\sfU^n)\le\|\psi\|_{\ell^\infty}\operatorname{osc}_K\sfU^n+\|\sfU^n\|_{\ell^\infty}\operatorname{osc}_K\psi\le\|\psi\|_{\ell^\infty}\operatorname{osc}_K\sfU^n+h\|\nabla\psi\|_{L^\infty}\|\sfU^n\|_{\ell^\infty}$, where $\psi$ is the restriction of a smooth function to the nodes,
    \[
      \Big|\sum_{K\in\calT_h} E_K(\psi\sfU^n)\Big|\le C\Big(\|\psi\|_{L^\infty}\sum_{K\in\calT_h} |K|\operatorname{osc}_K\sfU^n+h\|\nabla\psi\|_{L^\infty}C_\infty|\Omega|\Big)\le C\,h\big(\|\psi\|_{L^\infty}+\|\nabla\psi\|_{L^\infty}\big).
    \]
    \emph{Step 4 (summation over $n$).} For $\psi=\phi(\cdot,0)$ the bound is $\mathcal O(h)$. For $\psi=\phi(\cdot,t^n)-\phi(\cdot,t^{n-1})$ we have $\|\psi\|_{L^\infty}\le\tau\|\partial_t\phi\|_{L^\infty}$ and $\|\nabla\psi\|_{L^\infty}\le\tau\|\partial_t\nabla\phi\|_{L^\infty}$, so each of the $N-1$ terms is $\mathcal O(\tau h)$, and their sum is $\mathcal O(Th)$. This uses $\partial_t\nabla\phi$, hence a smooth test function; $C^\infty_c$ is dense in $C^1_c$.
\end{proof}

\begin{remark}[Consistency with the weak formulation]\label{rem:LW}
  With the notation and the bounds of Lemma \ref{lem:mismatch}, which we assume but do not prove, the corrections are compatible with the Lax--Wendroff argument. In the $\polP_k$ form the limited update is in conservation form with antisymmetric edge fluxes (Proposition~\ref{thm:pk_form}), so the classical argument applies directly. In the fine-$\polP_1$ form the lumping source converts between the two masses, and by Lemma~\ref{lem:mismatch} testing with either family of masses gives the same weak limit. The exchange returns to each element the defect created there and moves mass only between nodes of that element; since $|\phi_i^n-\phi_j^n|\le Ch\|\nabla\phi\|_{L^\infty}$ for $i,j\in\calI(K)$, its residual in the discrete weak form is bounded by $Ch\|\nabla\phi\|_{L^\infty}\sum_n\sum_K|D_K^n|$, which vanishes as $h\to0$ if $\sum_n\sum_K|D_K^n|$ is bounded. A global redistribution offers no such bound, and it has been observed to converge to solutions that violate the Rankine--Hugoniot condition \citep[Rem.~2.4]{Guermond_Wang_2025}.
\end{remark}

This exchange is of the borrow-and-lend type of \citep{Guermond_2026} and is related to the local residual distribution used to build the nonlinear low-order scheme of \citep[Sec.~5]{Hajduk_etal_2020a}, but it acts only on the defect left by the flux limiter, which is supported on the limited elements and small: $|D_K|\le C|K|\operatorname{osc}_K\sfW^{n+1}$ by \eqref{eq:EK_bound}, and $\sfW_i^{n+1}=0$ wherever no flux is limited; in the smooth tests of Section~\ref{Sec:numerics} every element has enough capacity, and the $\polP_k$ mass is conserved exactly. In the tests with discontinuities the capacity is insufficient on a small fraction of the elements (about $2\%$ of the element-stages); the $\polP_k$ mass that is not recovered there, summed over all time steps, is at most about $3\times10^{-3}$ of the total mass in our runs. Extending the exchange to the patch of $K$ would recover it and keep the correction local; we have not implemented this.

\section{Numerical illustration}
\label{Sec:numerics}

In this section we demonstrate the proposed numerical algorithm for solving several linear and nonlinear problems from scalar conservation laws. All computations are carried out using the FEniCSx project \citep{BarattaEtal2023}, an open-source finite element library. The implementation uses the C++ interface of the FEniCSx project and is independent of the space dimension and of the polynomial degree.

In this work, we restrict the numerical tests to two space dimensions and we use polynomial spaces from first to fourth order. The unstructured meshes used in all the tests below are Delaunay triangulations generated with Gmsh \citep{Geuzaine_Remacle_2009}. We use the
fourth-order, five-stage strong stability preserving explicit Runge-Kutta
method of \citep{Kraaijevanger_1991}. The time step is computed from the
CFL condition \eqref{eq:cfl}, by \eqref{eq:dt_nu} of Remark~\ref{rem:cfl}. The limiting procedure is applied at every
nodal point and at each Runge-Kutta stage. In all computations we use the relaxed bounds of Remark~\ref{remark:relaxed_bounds} with $C_{\rm rel}=1$, intersected with the global bounds. The residual-viscosity constant $C_{\rm R}$ is given for each test. All presented convergence histories use relative norms in $L^1$ or $L^{\infty}$.

\subsection{Smooth transport problem}
We solve the linear transport equation
\[
  \p_t u + \bbetaa\SCAL \GRAD u = 0, \quad \bbetaa := (-1, -1)^\top,
\]
in $\Omega=(-1,1)^2$, with $u=0$ on $\p\Omega$, and $t\in[0, 0.2]$. With the $C^\infty$ bump
\[
  b(\bx;\bx_0,r_0):=
  \begin{cases}
    \exp\Big(\dfrac{r^2+2r_0^2}{r^2-r_0^2}\Big), & r=|\bx-\bx_0|<r_0,\\[2mm]
    0, & \text{otherwise},
  \end{cases}
\]
the initial condition is the dipole
\[
  u_0(\bx)=b\big(\bx;(\tfrac{7}{20},\tfrac12),0.4\big)-\tfrac12\,b\big(\bx;(-\tfrac{3}{10},-\tfrac15),0.3\big).
\]
The exact solution is $u(\bx,t)=u_0(\bx-\bbetaa t)$. Both supports stay inside $\Omega$ for $t\le0.2$, so the exact solution vanishes near $\p\Omega$, and the invariant interval $[\min u_0,\max u_0]=[-\frac12e^{-2},e^{-2}]$ is attained only at the two centres. Errors are computed against the exact solution interpolated in $\polP_{k+2}$.

The time step is \eqref{eq:dt_nu} with ${\rm CFL}=0.061$, small in order to reduce the temporal error, and we use $\gamma_{\rm cip}=1$ and three passes of iterated limiting. To separate the effect of the residual viscosity from that of the limiter, we use two high-order baselines: residual viscosity and CIP with $C_{\rm R}=0.1$, and CIP only, $C_{\rm R}=0$. The results are reported in Tables~\ref{tab:dipole} and~\ref{tab:dipole_linf}. With CIP only, the $L^1$ rates are optimal for all polynomial degrees, with and without limiting. With residual viscosity and CIP they are optimal up to $\polP_3$; for $\polP_4$ the unlimited and the limited scheme have the same reduced rate, so the reduction is caused by the residual viscosity and not by the limiter. In $L^\infty$ the limited errors are close to the unlimited ones up to $\polP_3$, while for $\polP_4$ the limited error stagnates at about $7.8\times10^{-6}$ on the two finest meshes, for both stabilizations. With the post-step correction of Remark~\ref{rem:post_step} instead of the intersection with $[\min u_0, \max u_0]$ (not shown in the tables), the $\polP_3$ and $\polP_4$ errors coincide with the unlimited ones in $L^\infty$ (e.g., $3.65\times10^{-7}$ for $\polP_4$ with $C_{\rm R}=0$ on the finest mesh), and for $\polP_4$ also in $L^1$. The loss is therefore caused by the intersection with the global bounds near the two points where they are attained.

\begin{table}[htbp]
  \centering\small
  \caption{Smooth transport of the dipole on unstructured meshes: relative $L^1$ errors and observed rates $p$ without limiting and with graph-Poisson limiting, for the high-order scheme with residual viscosity and CIP ($C_{\rm R}=0.1$) and with CIP only ($C_{\rm R}=0$).}
  \label{tab:dipole}
  \begin{tabular}{@{}cr@{\hspace{10pt}}rcrc@{\hspace{14pt}}rcrc@{}}
    \toprule
    & & \multicolumn{4}{c}{RV and CIP, $C_{\rm R}=0.1$} & \multicolumn{4}{c}{CIP only, $C_{\rm R}=0$} \\
    \cmidrule(lr{10pt}){3-6}\cmidrule(l){7-10}
    & & \multicolumn{2}{c}{no limiter} & \multicolumn{2}{c}{limiter} & \multicolumn{2}{c}{no limiter} & \multicolumn{2}{c}{limiter} \\
    \cmidrule(lr){3-4}\cmidrule(lr{10pt}){5-6}\cmidrule(lr){7-8}\cmidrule(l){9-10}
    & dofs & error & $p$ & error & $p$ & error & $p$ & error & $p$ \\
    \midrule
    \multirow{6}{*}{$\polP_1$} & 1664 & $6.99\mathrm{e}{-2}$ & -- & $6.99\mathrm{e}{-2}$ & -- & $6.80\mathrm{e}{-2}$ & -- & $6.80\mathrm{e}{-2}$ & -- \\
     & 4995 & $2.06\mathrm{e}{-2}$ & 2.22 & $2.06\mathrm{e}{-2}$ & 2.22 & $2.02\mathrm{e}{-2}$ & 2.21 & $2.02\mathrm{e}{-2}$ & 2.21 \\
     & 15915 & $5.29\mathrm{e}{-3}$ & 2.35 & $5.29\mathrm{e}{-3}$ & 2.35 & $5.21\mathrm{e}{-3}$ & 2.34 & $5.21\mathrm{e}{-3}$ & 2.34 \\
     & 50039 & $1.36\mathrm{e}{-3}$ & 2.37 & $1.36\mathrm{e}{-3}$ & 2.37 & $1.35\mathrm{e}{-3}$ & 2.36 & $1.35\mathrm{e}{-3}$ & 2.36 \\
     & 159660 & $3.72\mathrm{e}{-4}$ & 2.24 & $3.72\mathrm{e}{-4}$ & 2.24 & $3.69\mathrm{e}{-4}$ & 2.23 & $3.69\mathrm{e}{-4}$ & 2.23 \\
     & 501275 & $1.09\mathrm{e}{-4}$ & 2.14 & $1.09\mathrm{e}{-4}$ & 2.14 & $1.09\mathrm{e}{-4}$ & 2.13 & $1.09\mathrm{e}{-4}$ & 2.13 \\
    \midrule
    \multirow{6}{*}{$\polP_2$} & 1649 & $3.01\mathrm{e}{-2}$ & -- & $3.00\mathrm{e}{-2}$ & -- & $2.94\mathrm{e}{-2}$ & -- & $2.92\mathrm{e}{-2}$ & -- \\
     & 5017 & $5.41\mathrm{e}{-3}$ & 3.09 & $5.41\mathrm{e}{-3}$ & 3.08 & $5.31\mathrm{e}{-3}$ & 3.07 & $5.32\mathrm{e}{-3}$ & 3.06 \\
     & 15713 & $7.90\mathrm{e}{-4}$ & 3.37 & $7.90\mathrm{e}{-4}$ & 3.37 & $7.81\mathrm{e}{-4}$ & 3.36 & $7.81\mathrm{e}{-4}$ & 3.36 \\
     & 50125 & $1.22\mathrm{e}{-4}$ & 3.23 & $1.22\mathrm{e}{-4}$ & 3.23 & $1.21\mathrm{e}{-4}$ & 3.22 & $1.21\mathrm{e}{-4}$ & 3.21 \\
     & 158765 & $2.09\mathrm{e}{-5}$ & 3.06 & $2.09\mathrm{e}{-5}$ & 3.06 & $2.08\mathrm{e}{-5}$ & 3.05 & $2.08\mathrm{e}{-5}$ & 3.05 \\
     & 502677 & $3.62\mathrm{e}{-6}$ & 3.04 & $3.63\mathrm{e}{-6}$ & 3.04 & $3.61\mathrm{e}{-6}$ & 3.04 & $3.62\mathrm{e}{-6}$ & 3.04 \\
    \midrule
    \multirow{6}{*}{$\polP_3$} & 1765 & $1.91\mathrm{e}{-2}$ & -- & $1.92\mathrm{e}{-2}$ & -- & $1.90\mathrm{e}{-2}$ & -- & $1.91\mathrm{e}{-2}$ & -- \\
     & 5212 & $2.53\mathrm{e}{-3}$ & 3.73 & $2.57\mathrm{e}{-3}$ & 3.71 & $2.54\mathrm{e}{-3}$ & 3.72 & $2.58\mathrm{e}{-3}$ & 3.70 \\
     & 15580 & $2.84\mathrm{e}{-4}$ & 4.00 & $2.84\mathrm{e}{-4}$ & 4.02 & $2.85\mathrm{e}{-4}$ & 3.99 & $2.86\mathrm{e}{-4}$ & 4.02 \\
     & 50314 & $2.41\mathrm{e}{-5}$ & 4.21 & $2.39\mathrm{e}{-5}$ & 4.22 & $2.37\mathrm{e}{-5}$ & 4.24 & $2.39\mathrm{e}{-5}$ & 4.24 \\
     & 160195 & $2.25\mathrm{e}{-6}$ & 4.09 & $2.14\mathrm{e}{-6}$ & 4.17 & $2.00\mathrm{e}{-6}$ & 4.27 & $2.01\mathrm{e}{-6}$ & 4.27 \\
     & 500767 & $2.62\mathrm{e}{-7}$ & 3.77 & $2.32\mathrm{e}{-7}$ & 3.90 & $1.81\mathrm{e}{-7}$ & 4.22 & $1.83\mathrm{e}{-7}$ & 4.21 \\
    \midrule
    \multirow{6}{*}{$\polP_4$} & 1537 & $1.95\mathrm{e}{-2}$ & -- & $1.95\mathrm{e}{-2}$ & -- & $1.95\mathrm{e}{-2}$ & -- & $1.95\mathrm{e}{-2}$ & -- \\
     & 5521 & $1.96\mathrm{e}{-3}$ & 3.59 & $1.99\mathrm{e}{-3}$ & 3.57 & $1.96\mathrm{e}{-3}$ & 3.60 & $1.98\mathrm{e}{-3}$ & 3.58 \\
     & 16721 & $1.56\mathrm{e}{-4}$ & 4.57 & $1.58\mathrm{e}{-4}$ & 4.57 & $1.52\mathrm{e}{-4}$ & 4.62 & $1.54\mathrm{e}{-4}$ & 4.62 \\
     & 49953 & $1.09\mathrm{e}{-5}$ & 4.86 & $1.12\mathrm{e}{-5}$ & 4.83 & $9.26\mathrm{e}{-6}$ & 5.11 & $9.65\mathrm{e}{-6}$ & 5.06 \\
     & 158849 & $9.64\mathrm{e}{-7}$ & 4.20 & $9.94\mathrm{e}{-7}$ & 4.19 & $4.92\mathrm{e}{-7}$ & 5.08 & $5.27\mathrm{e}{-7}$ & 5.03 \\
     & 497609 & $1.62\mathrm{e}{-7}$ & 3.12 & $1.71\mathrm{e}{-7}$ & 3.08 & $2.81\mathrm{e}{-8}$ & 5.01 & $3.79\mathrm{e}{-8}$ & 4.61 \\
    \bottomrule
  \end{tabular}
\end{table}

\begin{table}[htbp]
  \centering\small
  \caption{Smooth transport of the dipole on unstructured meshes: relative $L^\infty$ errors and observed rates $p$; same runs as in Table~\ref{tab:dipole}.}
  \label{tab:dipole_linf}
  \begin{tabular}{@{}cr@{\hspace{10pt}}rcrc@{\hspace{14pt}}rcrc@{}}
    \toprule
    & & \multicolumn{4}{c}{RV and CIP, $C_{\rm R}=0.1$} & \multicolumn{4}{c}{CIP only, $C_{\rm R}=0$} \\
    \cmidrule(lr{10pt}){3-6}\cmidrule(l){7-10}
    & & \multicolumn{2}{c}{no limiter} & \multicolumn{2}{c}{limiter} & \multicolumn{2}{c}{no limiter} & \multicolumn{2}{c}{limiter} \\
    \cmidrule(lr){3-4}\cmidrule(lr{10pt}){5-6}\cmidrule(lr){7-8}\cmidrule(l){9-10}
    & dofs & error & $p$ & error & $p$ & error & $p$ & error & $p$ \\
    \midrule
    \multirow{6}{*}{$\polP_1$} & 1664 & $3.63\mathrm{e}{-2}$ & -- & $3.63\mathrm{e}{-2}$ & -- & $3.58\mathrm{e}{-2}$ & -- & $3.58\mathrm{e}{-2}$ & -- \\
     & 4995 & $1.60\mathrm{e}{-2}$ & 1.49 & $1.60\mathrm{e}{-2}$ & 1.49 & $1.58\mathrm{e}{-2}$ & 1.49 & $1.58\mathrm{e}{-2}$ & 1.49 \\
     & 15915 & $4.52\mathrm{e}{-3}$ & 2.18 & $4.52\mathrm{e}{-3}$ & 2.18 & $4.47\mathrm{e}{-3}$ & 2.18 & $4.47\mathrm{e}{-3}$ & 2.18 \\
     & 50039 & $1.43\mathrm{e}{-3}$ & 2.01 & $1.43\mathrm{e}{-3}$ & 2.01 & $1.42\mathrm{e}{-3}$ & 2.00 & $1.42\mathrm{e}{-3}$ & 2.00 \\
     & 159660 & $4.40\mathrm{e}{-4}$ & 2.03 & $4.40\mathrm{e}{-4}$ & 2.03 & $4.39\mathrm{e}{-4}$ & 2.03 & $4.39\mathrm{e}{-4}$ & 2.03 \\
     & 501275 & $1.38\mathrm{e}{-4}$ & 2.02 & $1.38\mathrm{e}{-4}$ & 2.02 & $1.38\mathrm{e}{-4}$ & 2.02 & $1.38\mathrm{e}{-4}$ & 2.02 \\
    \midrule
    \multirow{6}{*}{$\polP_2$} & 1649 & $1.72\mathrm{e}{-2}$ & -- & $1.70\mathrm{e}{-2}$ & -- & $1.70\mathrm{e}{-2}$ & -- & $1.69\mathrm{e}{-2}$ & -- \\
     & 5017 & $5.26\mathrm{e}{-3}$ & 2.12 & $5.26\mathrm{e}{-3}$ & 2.11 & $5.55\mathrm{e}{-3}$ & 2.01 & $5.55\mathrm{e}{-3}$ & 2.00 \\
     & 15713 & $1.08\mathrm{e}{-3}$ & 2.78 & $1.08\mathrm{e}{-3}$ & 2.78 & $1.09\mathrm{e}{-3}$ & 2.86 & $1.09\mathrm{e}{-3}$ & 2.86 \\
     & 50125 & $2.33\mathrm{e}{-4}$ & 2.64 & $2.33\mathrm{e}{-4}$ & 2.64 & $2.37\mathrm{e}{-4}$ & 2.63 & $2.37\mathrm{e}{-4}$ & 2.63 \\
     & 158765 & $4.77\mathrm{e}{-5}$ & 2.75 & $4.77\mathrm{e}{-5}$ & 2.75 & $4.78\mathrm{e}{-5}$ & 2.77 & $4.78\mathrm{e}{-5}$ & 2.77 \\
     & 502677 & $1.27\mathrm{e}{-5}$ & 2.29 & $1.27\mathrm{e}{-5}$ & 2.29 & $1.27\mathrm{e}{-5}$ & 2.30 & $1.27\mathrm{e}{-5}$ & 2.30 \\
    \midrule
    \multirow{6}{*}{$\polP_3$} & 1765 & $1.30\mathrm{e}{-2}$ & -- & $1.29\mathrm{e}{-2}$ & -- & $1.39\mathrm{e}{-2}$ & -- & $1.39\mathrm{e}{-2}$ & -- \\
     & 5212 & $2.48\mathrm{e}{-3}$ & 3.06 & $2.47\mathrm{e}{-3}$ & 3.06 & $2.47\mathrm{e}{-3}$ & 3.20 & $2.47\mathrm{e}{-3}$ & 3.19 \\
     & 15580 & $4.70\mathrm{e}{-4}$ & 3.03 & $4.67\mathrm{e}{-4}$ & 3.05 & $4.74\mathrm{e}{-4}$ & 3.01 & $4.74\mathrm{e}{-4}$ & 3.02 \\
     & 50314 & $1.09\mathrm{e}{-4}$ & 2.50 & $1.09\mathrm{e}{-4}$ & 2.48 & $1.09\mathrm{e}{-4}$ & 2.51 & $1.09\mathrm{e}{-4}$ & 2.51 \\
     & 160195 & $1.08\mathrm{e}{-5}$ & 3.99 & $1.08\mathrm{e}{-5}$ & 3.99 & $1.07\mathrm{e}{-5}$ & 3.99 & $1.07\mathrm{e}{-5}$ & 3.99 \\
     & 500767 & $1.43\mathrm{e}{-6}$ & 3.55 & $1.83\mathrm{e}{-6}$ & 3.11 & $1.43\mathrm{e}{-6}$ & 3.55 & $1.83\mathrm{e}{-6}$ & 3.11 \\
    \midrule
    \multirow{6}{*}{$\polP_4$} & 1537 & $1.09\mathrm{e}{-2}$ & -- & $1.08\mathrm{e}{-2}$ & -- & $1.08\mathrm{e}{-2}$ & -- & $1.08\mathrm{e}{-2}$ & -- \\
     & 5521 & $2.96\mathrm{e}{-3}$ & 2.03 & $2.96\mathrm{e}{-3}$ & 2.03 & $2.93\mathrm{e}{-3}$ & 2.05 & $2.93\mathrm{e}{-3}$ & 2.04 \\
     & 16721 & $4.06\mathrm{e}{-4}$ & 3.58 & $4.06\mathrm{e}{-4}$ & 3.58 & $4.06\mathrm{e}{-4}$ & 3.56 & $4.06\mathrm{e}{-4}$ & 3.57 \\
     & 49953 & $4.65\mathrm{e}{-5}$ & 3.96 & $4.65\mathrm{e}{-5}$ & 3.96 & $4.62\mathrm{e}{-5}$ & 3.97 & $4.62\mathrm{e}{-5}$ & 3.97 \\
     & 158849 & $3.48\mathrm{e}{-6}$ & 4.48 & $7.68\mathrm{e}{-6}$ & 3.11 & $3.48\mathrm{e}{-6}$ & 4.47 & $7.67\mathrm{e}{-6}$ & 3.10 \\
     & 497609 & $6.89\mathrm{e}{-7}$ & 2.84 & $7.84\mathrm{e}{-6}$ & -0.04 & $3.65\mathrm{e}{-7}$ & 3.95 & $7.84\mathrm{e}{-6}$ & -0.04 \\
    \bottomrule
  \end{tabular}
\end{table}

\textit{Solution on a plateau at a global bound.} The intersection is more restrictive when the solution equals a global bound on a whole region. To illustrate this we keep only the first term of the dipole, $u_0=b(\cdot;(\frac{7}{20},\frac12),0.4)$, so that $\min u_0=0$ is attained on the whole background. On the same meshes the unlimited errors converge at the optimal rates (with $C_{\rm R}=0$ also for $\polP_4$), but with the limiter the $L^1$ rate drops to about one for $\polP_2$--$\polP_4$ and the $L^\infty$ error stays at $2$--$3\times10^{-3}$ for all degrees. For instance, with $C_{\rm R}=0$ on the finest meshes, the $L^1$ error is $8.45\times10^{-5}$ instead of $2.78\times10^{-6}$ for $\polP_2$, and $8.94\times10^{-5}$ instead of $1.63\times10^{-8}$ for $\polP_4$. Near the edge of the support the bounds of all nodes collapse to $\min u_0$, so these nodes have no room, whereas the graph-Poisson fluxes, which are not local, do not vanish there; the limiter therefore cuts most of the fluxes in this region (see also Section~\ref{sec:conclusions}). This is confirmed by a $\polP_1$ experiment with $C_{\rm R}=0$ in which, since $m_i^{\polP_1}=m_i^{\polP_k}$, the graph-Poisson fluxes can be replaced by the standard local antidiffusive fluxes, with the same limiter, bounds and time step: the limited errors then coincide with the unlimited ones on the two finest meshes ($9.14\times10^{-5}$ in $L^1$ and $1.38\times10^{-4}$ in $L^\infty$ on the finest mesh), whereas with the graph-Poisson fluxes the $L^\infty$ error stays at $2.25\times10^{-3}$. The loss decreases with the time step, since the fluxes cut in a stage are proportional to its length: for $\polP_3$ with $C_{\rm R}=0$ on the mesh with $50314$ degrees of freedom, reducing the time step by the factors $2$ and $4$ reduces the limited $L^1$ error from $2.91\times10^{-4}$ to $1.99\times10^{-4}$ and $1.52\times10^{-4}$ (unlimited: $1.48\times10^{-5}$), but the $L^\infty$ error stays at about $4\times10^{-3}$. With the post-step correction of Remark~\ref{rem:post_step}, the limited errors again coincide with the unlimited ones, e.g., $1.74\times10^{-8}$ versus $1.63\times10^{-8}$ in $L^1$ and $2.13\times10^{-7}$ in $L^\infty$ for both, for $\polP_4$. On uniform meshes the same test is much less affected: the limited $L^1$ errors differ from the unlimited ones by at most $10\%$ on the finest meshes, and the effect is visible mainly in $L^\infty$ for $\polP_3$ and $\polP_4$.

\textit{Cost.} Table~\ref{tab:cost} reports the wall-clock times of the CIP-only runs on the finest meshes. All runs are serial on one core, and the four degrees were run at the same time on the same machine, so the times are indicative. With limiting, the time per step is $1.7$--$1.8$ times that without limiting, for all degrees. About a quarter of the total time is spent in the graph-Laplacian solves, six per time step (the five stages and the final combination of the SSP-RK(5,4) method); their cost does not depend on $k$ at a fixed number of degrees of freedom, since the fine-$\polP_1$ graph then has the same size. The rest of the limiter (bounds, three passes of limiting and, for $\polP_2$ and $\polP_4$, the exchange) takes $14$--$19\%$ of the total time. For $\polP_3$ in the $\polP_k$ form the time step is in addition divided by $r_{\max}=1.48$, see \eqref{eq:dt_nu}, so that the total time is $2.5$ times that of the unlimited scheme.

\begin{table}[htbp]
  \centering\small
  \caption{Smooth transport of the dipole, CIP only, finest meshes: number of time steps, wall-clock time per step without and with limiting, ratio of the total times, and share of the total time spent in the limiter and, as part of it, in the graph-Laplacian solves.}
  \label{tab:cost}
  \begin{tabular}{@{}cccrrrrrrr@{}}
    \toprule
    & & & \multicolumn{2}{c}{steps} & \multicolumn{2}{c}{time/step (s)} & total time & \multicolumn{2}{c}{share of total} \\
    \cmidrule(lr){4-5}\cmidrule(lr){6-7}\cmidrule(lr){8-8}\cmidrule(l){9-10}
    & mass form & dofs & no lim. & lim. & no lim. & lim. & lim./no lim. & limiter & graph \\
    \midrule
    $\polP_1$ & $\polP_k$ & 501275 & 2304 & 2304 & 0.80 & 1.37 & 1.70 & 41\% & 28\% \\
    $\polP_2$ & fine-$\polP_1$ & 502677 & 2334 & 2334 & 0.79 & 1.43 & 1.81 & 45\% & 26\% \\
    $\polP_3$ & $\polP_k$ & 500767 & 2527 & 3744 & 0.80 & 1.36 & 2.52 & 42\% & 27\% \\
    $\polP_4$ & fine-$\polP_1$ & 497609 & 2558 & 2558 & 0.83 & 1.46 & 1.76 & 44\% & 25\% \\
    \bottomrule
  \end{tabular}
\end{table}

\subsection{Three body rotation}
In this test, we want to solve a linear transport equation
\[
  \p_t u + \bbetaa\SCAL \GRAD u = 0, \quad \bbetaa := (-2\pi y, 2\pi x)^\top,
\]
in the unit disk $\Omega=\{\bx: |\bx|<1\}$, with $u=0$ on $\p\Omega$, and $t\in[0, 1]$ (one full revolution).

The initial condition consists of three bodies: a slotted cylinder, a cosine cone, and a conical hump. Let
$
  r_0=0.3,
$
and define
$
  r_1(x,y)=\big(x^2+(y-0.5)^2\big)^{{\frac12}},
  \ \
  r_2(x,y)=\big((x+0.5)^2+y^2\big)^{{\frac12}},
  \ \
  r_3(x,y)
  =
  \Bigg(
  \Big(x-\frac{\sqrt{3}}{4}\Big)^2
  +
  \Big(y+\frac14\Big)^2
  \Bigg)^{{\frac12}}.
$
The initial condition is
\[
  u_0(x,y)=
  \begin{cases}
    1,
    &
      r_1(x,y)\le r_0
      \ \text{and}\
      \left(|x|\ge 0.05 \ \text{or}\ y\ge 0.7\right),
    \\[2mm]
    \displaystyle
    \frac14
    \left(
    1+\cos\left(\pi \frac{r_2(x,y)}{r_0}\right)
    \right),
    &
      r_2(x,y)\le r_0,
    \\[3mm]
    \displaystyle
    1-\frac{r_3(x,y)}{r_0},
    &
      r_3(x,y)\le r_0,
    \\[2mm]
    0,
    &
      \text{otherwise}.
  \end{cases}
\]

 The time step is \eqref{eq:dt_nu} with ${\rm CFL}=0.23$; we use unstructured Delaunay triangulations of $\Omega$ with comparable numbers of degrees of freedom across the polynomial degrees, $C_{\rm R}=1$, $\gamma_{\rm cip}=1$ and three passes of iterated limiting. The bodies stay inside the disk of radius $0.8$, so the exact solution vanishes near $\p\Omega$, and the invariant interval is $[\min u_0,\max u_0]=[0,1]$. Errors are computed against the exact solution, the initial data rotated by the angle $2\pi t$, interpolated in $\polP_{k+2}$. The numerical results are depicted in Figure~\ref{fig:transport_unstr} for polynomial degrees $\polP_1$--$\polP_4$ and different mesh resolutions. 
All approximations converge to the exact solution. Without relaxation of the local bounds we observe clipping of the local extrema (not reported here).

In Table~\ref{tab:three_body_l1_convergence}, we present convergence history with respect to the relative $L^1$-norm of the error. One can observe almost linear convergence rates for all polynomials, with slightly improved rates for higher polynomial degrees.

\begin{table}[htbp]
  \centering\small
  \caption{Convergence history for the nonsmooth three-body rotation problem on unstructured Delaunay triangulations of the unit disk. Only the $L^1$ error is reported.}
  \label{tab:three_body_l1_convergence}
  \vspace{0.05in}
      \begin{tabular}{rcccc@{\hspace{14pt}}rcccc}
      \toprule
      \multicolumn{5}{c}{$\polP_1$} & \multicolumn{5}{c}{$\polP_2$} \\
      \cmidrule(r){1-5}\cmidrule(l){6-10}
      \multirow{2}{*}{dofs} & \multicolumn{2}{c}{RV} & \multicolumn{2}{c}{RV + limiter}
      & \multirow{2}{*}{dofs} & \multicolumn{2}{c}{RV} & \multicolumn{2}{c}{RV + limiter} \\
      \cmidrule(lr){2-5}\cmidrule(lr){7-10}
      & $L^1$ & $p$ & $L^1$ & $p$ & & $L^1$ & $p$ & $L^1$ & $p$ \\
      \midrule
      1607 & $5.76\times 10^{-1}$ & -- & $5.37\times 10^{-1}$ & --
      & 1573 & $4.05\times 10^{-1}$ & -- & $4.07\times 10^{-1}$ & -- \\
      5069 & $3.52\times 10^{-1}$ & 0.86 & $3.48\times 10^{-1}$ & 0.76
      & 5045 & $2.68\times 10^{-1}$ & 0.71 & $2.77\times 10^{-1}$ & 0.66 \\
      15973 & $2.28\times 10^{-1}$ & 0.76 & $2.23\times 10^{-1}$ & 0.78
      & 16039 & $1.62\times 10^{-1}$ & 0.87 & $1.72\times 10^{-1}$ & 0.82 \\
      49800 & $1.43\times 10^{-1}$ & 0.82 & $1.40\times 10^{-1}$ & 0.82
      & 50177 & $1.05\times 10^{-1}$ & 0.75 & $1.11\times 10^{-1}$ & 0.78 \\
      160112 & $9.40\times 10^{-2}$ & 0.72 & $8.91\times 10^{-2}$ & 0.77
      & 159841 & $6.63\times 10^{-2}$ & 0.80 & $6.94\times 10^{-2}$ & 0.80 \\
      500774 & $6.12\times 10^{-2}$ & 0.75 & $5.79\times 10^{-2}$ & 0.76
      & 500695 & $4.23\times 10^{-2}$ & 0.79 & $4.43\times 10^{-2}$ & 0.79 \\
      \bottomrule
    \end{tabular}

  \vspace{0.15in}

      \begin{tabular}{rcccc@{\hspace{14pt}}rcccc}
      \toprule
      \multicolumn{5}{c}{$\polP_3$} & \multicolumn{5}{c}{$\polP_4$} \\
      \cmidrule(r){1-5}\cmidrule(l){6-10}
      \multirow{2}{*}{dofs} & \multicolumn{2}{c}{RV} & \multicolumn{2}{c}{RV + limiter}
      & \multirow{2}{*}{dofs} & \multicolumn{2}{c}{RV} & \multicolumn{2}{c}{RV + limiter} \\
      \cmidrule(lr){2-5}\cmidrule(lr){7-10}
      & $L^1$ & $p$ & $L^1$ & $p$ & & $L^1$ & $p$ & $L^1$ & $p$ \\
      \midrule
      1657 & $3.68\times 10^{-1}$ & -- & $3.82\times 10^{-1}$ & --
      & 1677 & $3.71\times 10^{-1}$ & -- & $3.87\times 10^{-1}$ & -- \\
      5179 & $2.24\times 10^{-1}$ & 0.87 & $2.39\times 10^{-1}$ & 0.82
      & 4969 & $2.51\times 10^{-1}$ & 0.72 & $2.77\times 10^{-1}$ & 0.62 \\
      16021 & $1.42\times 10^{-1}$ & 0.81 & $1.55\times 10^{-1}$ & 0.77
      & 15841 & $1.48\times 10^{-1}$ & 0.91 & $1.64\times 10^{-1}$ & 0.91 \\
      50341 & $8.70\times 10^{-2}$ & 0.86 & $9.24\times 10^{-2}$ & 0.90
      & 49573 & $9.05\times 10^{-2}$ & 0.86 & $9.87\times 10^{-2}$ & 0.89 \\
      159919 & $5.50\times 10^{-2}$ & 0.79 & $5.86\times 10^{-2}$ & 0.79
      & 160681 & $5.55\times 10^{-2}$ & 0.83 & $5.97\times 10^{-2}$ & 0.85 \\
      532864 & $3.39\times 10^{-2}$ & 0.80 & $3.63\times 10^{-2}$ & 0.80
      & 501921 & $3.42\times 10^{-2}$ & 0.85 & $3.67\times 10^{-2}$ & 0.85 \\
      \bottomrule
    \end{tabular}
\end{table}

\begin{figure}[htbp]
  \centering
  \setlength{\tabcolsep}{2pt}
  \renewcommand{\arraystretch}{1.0}

  \newcommand{\rowlabel}[1]{%
    \raisebox{0pt}[\height][\depth]{%
      \makebox[0.035\textwidth][c]{\rotatebox{90}{#1}}%
    }%
  }

  \begin{tabular}{ccccc}
   {\rotatebox{90}{\hspace{-0.3in} dofs \hspace{0.35in} $\polP_1$}}
    &
      \hspace{0.1in}
    \includegraphics[width=0.22\textwidth]{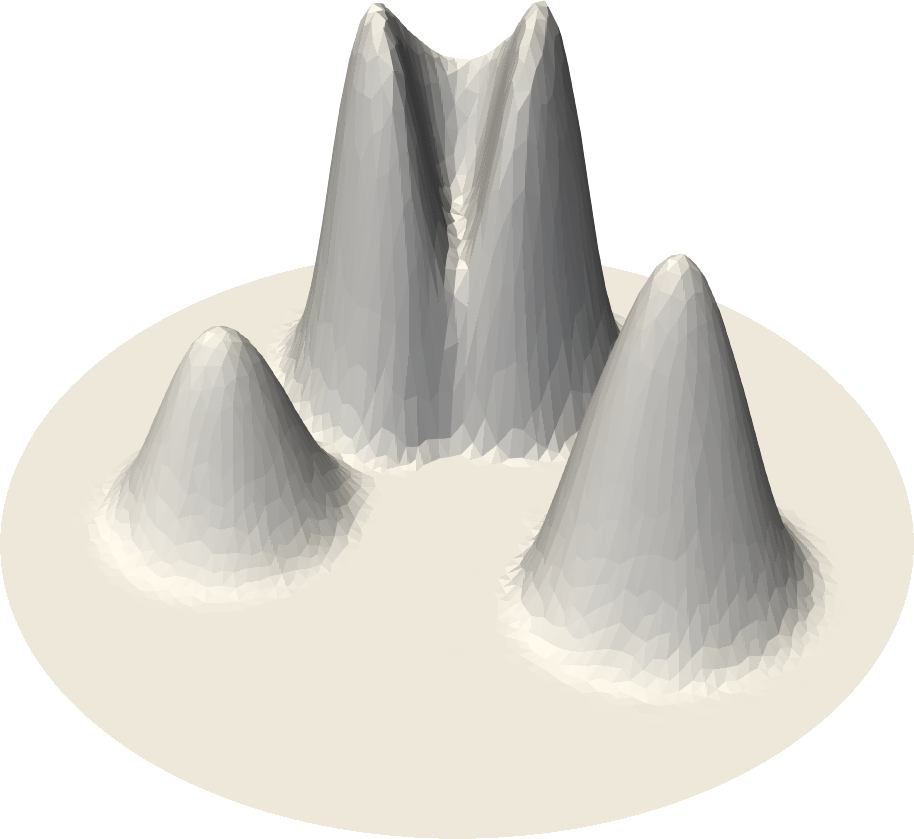}
    &
    \includegraphics[width=0.22\textwidth]{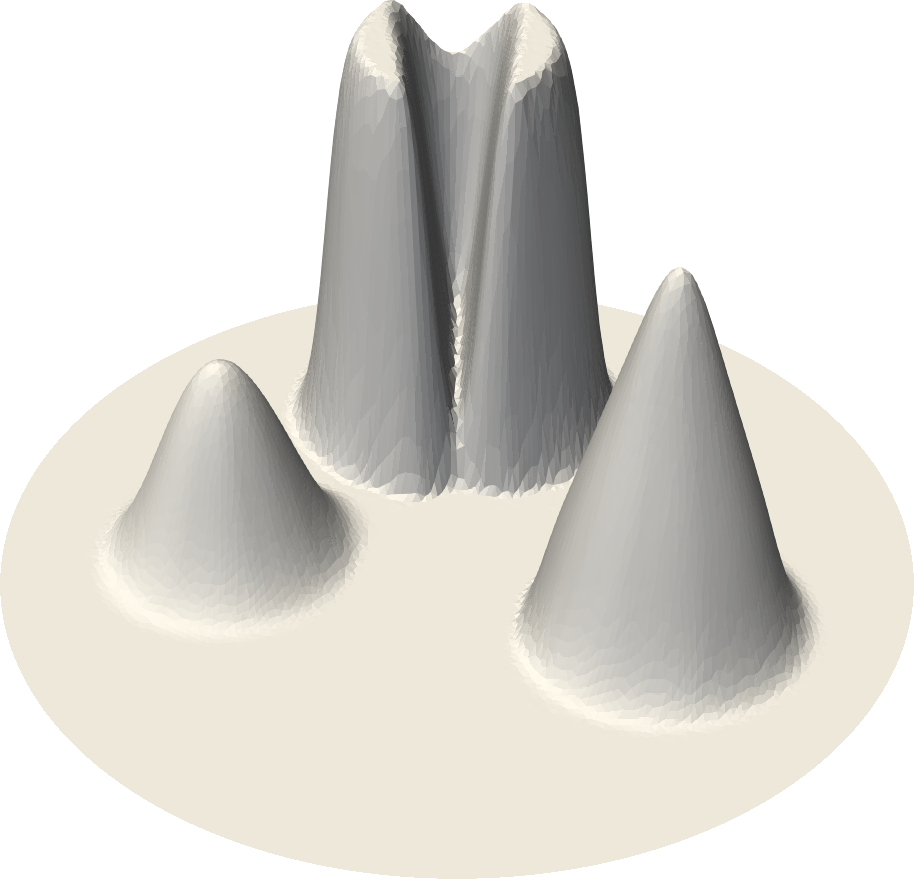}
    &
    \includegraphics[width=0.22\textwidth]{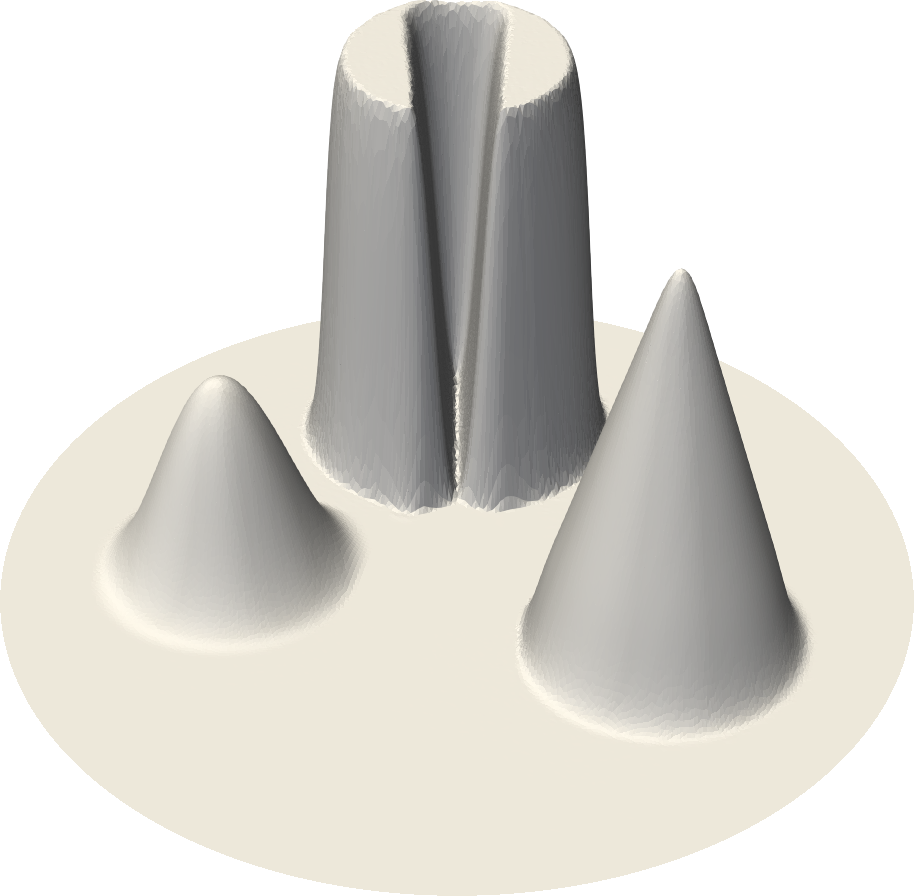}
    &
    \includegraphics[width=0.22\textwidth]{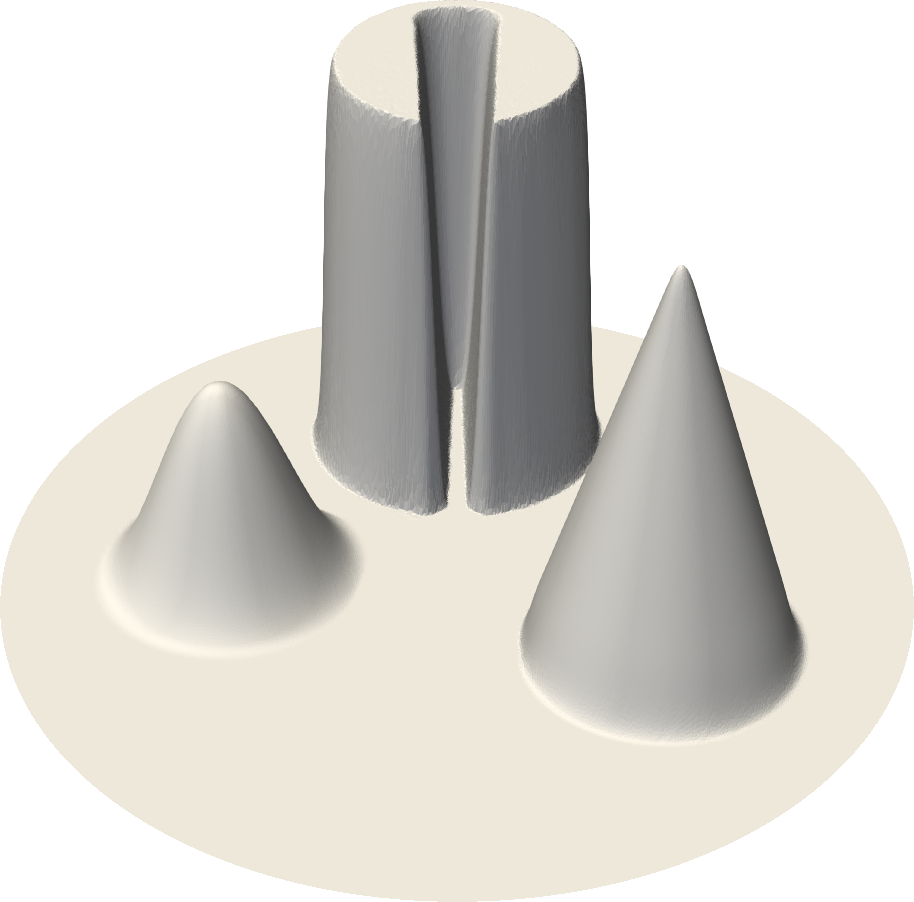}
    \\
    & 5069 & 15973 & 49800 & 160112 \\
    \\[2pt]
    \midrule
   {\rotatebox{90}{\hspace{-0.3in} dofs \hspace{0.35in} $\polP_2$}}
    &
      \hspace{0.1in}
    \includegraphics[width=0.22\textwidth]{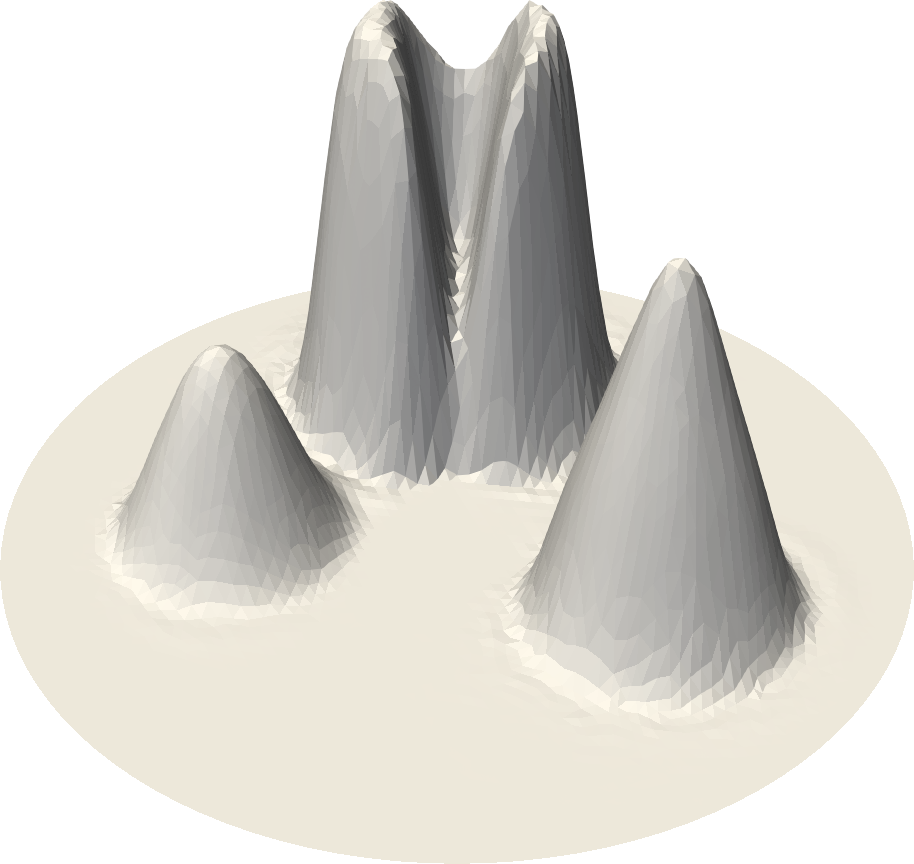}
    &
    \includegraphics[width=0.22\textwidth]{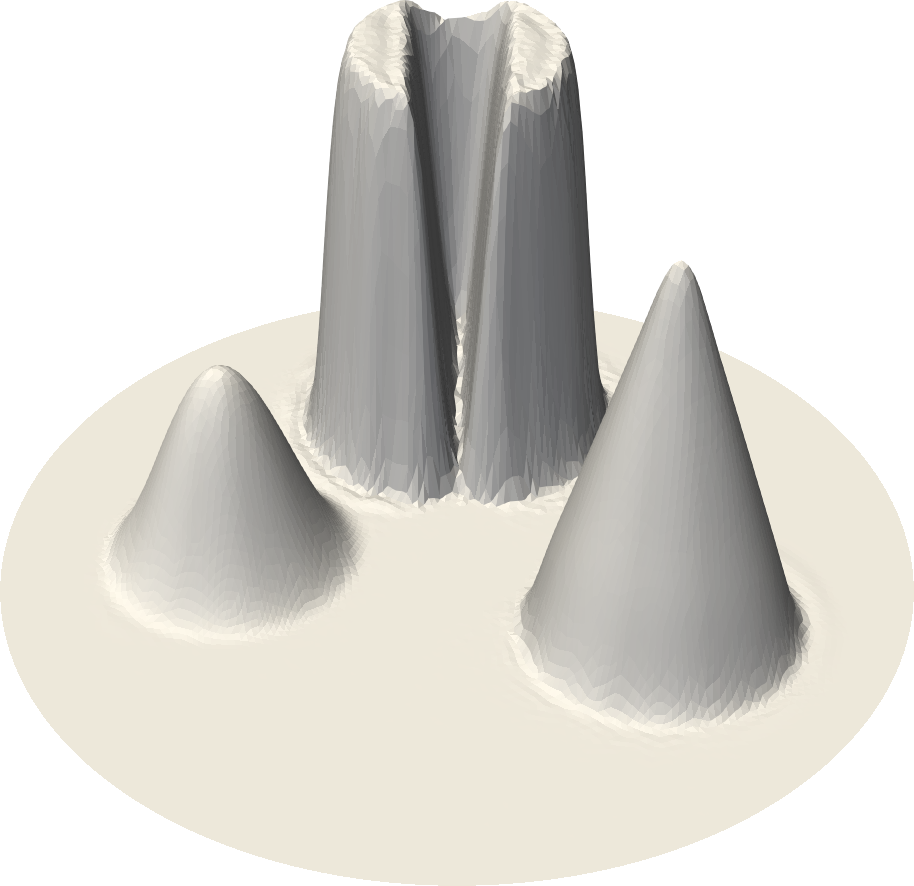}
    &
    \includegraphics[width=0.22\textwidth]{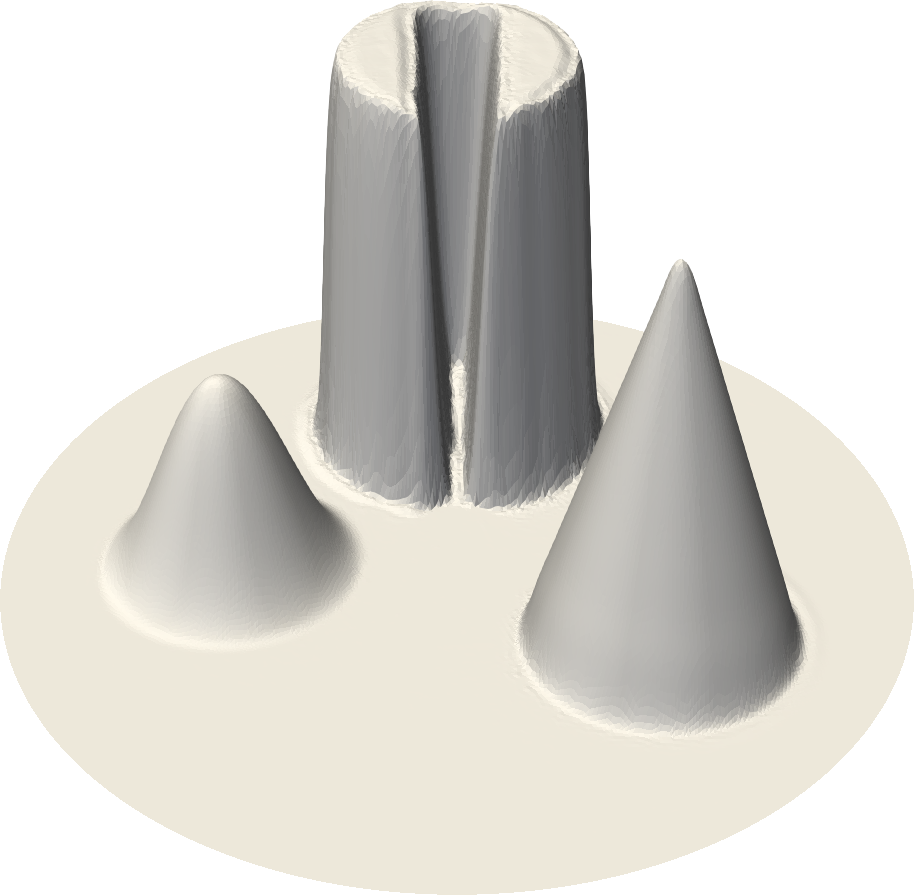}
    &
    \includegraphics[width=0.22\textwidth]{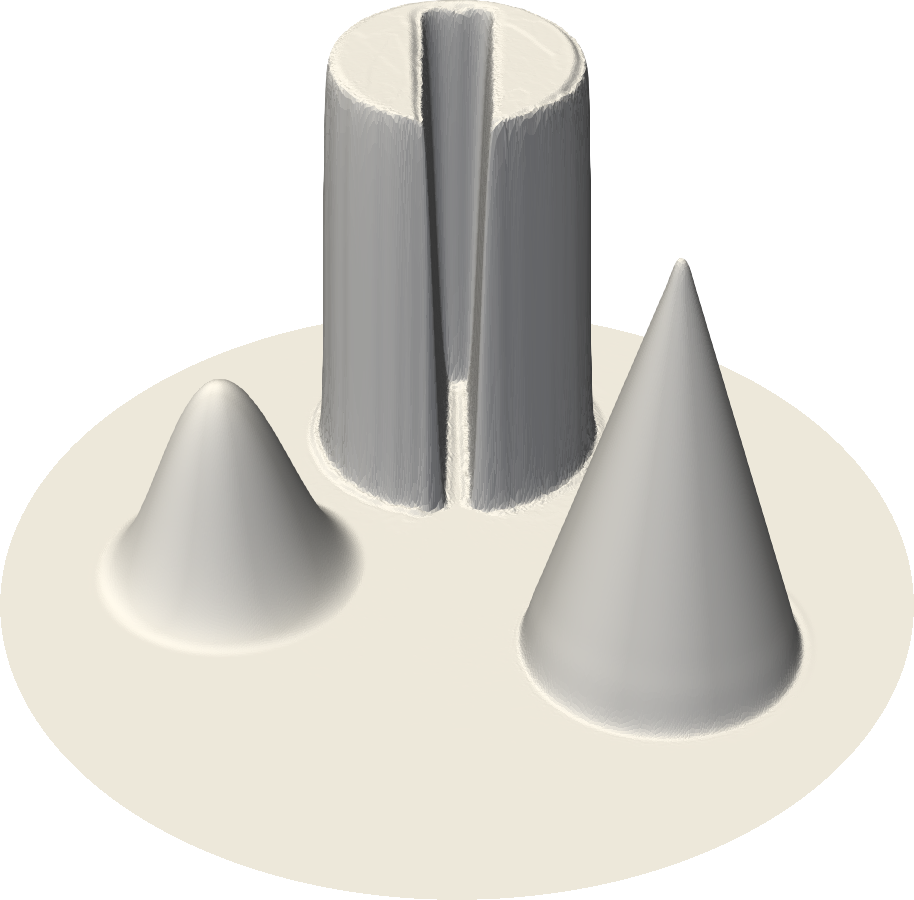}
    \\
    & 5045 & 16039 & 50177 & 159841 \\
    \\[2pt]
    \midrule
   {\rotatebox{90}{\hspace{-0.3in} dofs \hspace{0.35in} $\polP_3$}}
    &
      \hspace{0.1in}
    \includegraphics[width=0.22\textwidth]{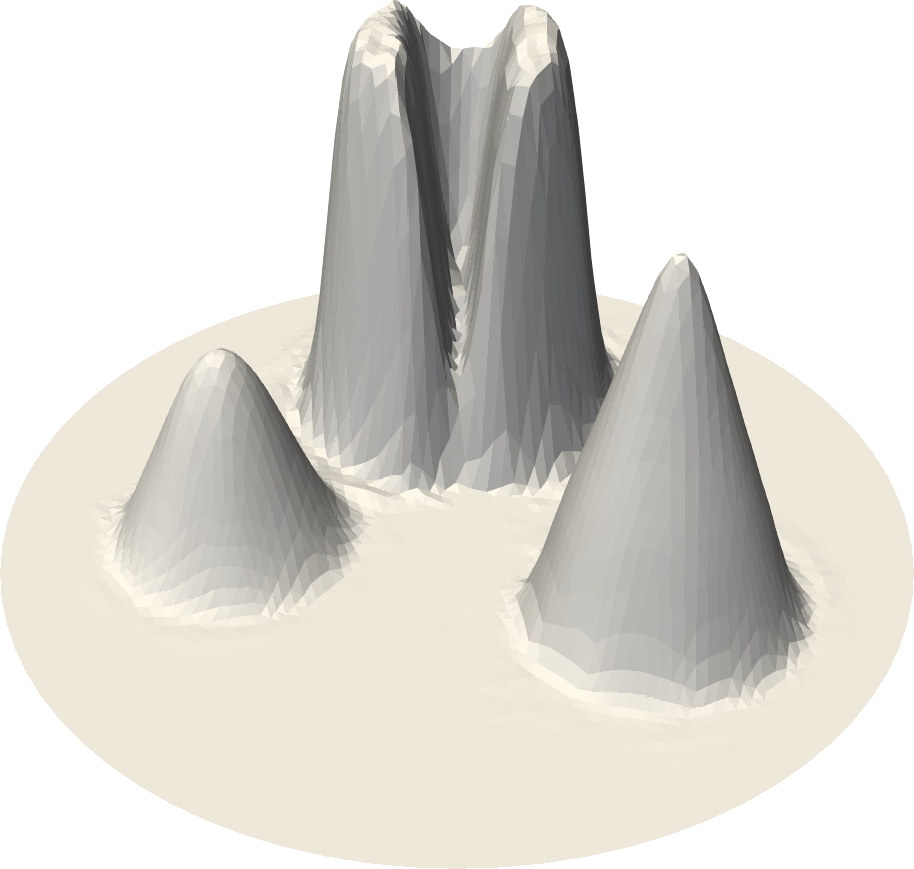}
    &
    \includegraphics[width=0.22\textwidth]{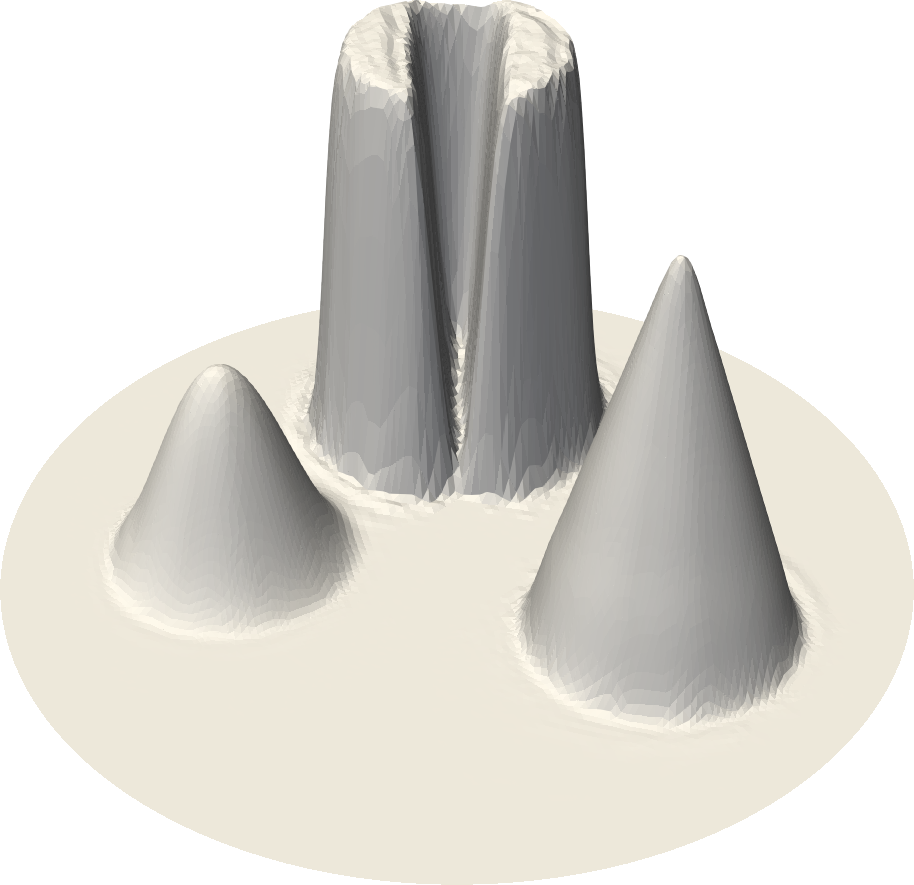}
    &
    \includegraphics[width=0.22\textwidth]{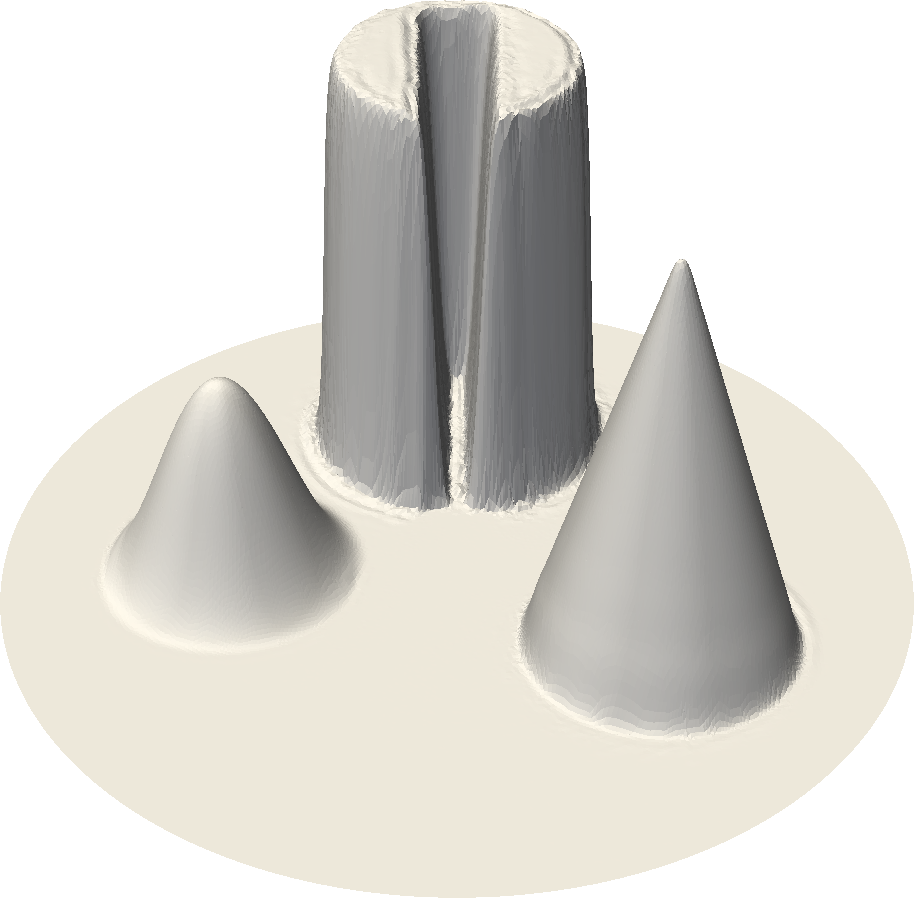}
    &
    \includegraphics[width=0.22\textwidth]{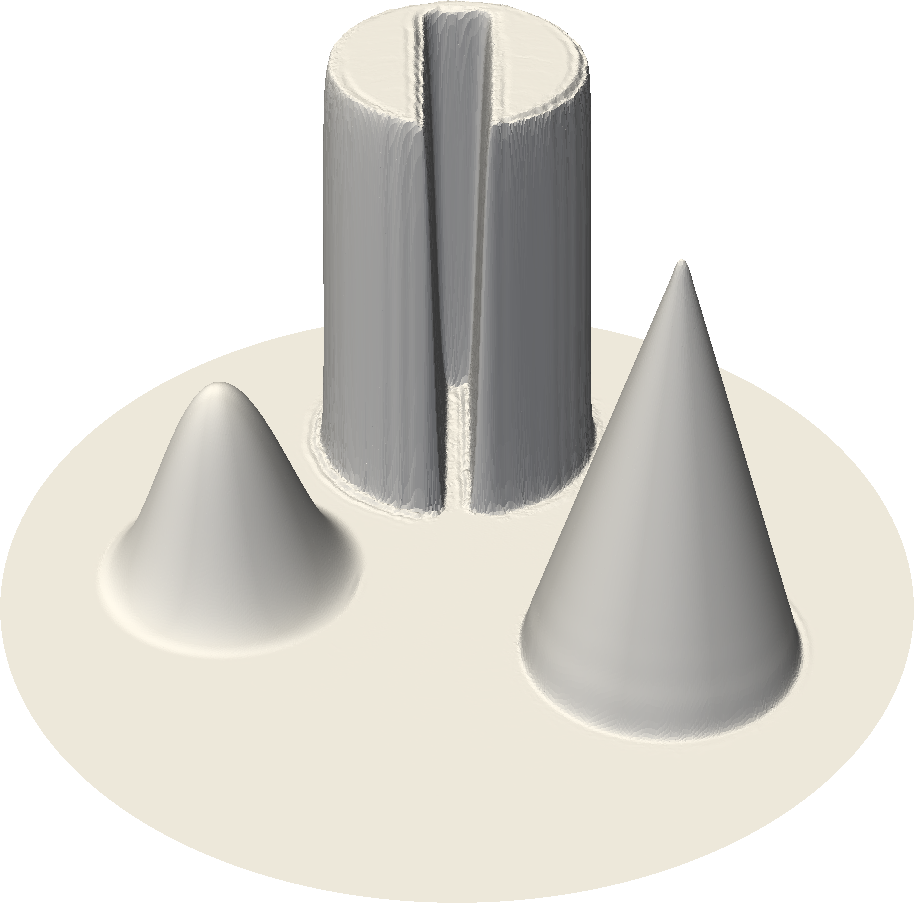}
    \\
    & 5179 & 16021 & 50341 & 159919 \\
    \\[2pt]
    \midrule
   {\rotatebox{90}{\hspace{-0.3in} dofs \hspace{0.35in} $\polP_4$}}
    &
      \hspace{0.1in}
    \includegraphics[width=0.22\textwidth]{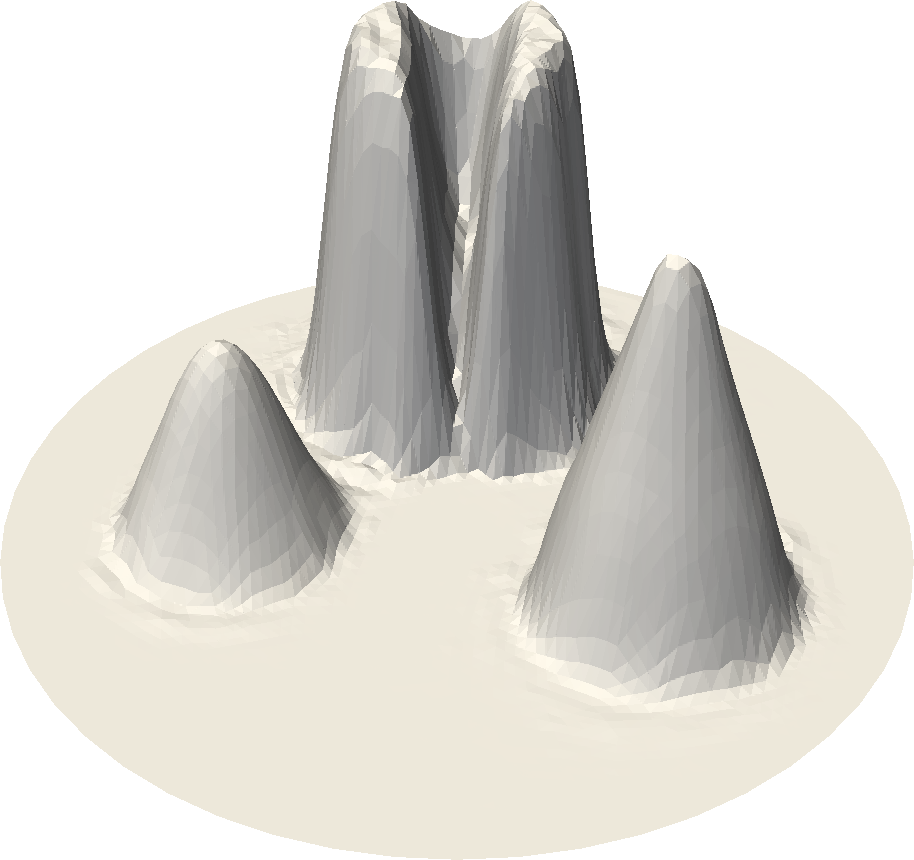}
    &
    \includegraphics[width=0.22\textwidth]{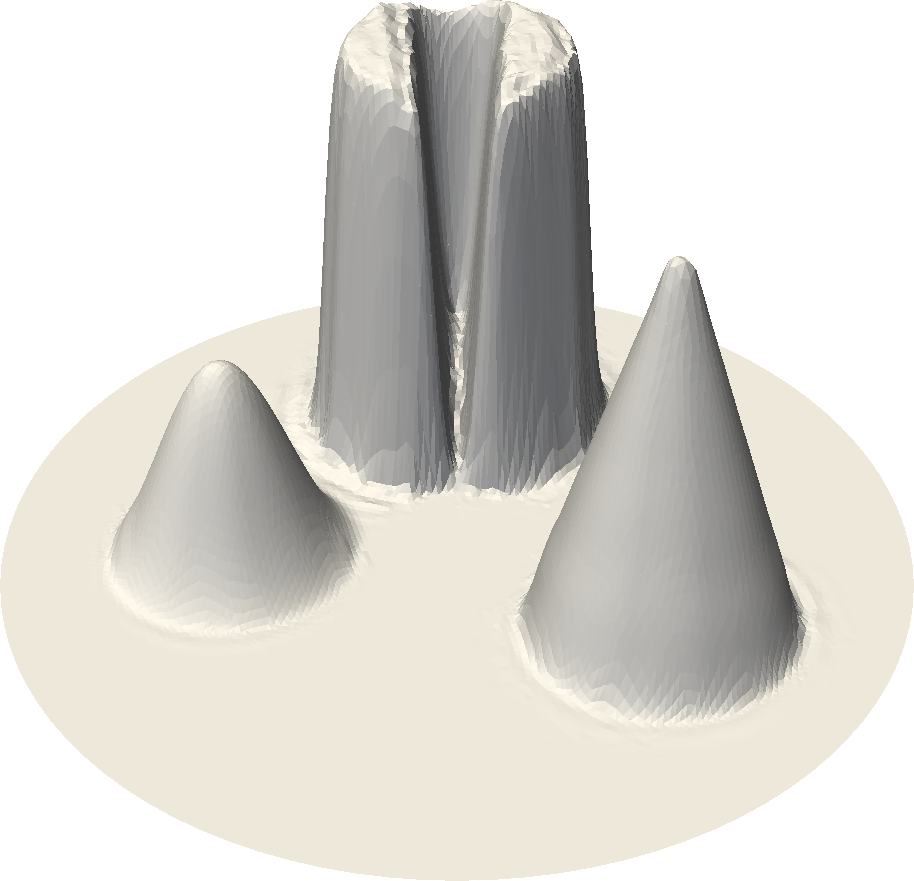}
    &
    \includegraphics[width=0.22\textwidth]{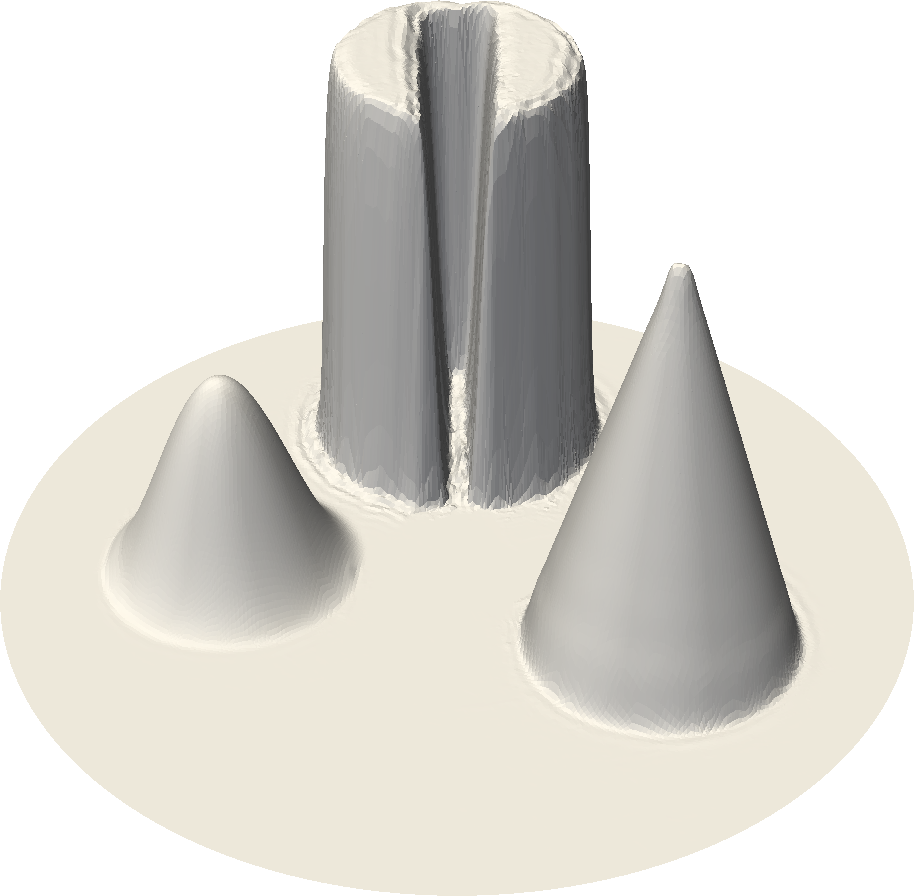}
    &
    \includegraphics[width=0.22\textwidth]{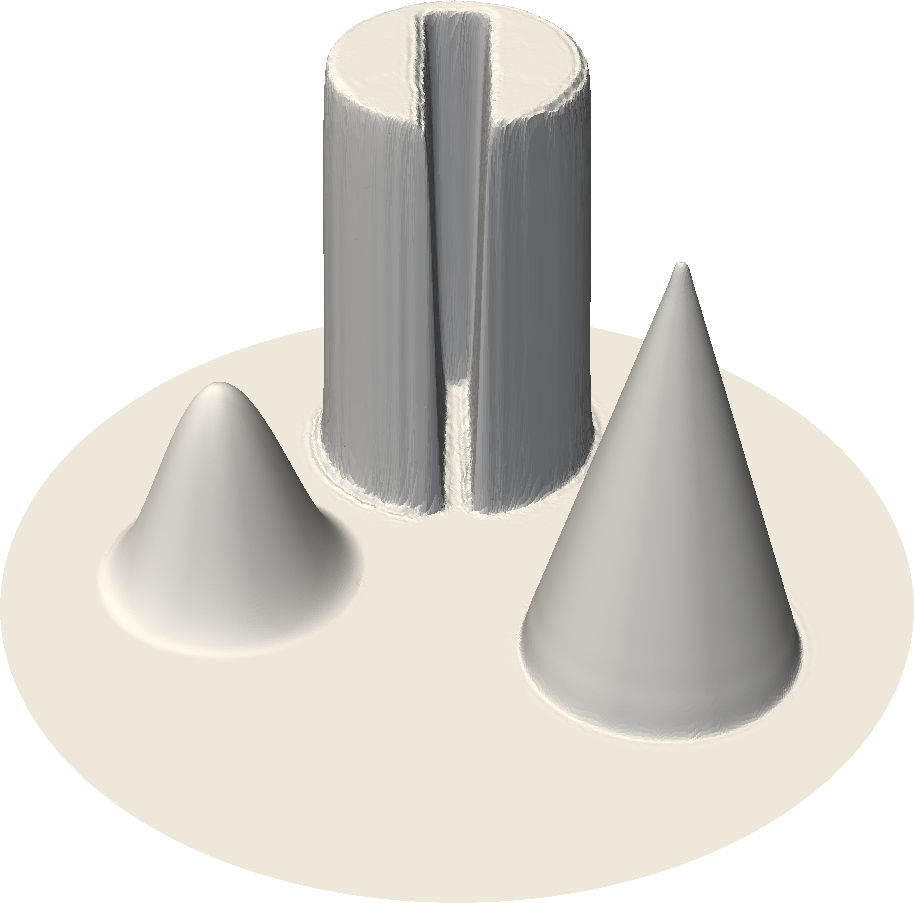}
    \\
    & 4969 & 15841 & 49573 & 160681 \\
    \\[2pt]
  \end{tabular}

  \caption{{Three body rotation on unstructured Delaunay triangulations of the unit disk: limited solutions at $t=1$ (one revolution) using polynomial degrees $\polP_1$--$\polP_4$; the meshes are chosen so that the numbers of degrees of freedom are comparable across the degrees.}}
  \label{fig:transport_unstr}
\end{figure}

\subsection{Burgers' equation}
In this and the next section, we apply the above algorithm to solve nonlinear problems. We consider Burgers' equation in the domain 
$\Omega = (-0.25, 1.75)^2$:
\begin{equation}\label{eq:burger}
  \partial_t u + \DIV \bef(u) =0, \quad \bef(u) = \frac12(u^2,u^2)^\top, 
\end{equation}
and 
\[
  u_0(\bx) = 
  \begin{cases}
    1, & \text{if } |x_1-\frac12|\le 1 \text{ and }  |x_2-\frac12| \le 1,\\
    -0.75, & \text{otherwise}.
  \end{cases}
\]
The exact solution for this problem is described in \citep[Sec.~6.1]{Guermond_Popov_2017}. This is a difficult task for numerical methods because the solution contains a sonic point, \ie there exists $u$ such that $\bef'(u) = 0$. The lack of viscosity due to the sonic point can lead to convergence to a non-entropy solution, see \eg \citep[Sec.~6]{Guermond_Popov_2017}. We solve the problem up to $t=0.75$ with ${\rm CFL}=0.27$ in \eqref{eq:dt_nu} on $\polP_1$, $\polP_2$, $\polP_3$ and $\polP_4$ polynomial spaces.

We solve the problem using the proposed limiting algorithm. For the high-order baseline, we compare two choices. In the first case, we use the residual-viscosity method \eqref{eq:HO} with $C_{\rm R}=4$ and $\gamma_{\rm cip}=1$. In the second case, we set $C_{\rm R}=0$ and $\gamma_{\rm cip}=1$ in \eqref{eq:HO}, so that the residual-viscosity term is removed and the high-order method is stabilized only by the CIP term. We refer to the latter method as CIP. 
The results from numerical simulations are collected in Figure~\ref{fig:burgers_comparison}. The first row shows the limited solutions for all polynomial degrees when the high-order method is stabilized by RV and CIP, and the second row when it is stabilized by CIP only. The choice of the high-order method is important for nonlinear problems. Both methods could solve the problem, with CIP producing slightly visible oscillations along the side discontinuities. 

Table~\ref{tab:burgers_l1_convergence} reports the convergence in the relative $L^1$-norm for all polynomial spaces. Both limited methods converge with rate close to one, which is the expected rate for a solution containing shocks. The CIP-only baseline gives a smaller error than RV+CIP, by about a factor of two for $\polP_1$ and $\polP_2$ and by a factor of about $1.3$--$1.5$ for $\polP_3$ and $\polP_4$ on the finest meshes. With $C_{\rm rel}=0.25$ instead of $C_{\rm rel}=1$ in \eqref{eq:ri}, the errors change by less than $1\%$ (not shown).

\begin{figure}[htbp]
  \centering
  \setlength{\tabcolsep}{2pt}
  \renewcommand{\arraystretch}{1.0}

  \newcommand{\rowlabel}[1]{%
    \raisebox{0pt}[\height][\depth]{%
      \makebox[0.035\textwidth][c]{\rotatebox{90}{#1}}%
    }%
  }

  \begin{tabular}{ccccc}
    \toprule
    & $\polP_1$ & $\polP_2$ & $\polP_3$ & $\polP_4$ \\
    \midrule

    \rotatebox{90}{\small \hspace{-0.3in} dofs \hspace{0.5in} RV}
    &
    \includegraphics[width=0.22\textwidth]{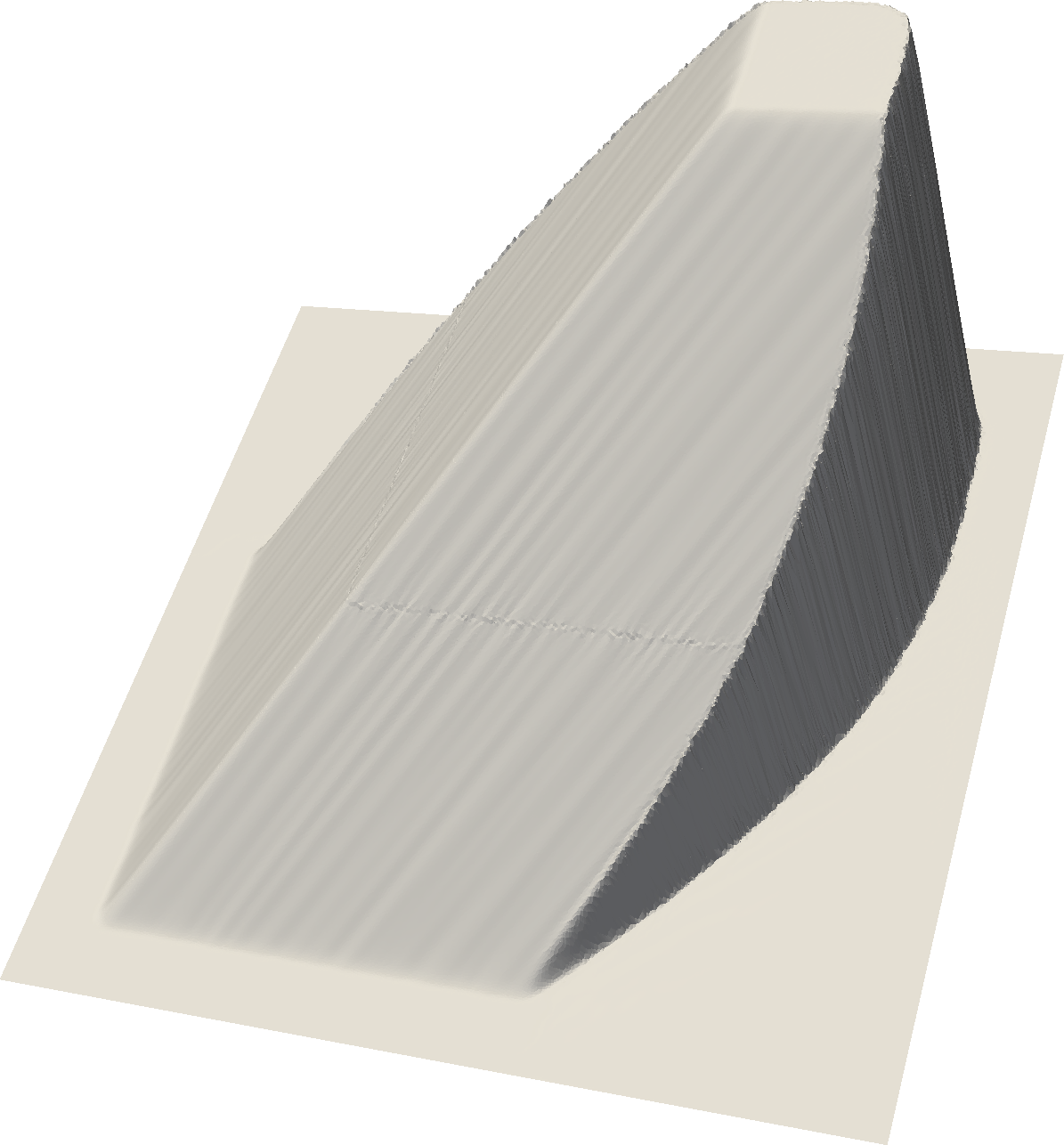}
    &
    \includegraphics[width=0.22\textwidth]{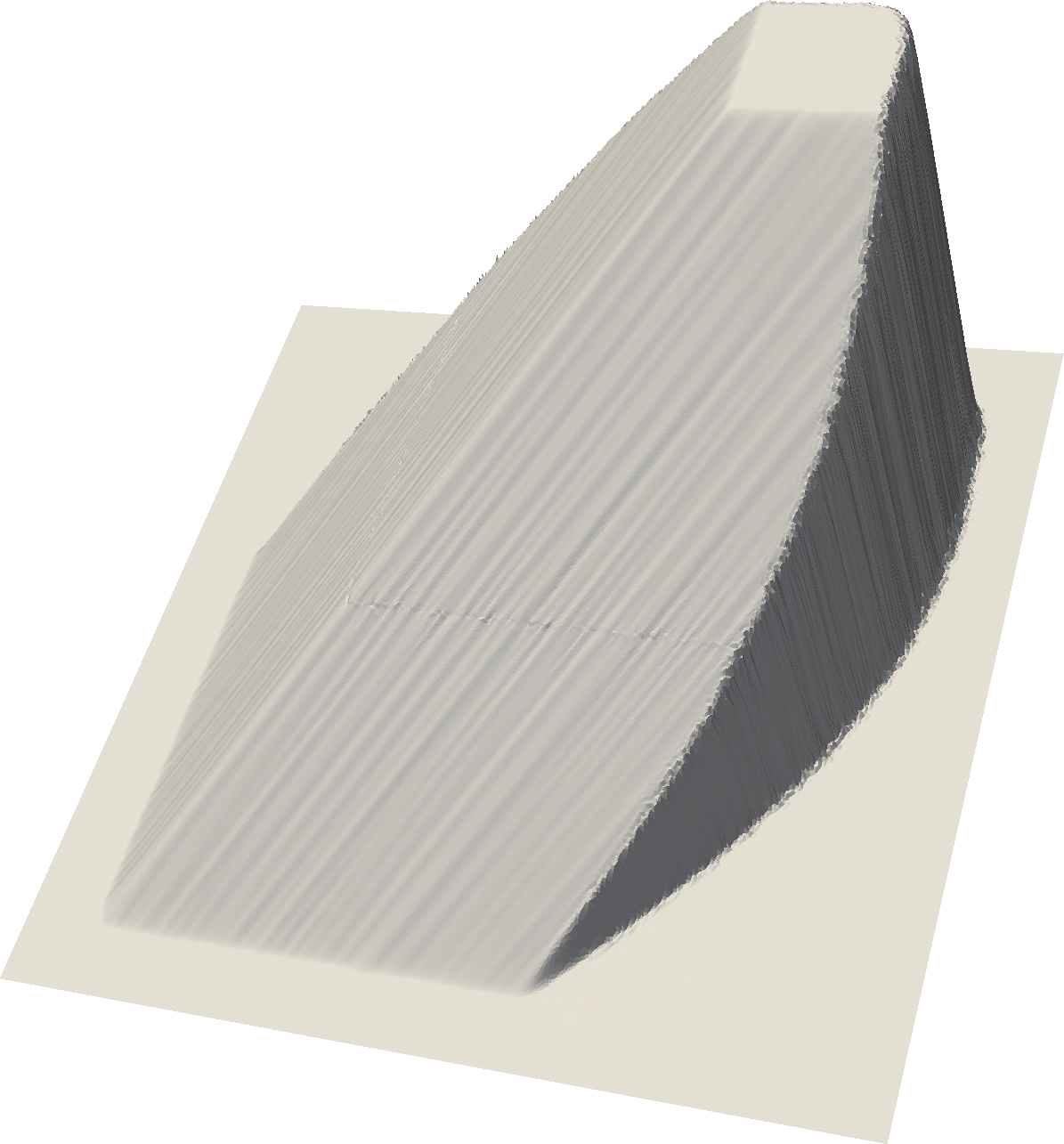}
    &
    \includegraphics[width=0.22\textwidth]{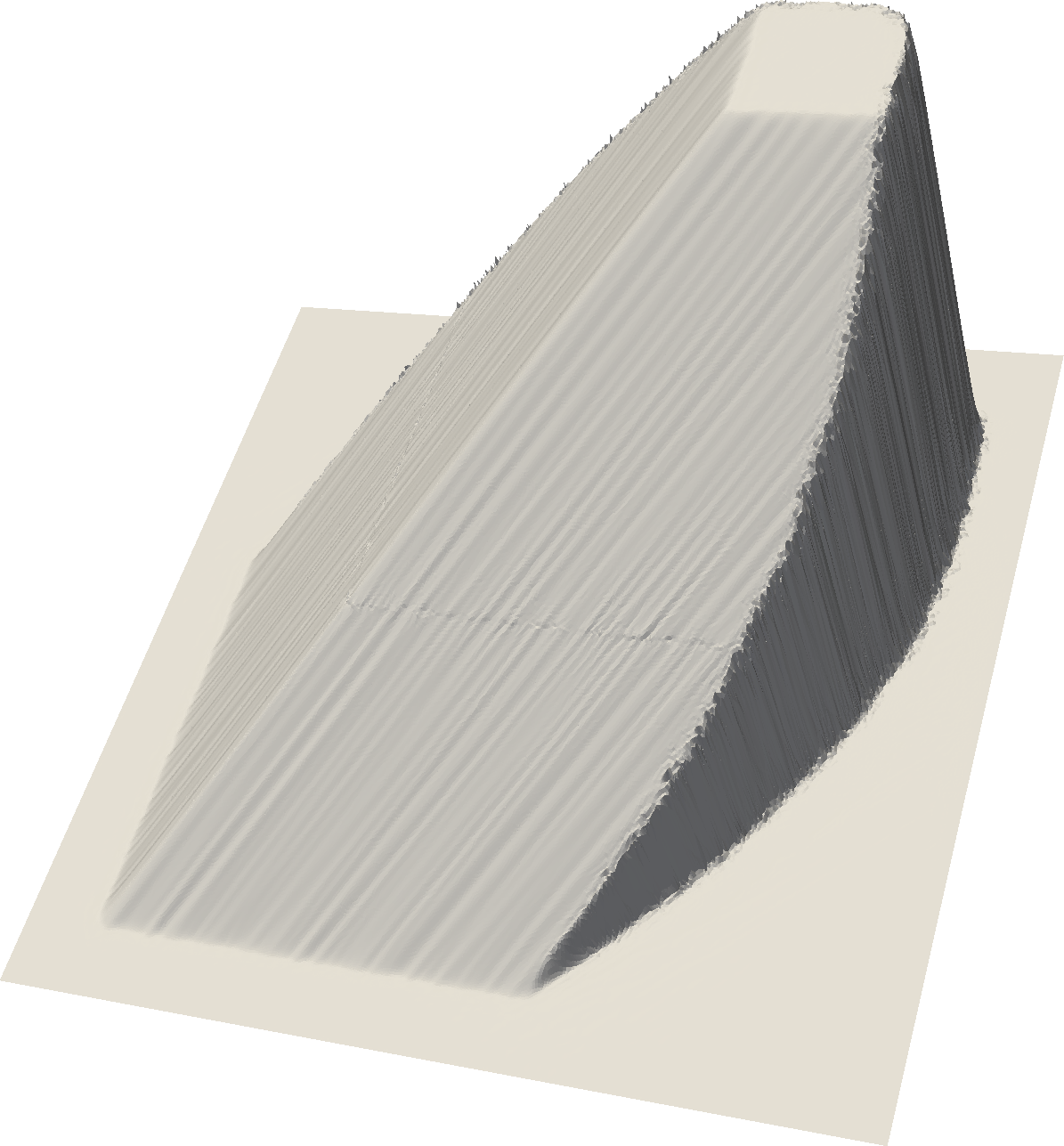}
    &
      \includegraphics[width=0.22\textwidth]{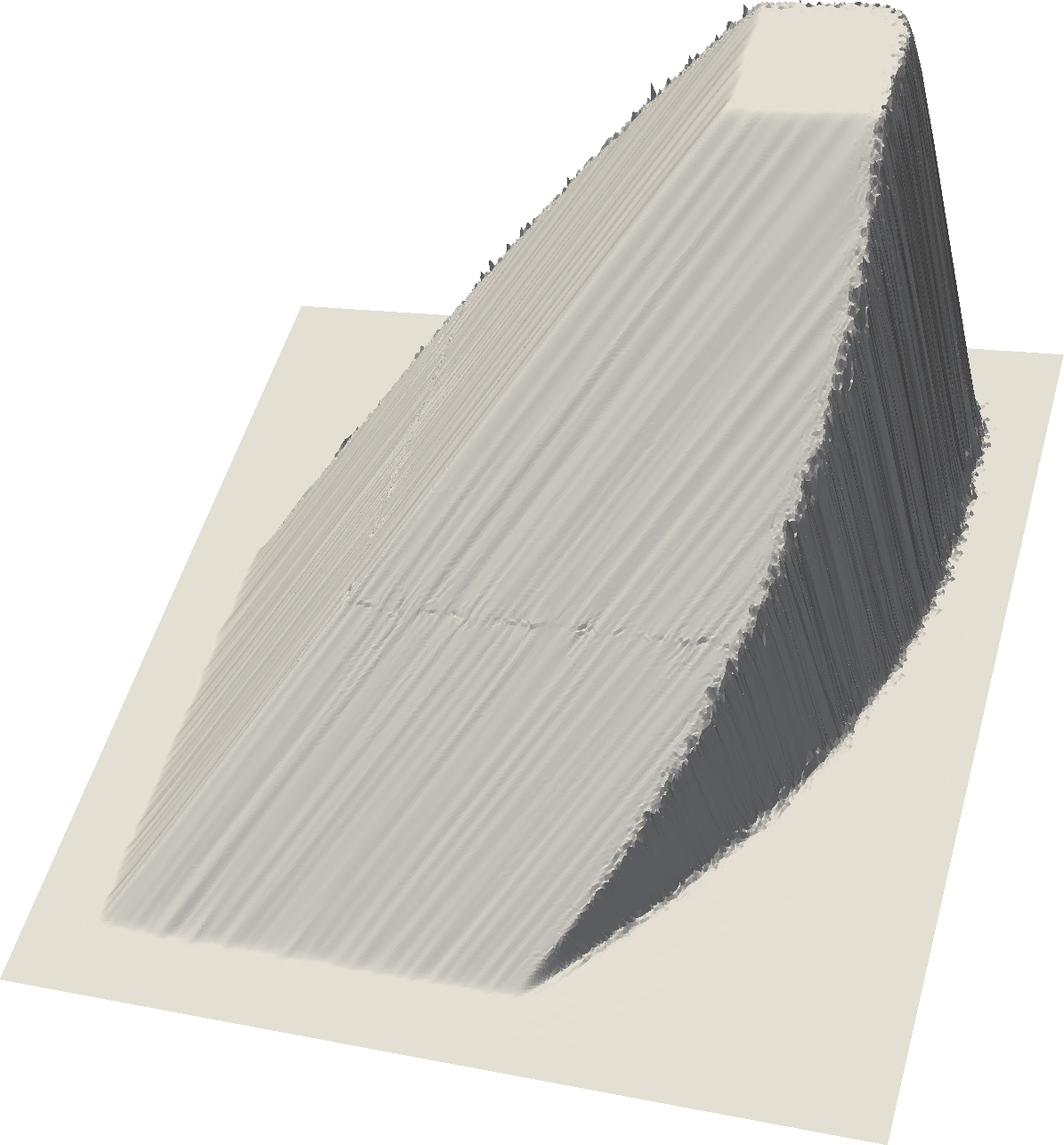}
    \\
    & 159560 & 158413 & 159709 & 161257 \\
    \\[2pt]

    \rotatebox{90}{\small \hspace{-0.3in} dofs \hspace{0.5in} CIP}
    &
    \includegraphics[width=0.22\textwidth]{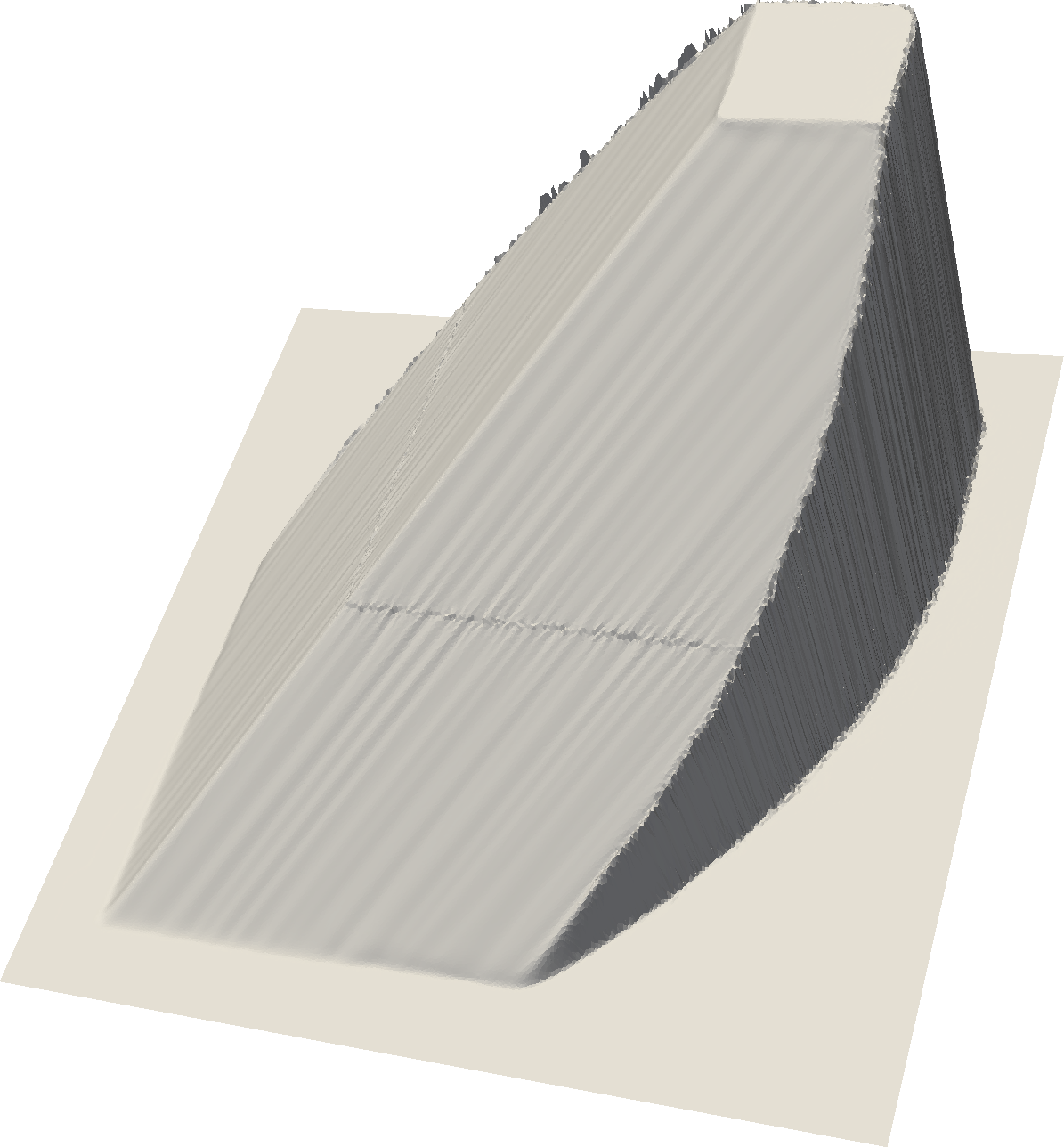}
    &
    \includegraphics[width=0.22\textwidth]{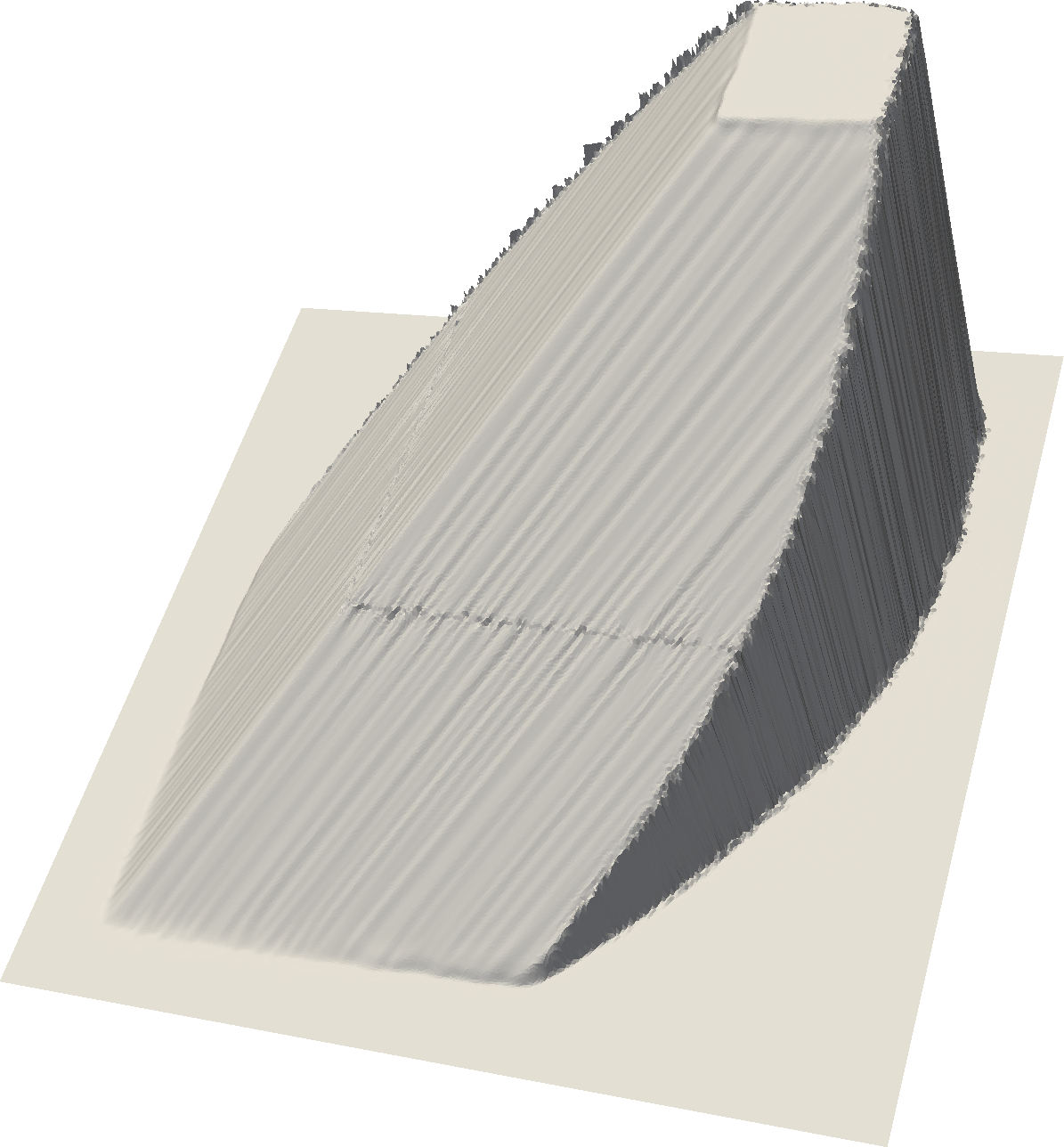}
    &   
    \includegraphics[width=0.22\textwidth]{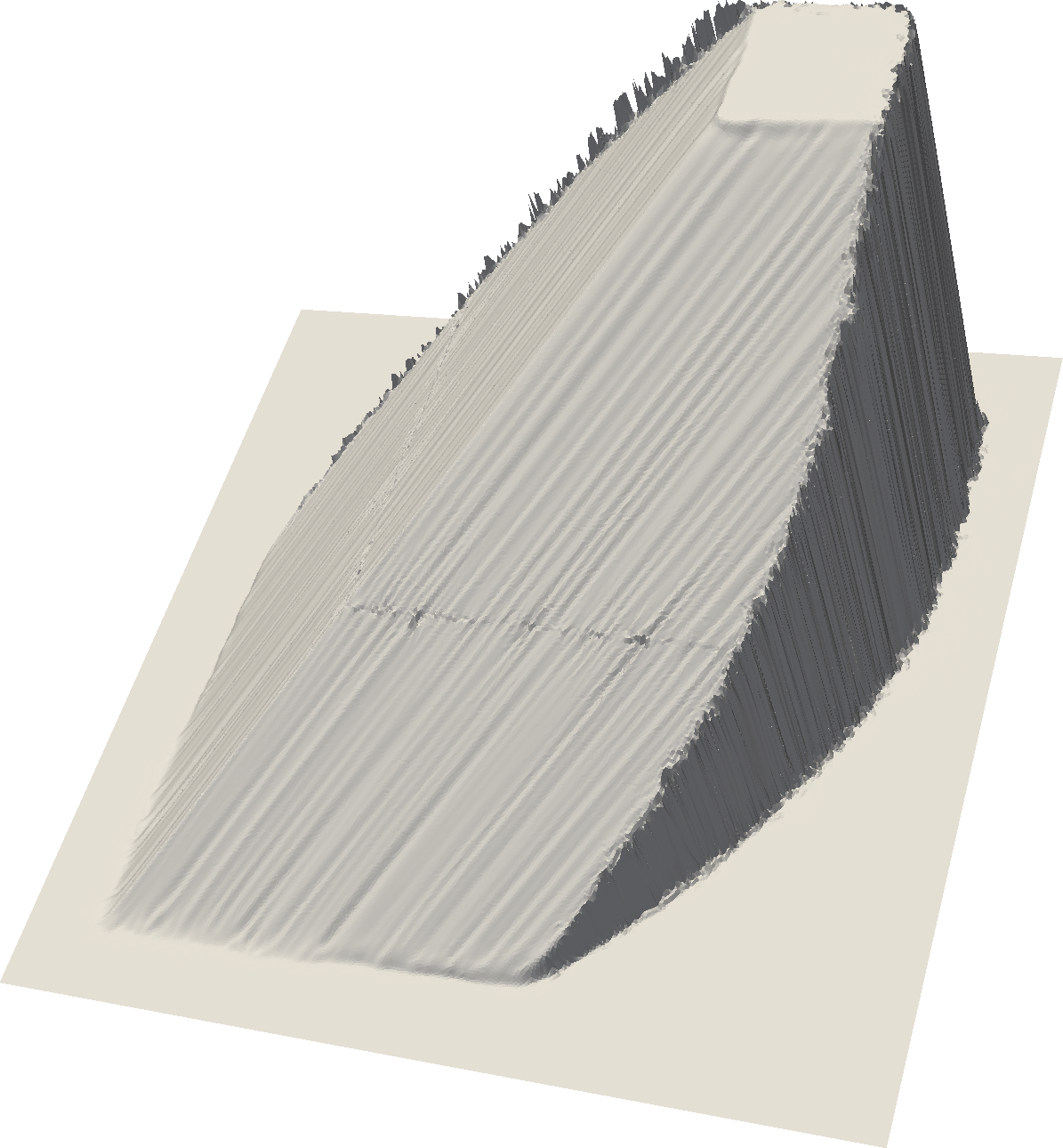}
    &
      \includegraphics[width=0.22\textwidth]{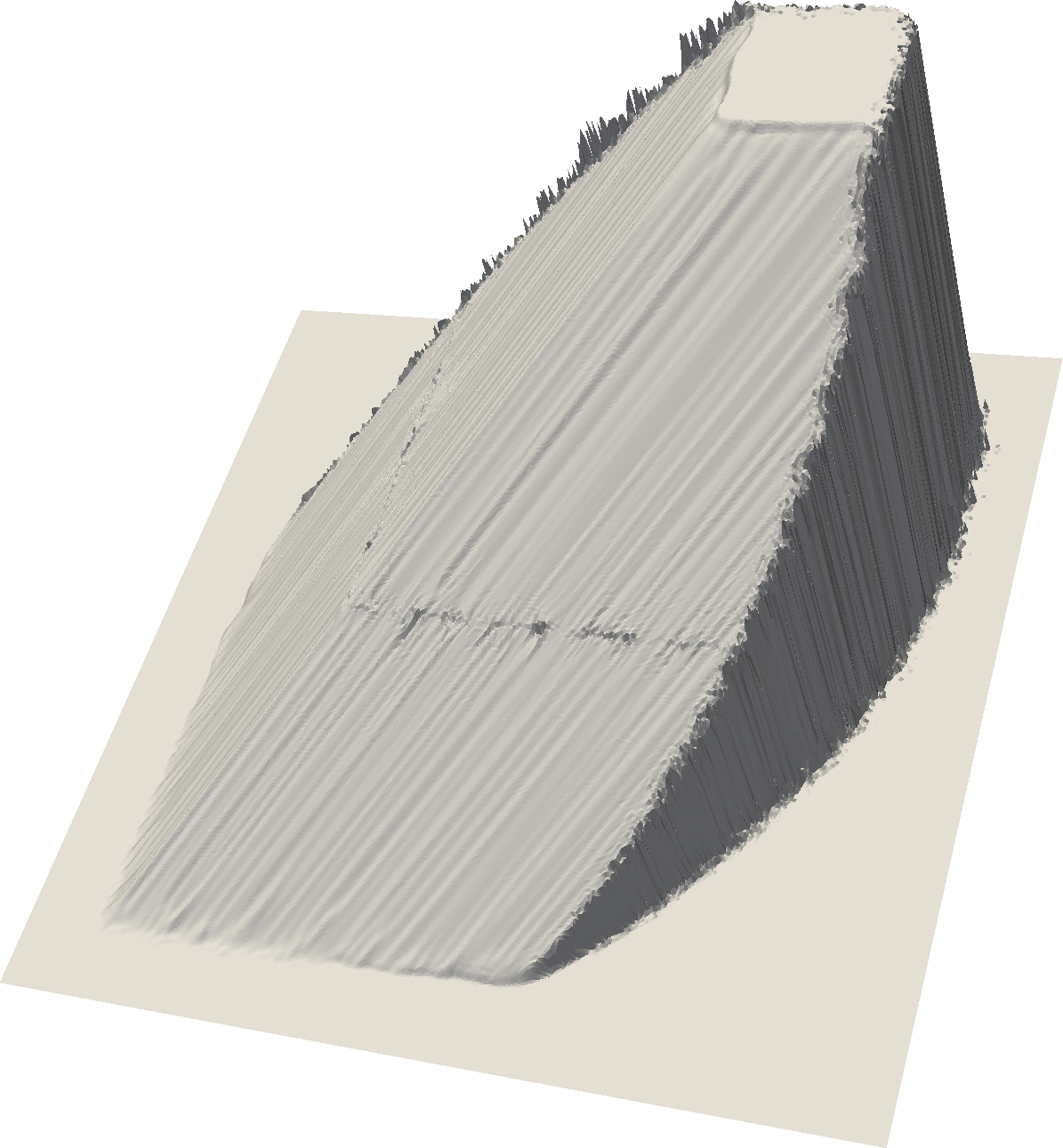}
    \\
    & 159560 & 158413 & 159709 & 161257 \\
    \\
    \bottomrule
  \end{tabular}

  \caption{Burgers' equation: comparison of limited solution using RV and CIP only using polynomial degrees $\polP_1$--$\polP_4$, $t=0.75$.}
  \label{fig:burgers_comparison}
\end{figure}

\begin{table}[htbp]
  \centering
  \small
  \caption{{Burgers' equation on unstructured meshes: relative $L^1$ error at $t=0.75$ and observed rate $p$ (with respect to $\mathrm{dofs}^{-1/2}$) for the graph-Poisson limiter with the RV+CIP and the CIP-only high-order baselines.}}
  \label{tab:burgers_l1_convergence}
  \vspace{0.05in}
      \begin{tabular}{rcccc@{\hspace{14pt}}rcccc}
      \toprule
      \multicolumn{5}{c}{$\polP_1$} & \multicolumn{5}{c}{$\polP_2$} \\
      \cmidrule(r){1-5}\cmidrule(l){6-10}
      \multirow{2}{*}{dofs} & \multicolumn{2}{c}{RV+CIP+lim} & \multicolumn{2}{c}{CIP+lim}
      & \multirow{2}{*}{dofs} & \multicolumn{2}{c}{RV+CIP+lim} & \multicolumn{2}{c}{CIP+lim} \\
      \cmidrule(lr){2-5}\cmidrule(lr){7-10}
      & $L^1$ & $p$ & $L^1$ & $p$ & & $L^1$ & $p$ & $L^1$ & $p$ \\
      \midrule
      1569 & $1.17\times 10^{-1}$ & -- & $5.54\times 10^{-2}$ & --
      & 1637 & $1.11\times 10^{-1}$ & -- & $6.18\times 10^{-2}$ & -- \\
      4998 & $6.69\times 10^{-2}$ & 0.97 & $3.26\times 10^{-2}$ & 0.92
      & 5153 & $6.65\times 10^{-2}$ & 0.90 & $3.87\times 10^{-2}$ & 0.82 \\
      16135 & $3.98\times 10^{-2}$ & 0.88 & $1.90\times 10^{-2}$ & 0.92
      & 15685 & $3.78\times 10^{-2}$ & 1.01 & $2.17\times 10^{-2}$ & 1.04 \\
      50073 & $2.34\times 10^{-2}$ & 0.94 & $1.08\times 10^{-2}$ & 0.99
      & 50177 & $2.22\times 10^{-2}$ & 0.92 & $1.26\times 10^{-2}$ & 0.93 \\
      159560 & $1.32\times 10^{-2}$ & 0.99 & $6.11\times 10^{-3}$ & 0.99
      & 158413 & $1.28\times 10^{-2}$ & 0.96 & $7.22\times 10^{-3}$ & 0.97 \\
      500766 & $7.49\times 10^{-3}$ & 0.99 & $3.51\times 10^{-3}$ & 0.97
      & 502513 & $7.19\times 10^{-3}$ & 0.99 & $3.99\times 10^{-3}$ & 1.03 \\
      \bottomrule
    \end{tabular}

  \vspace{0.15in}

      \begin{tabular}{rcccc@{\hspace{14pt}}rcccc}
      \toprule
      \multicolumn{5}{c}{$\polP_3$} & \multicolumn{5}{c}{$\polP_4$} \\
      \cmidrule(r){1-5}\cmidrule(l){6-10}
      \multirow{2}{*}{dofs} & \multicolumn{2}{c}{RV+CIP+lim} & \multicolumn{2}{c}{CIP+lim}
      & \multirow{2}{*}{dofs} & \multicolumn{2}{c}{RV+CIP+lim} & \multicolumn{2}{c}{CIP+lim} \\
      \cmidrule(lr){2-5}\cmidrule(lr){7-10}
      & $L^1$ & $p$ & $L^1$ & $p$ & & $L^1$ & $p$ & $L^1$ & $p$ \\
      \midrule
      1765 & $1.21\times 10^{-1}$ & -- & $8.29\times 10^{-2}$ & --
      & 1537 & $1.48\times 10^{-1}$ & -- & $1.26\times 10^{-1}$ & -- \\
      5212 & $7.36\times 10^{-2}$ & 0.92 & $6.02\times 10^{-2}$ & 0.59
      & 4761 & $8.37\times 10^{-2}$ & 1.01 & $6.54\times 10^{-2}$ & 1.16 \\
      16468 & $3.93\times 10^{-2}$ & 1.09 & $2.78\times 10^{-2}$ & 1.34
      & 16801 & $4.48\times 10^{-2}$ & 0.99 & $3.45\times 10^{-2}$ & 1.02 \\
      50197 & $2.31\times 10^{-2}$ & 0.95 & $1.66\times 10^{-2}$ & 0.92
      & 49745 & $2.59\times 10^{-2}$ & 1.01 & $1.93\times 10^{-2}$ & 1.07 \\
      159709 & $1.30\times 10^{-2}$ & 0.99 & $8.88\times 10^{-3}$ & 1.08
      & 161257 & $1.48\times 10^{-2}$ & 0.96 & $1.15\times 10^{-2}$ & 0.87 \\
      499930 & $7.62\times 10^{-3}$ & 0.94 & $5.24\times 10^{-3}$ & 0.93
      & 497737 & $8.69\times 10^{-3}$ & 0.94 & $6.69\times 10^{-3}$ & 0.97 \\
      \bottomrule
    \end{tabular}
\end{table}

\subsection{KPP problem}

In our final numerical validation, we consider the so-called KPP problem of Kurganov, Petrova, and Popov \citep{Kurganov_2007}. We solve
\begin{equation}\label{eq:kpp}
  \partial_t u + \DIV \bef(u) =0, \quad 
  \bef(u)=(\sin u,\cos u)^\top,
\end{equation}
on the square domain
\[
  \Omega=[-2,2]\times[-2.5,1.5],
\]
with the initial condition
\[
  u_0(\bx)=
  \begin{cases}
    3.5\pi, & \text{if } x^2+y^2<1,\\
    0.25\pi, & \text{otherwise}.
  \end{cases}
\]
The solution of this problem contains composite shock waves. This benchmark is known to be challenging because some high-order methods may converge to a non-entropic weak solution. For examples of such behavior, see \citep{Kurganov_2007}, where several finite volume methods are discussed, and \citep{Striernstrom_2021}, where high-order finite difference methods are tested.

We again run the algorithm with RV plus CIP terms, \ie $C_{\rm R}=10$ and $\gamma_{\rm cip}=1$, and with only CIP term, \ie $C_{\rm R}=0$ and $\gamma_{\rm cip}=1$. The results are shown in Figure~\ref{fig:kpp_comparison} for two mesh resolutions. In both cases, the time step is \eqref{eq:dt_nu} with ${\rm CFL}=0.18$.

The results show that the limited CIP method does not converge to the correct entropy solution for this benchmark, whereas the limited residual-viscosity method does. This demonstrates the importance of the residual-viscosity stabilization for selecting the entropy solution in the presence of composite shock waves. We stress that this is numerical evidence and not an entropy-convergence result: the limiter enforces bounds, not an entropy inequality, so the entropy behaviour is inherited from the high-order candidate. Each baseline has a limitation in our tests: with $C_{\rm R}=0.1$ the residual viscosity lowers the $\polP_4$ convergence rate of the smooth test on the finest meshes (Table~\ref{tab:dipole}), while linear CIP alone does not select the entropy solution of the KPP problem; deactivating CIP in shock regions, weighted CIP \citep{Ern_Guermond_2013}, or a convex blending of high- and low-order dissipation \citep{Kuzmin_Vedral_2023} are possible remedies, and the correction of this paper applies unchanged to any such candidate.

\begin{figure}[htbp]
  \centering
  \setlength{\tabcolsep}{2pt}
  \renewcommand{\arraystretch}{1.0}

  \newcommand{\rowlabel}[1]{%
    \raisebox{0pt}[\height][\depth]{%
      \makebox[0.035\textwidth][c]{\rotatebox{90}{#1}}%
    }%
  }

  \begin{tabular}{ccccc}
    \toprule
    & $\polP_1$ & $\polP_2$ & $\polP_3$ & $\polP_4$ \\
    \midrule

    \rotatebox{90}{\small \hspace{-0.3in} dofs \hspace{0.5in} RV}
    &
    \includegraphics[width=0.22\textwidth]{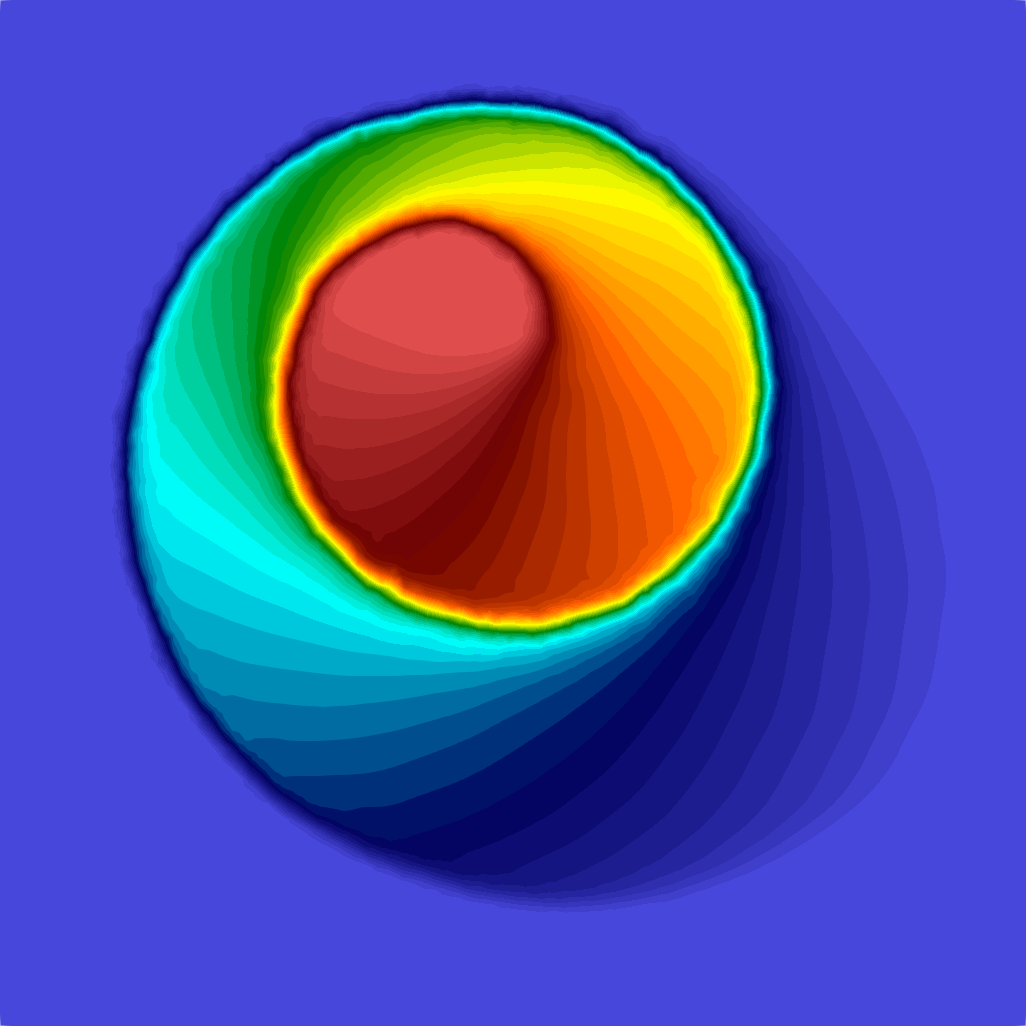}
    &
    \includegraphics[width=0.22\textwidth]{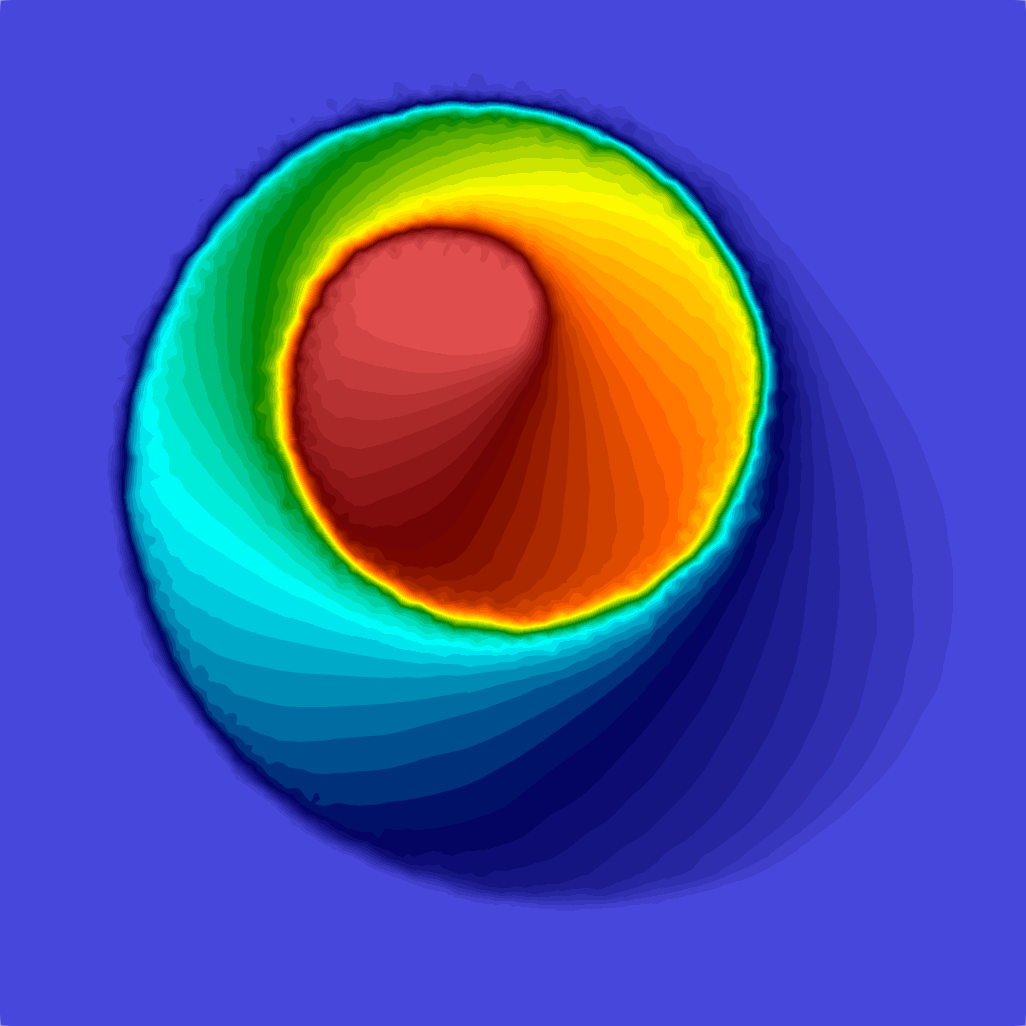}
    &
    \includegraphics[width=0.22\textwidth]{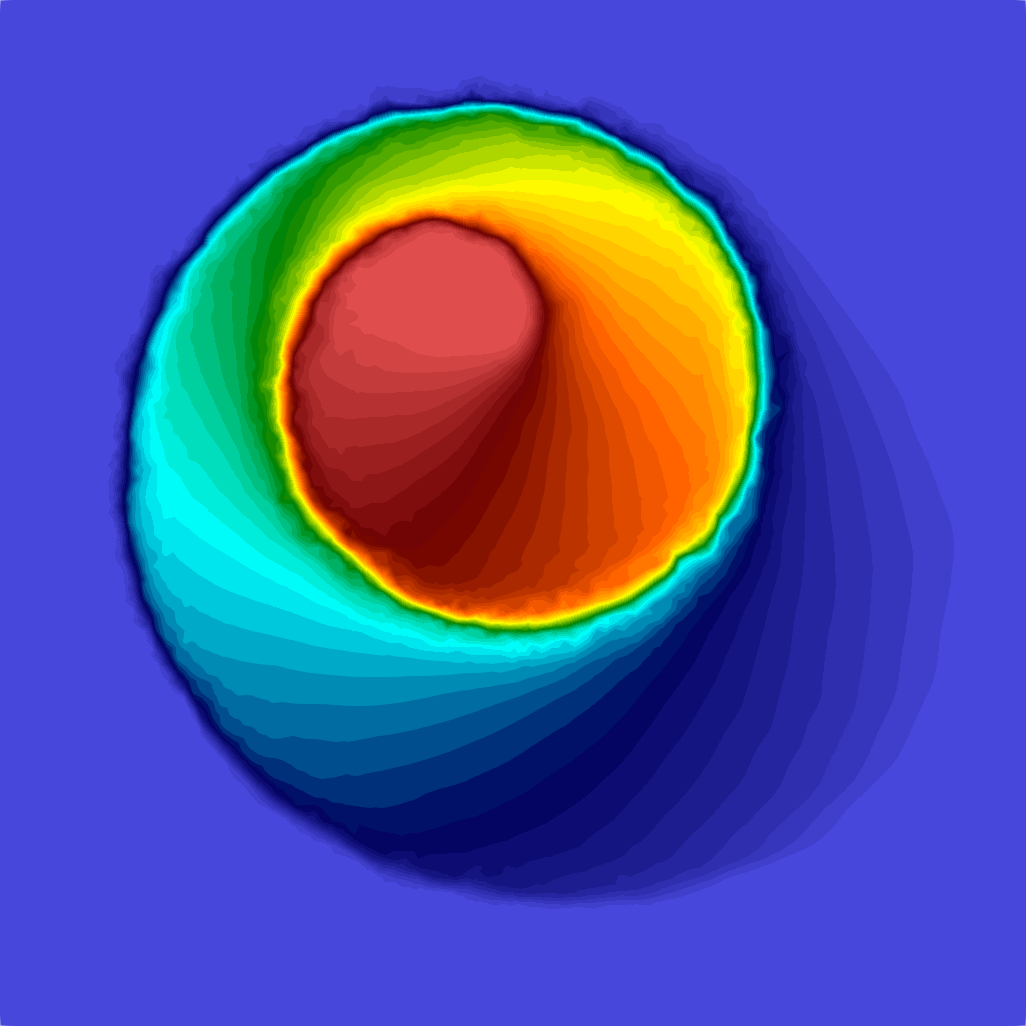}
    &
      \includegraphics[width=0.22\textwidth]{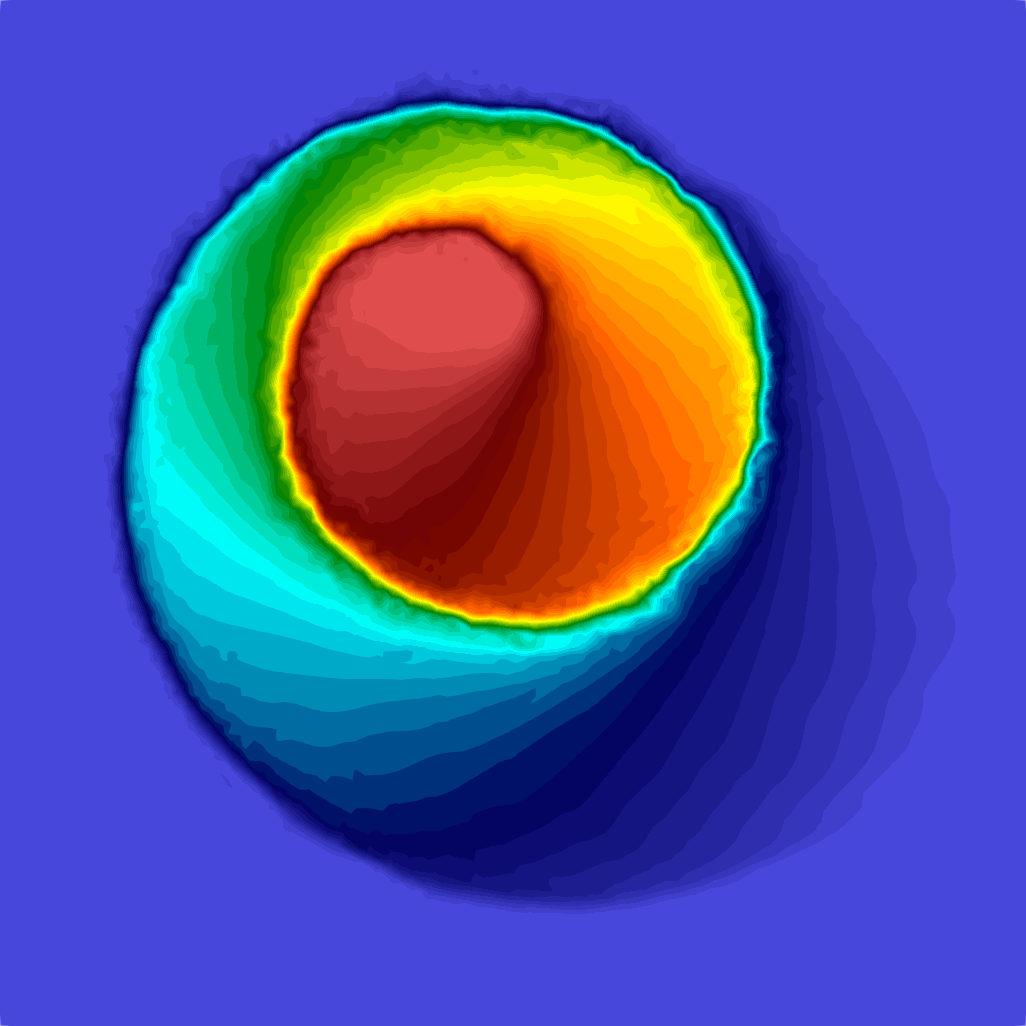}
    \\
    & 10042 & 9993 & 10048 & 10209
    \\[2pt]

    \rotatebox{90}{\small \hspace{-0.3in} dofs \hspace{0.5in} RV}
    &
    \includegraphics[width=0.22\textwidth]{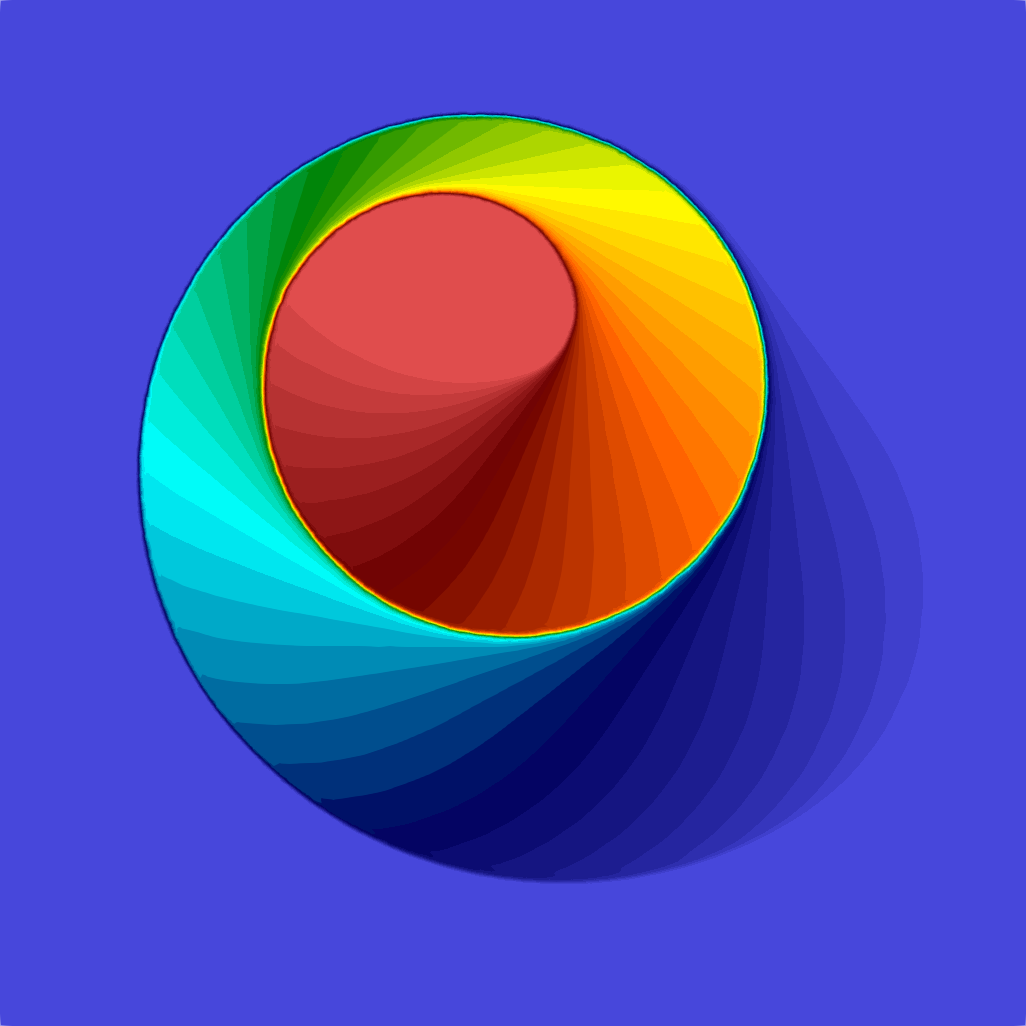}
    &
    \includegraphics[width=0.22\textwidth]{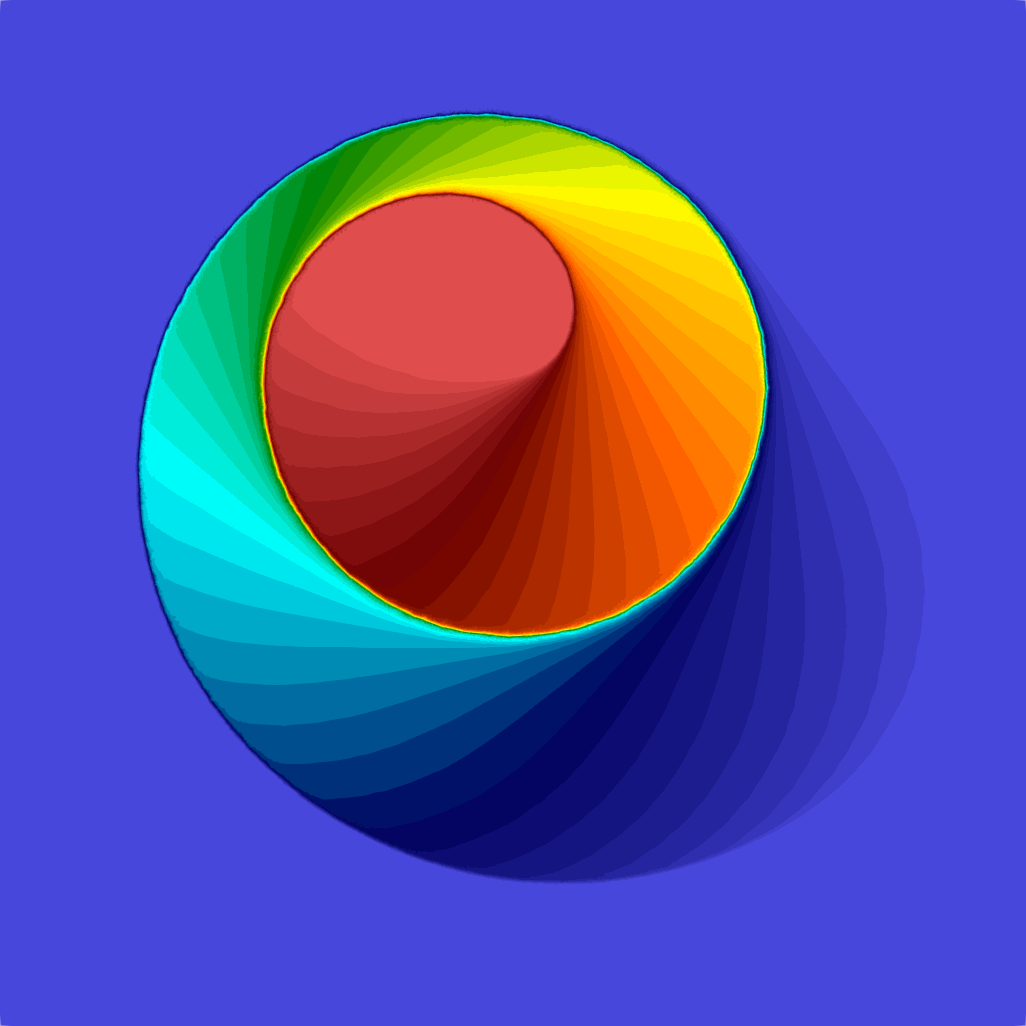}
    &
    \includegraphics[width=0.22\textwidth]{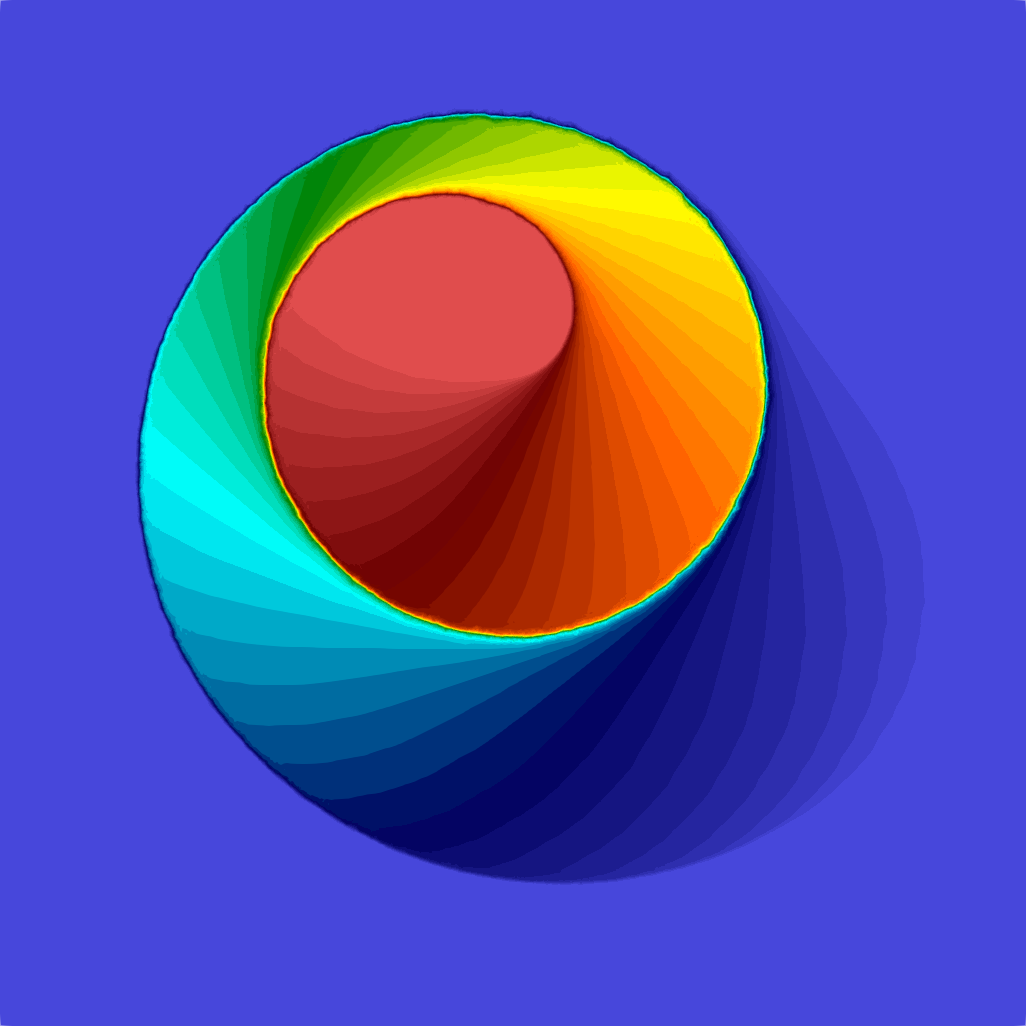}
    &
    \includegraphics[width=0.22\textwidth]{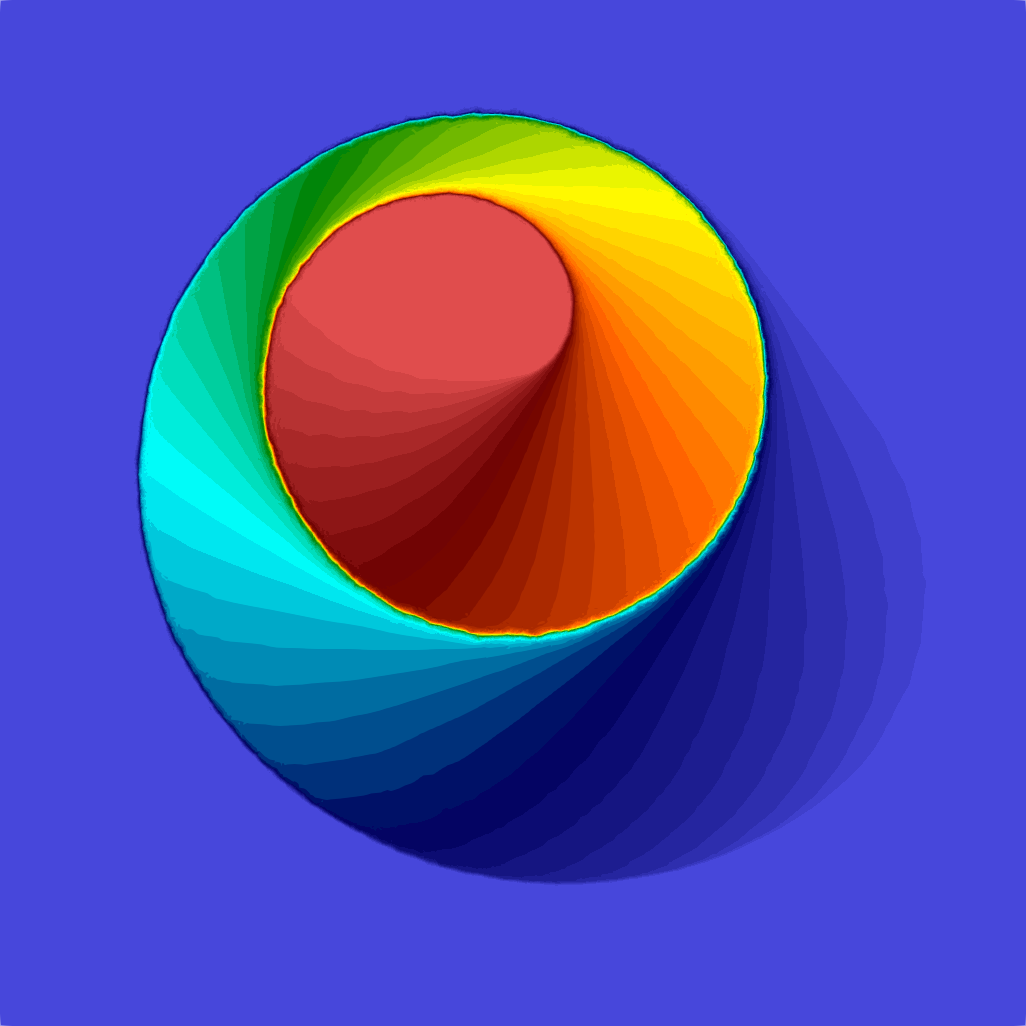}
    \\
    & 249773 & 249629 & 247699 & 251817
    \\[2pt]
    \midrule

    \rotatebox{90}{\small \hspace{-0.3in} dofs \hspace{0.5in} CIP}
    &
    \includegraphics[width=0.22\textwidth]{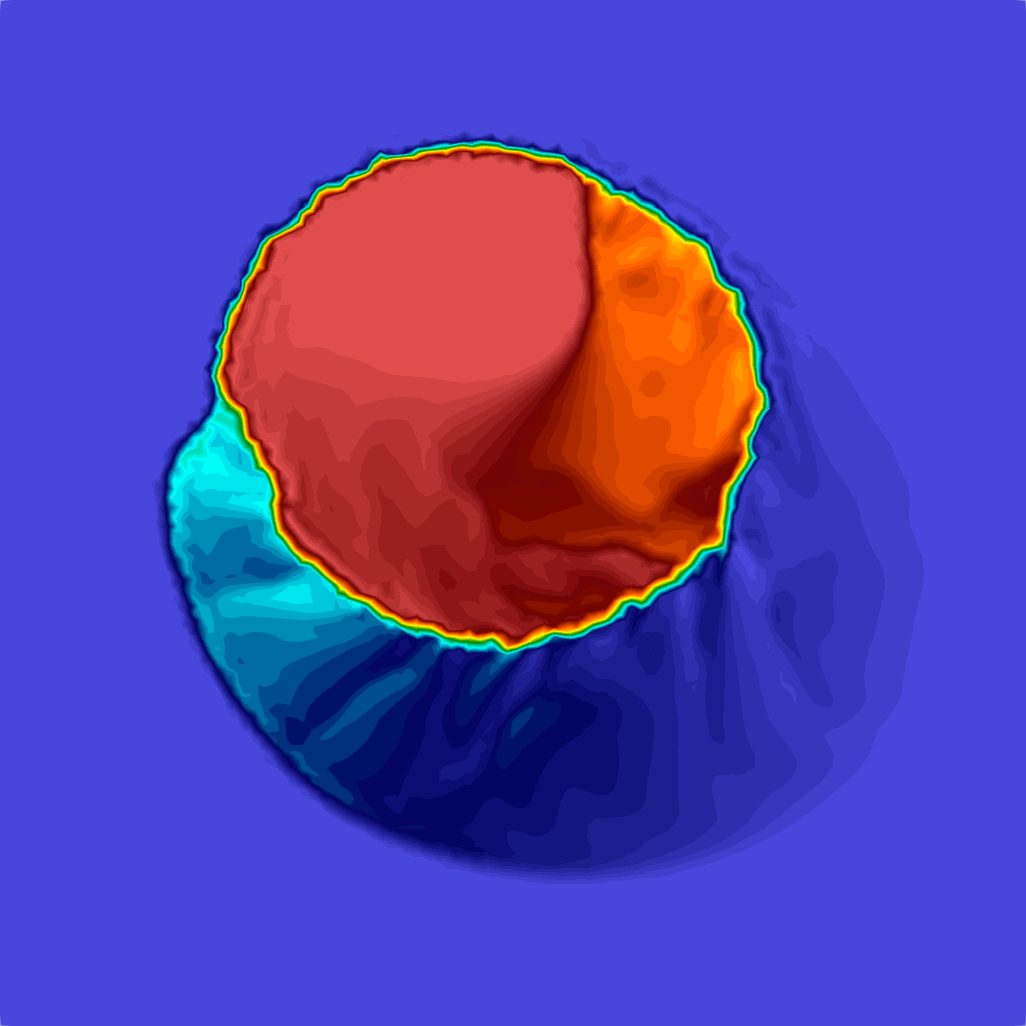}
    &
    \includegraphics[width=0.22\textwidth]{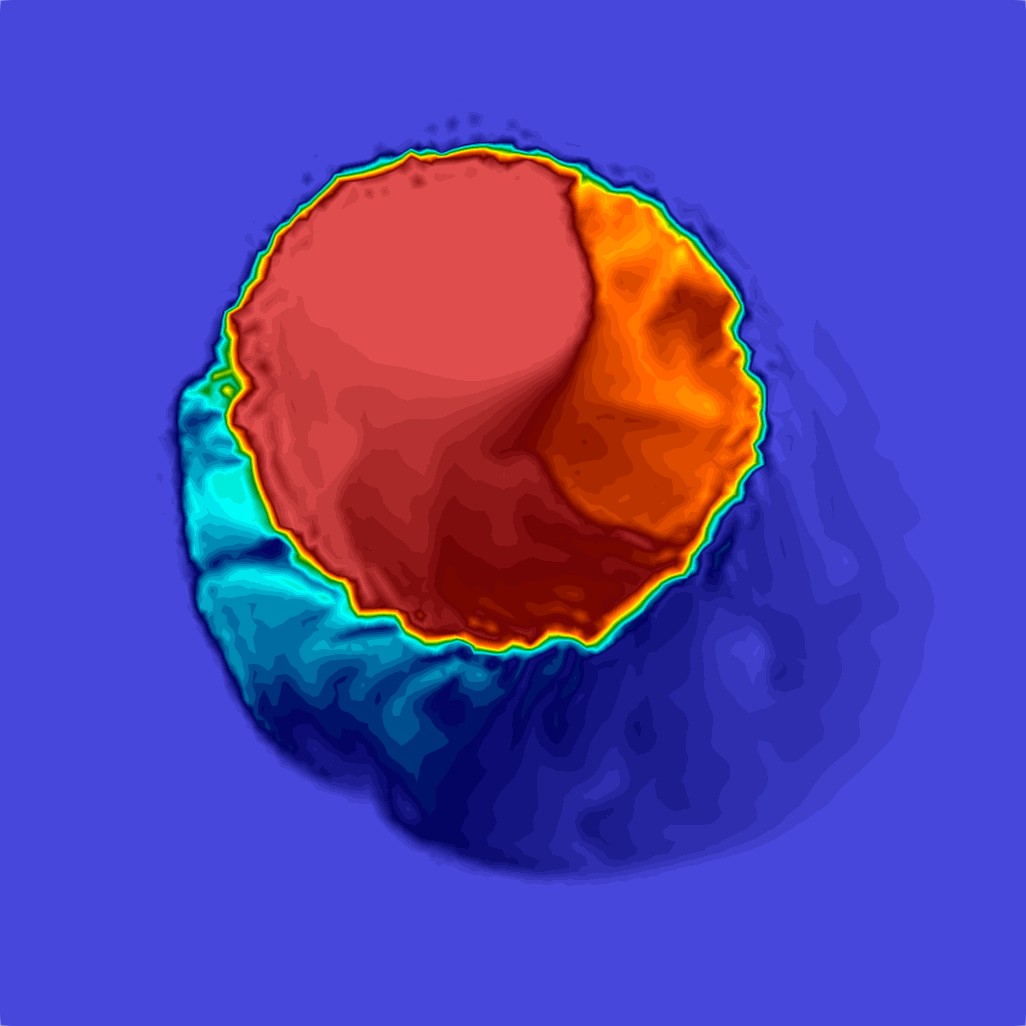}
    &
    \includegraphics[width=0.22\textwidth]{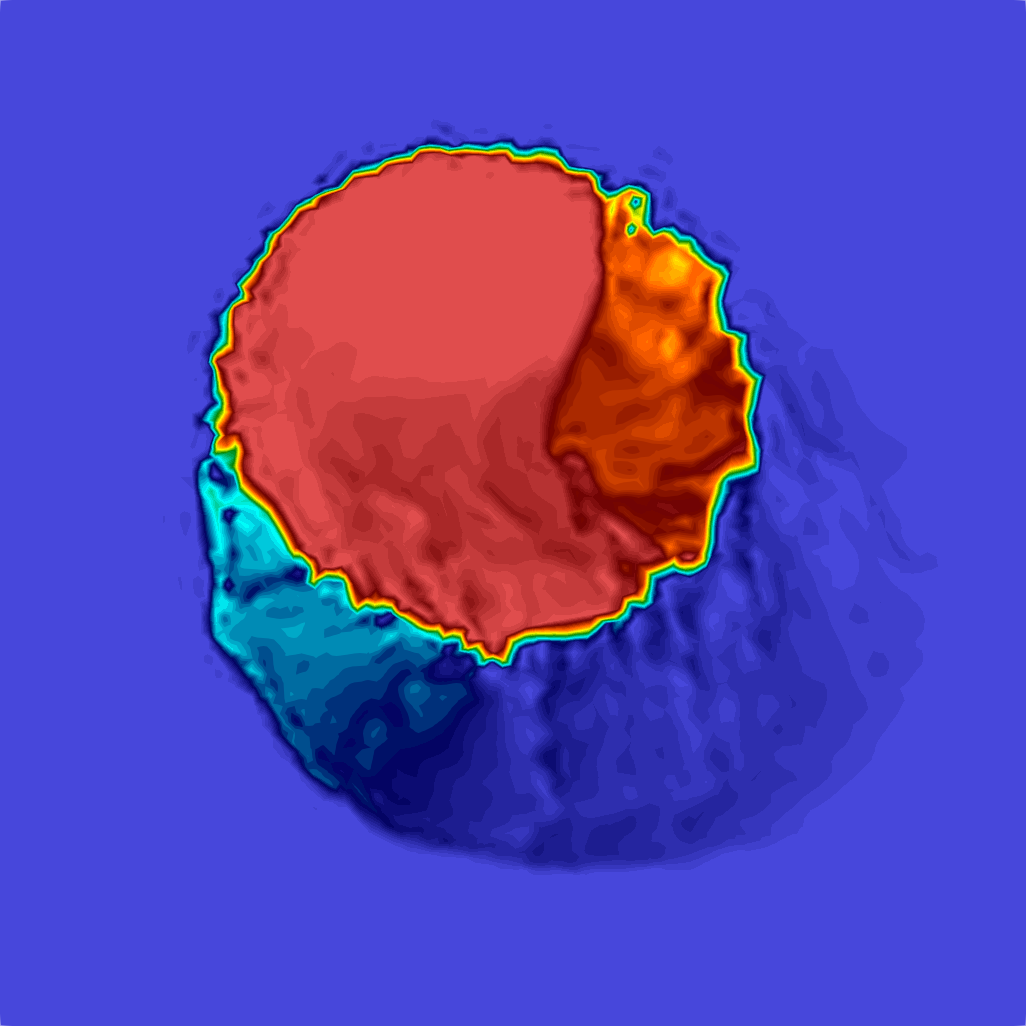}
    &
    \includegraphics[width=0.22\textwidth]{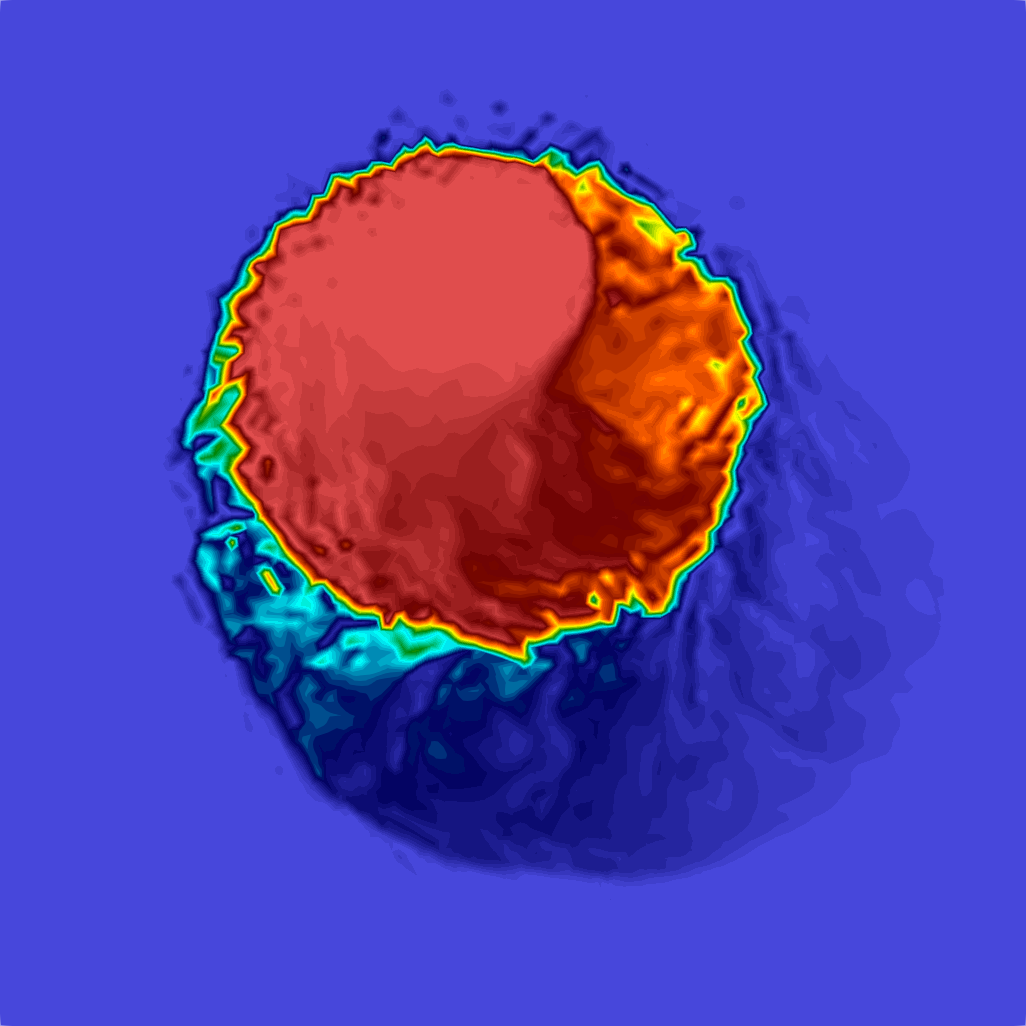}
    \\
    & 10042 & 9993 & 10048 & 10209
    \\[2pt]

    \rotatebox{90}{\small \hspace{-0.3in} dofs \hspace{0.5in} CIP}
    &
    \includegraphics[width=0.22\textwidth]{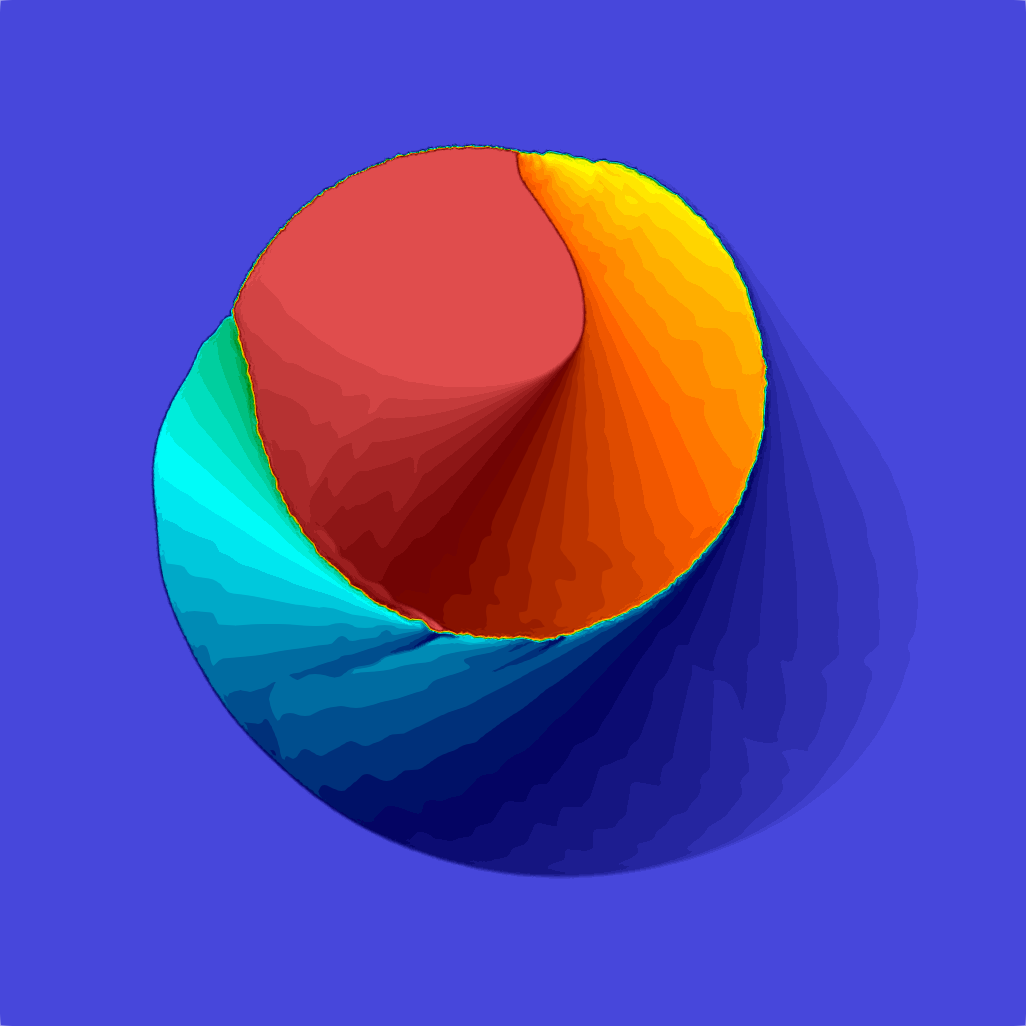}
    &
    \includegraphics[width=0.22\textwidth]{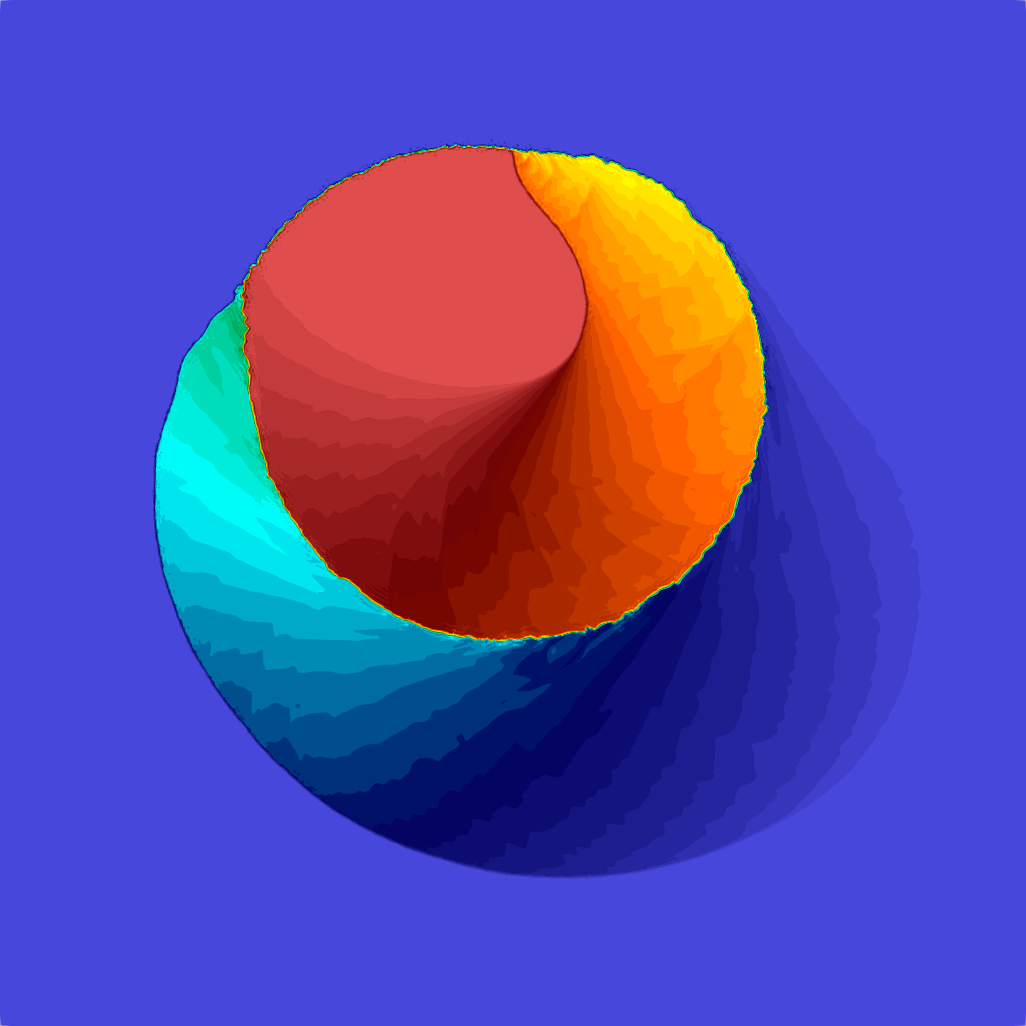}
    &
    \includegraphics[width=0.22\textwidth]{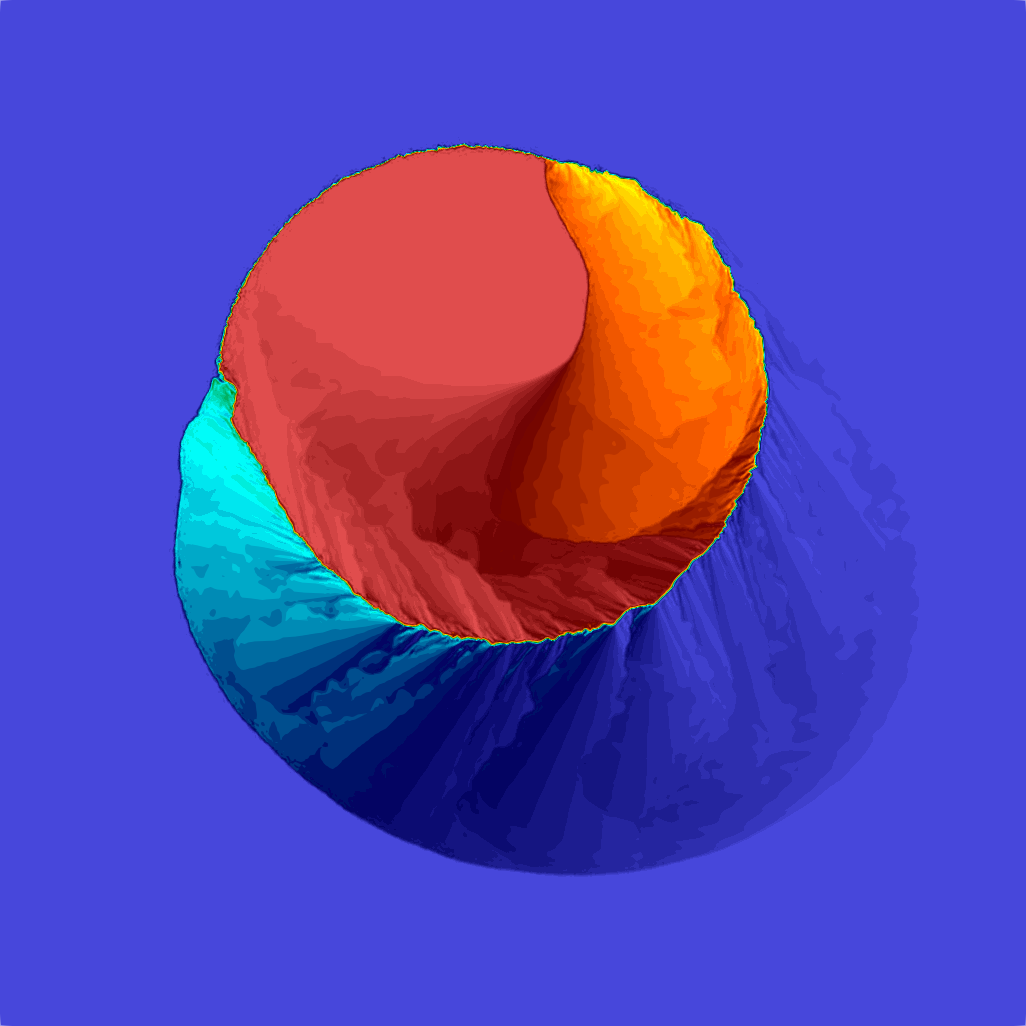}
    &
    \includegraphics[width=0.22\textwidth]{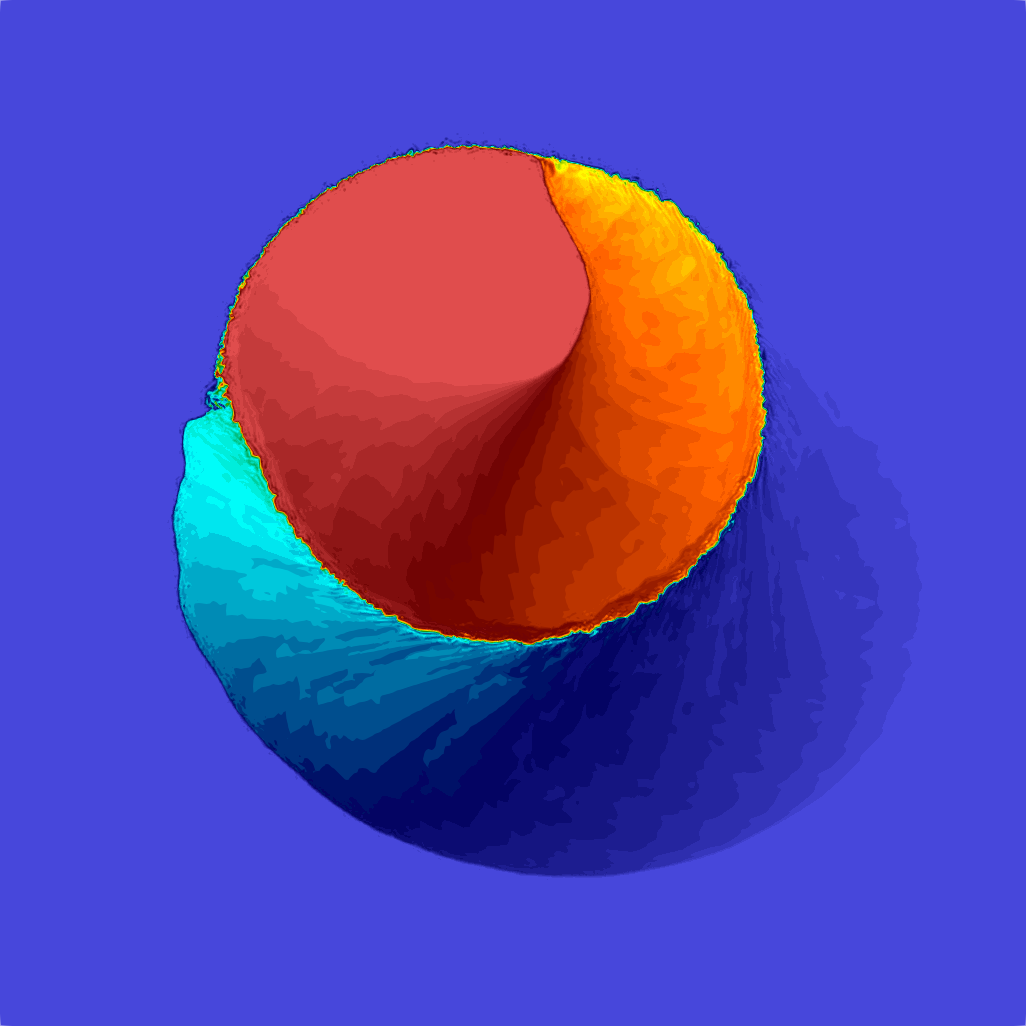}
    \\
    & 249773 & 249629 & 247699 & 251817
    \\
    \bottomrule
  \end{tabular}

  \caption{KPP problem: comparison of RV and CIP limited approximations at $t=1$ using polynomial degrees $\polP_1$--$\polP_4$ on unstructured Delaunay triangulations with about $10^4$ and $2.5\times10^5$ degrees of freedom.}
  \label{fig:kpp_comparison}
\end{figure}

\section{{Conclusions}}\label{sec:conclusions}
We have presented a high-order, locally conservative, invariant-domain preserving continuous finite element method for scalar conservation laws with standard Lagrange elements of arbitrary degree. The low-order scheme is the graph-viscosity method on the fine-$\polP_1$ submesh induced by the high-order nodes, so that it is compact and has positive lumped masses for every degree, while the high-order candidate is any stabilized Galerkin method that conserves the $\polP_k$ integral; here we used residual viscosity and continuous interior penalty methods. The difference between the two updates, each measured in its own conserved mass, has zero sum and is represented by antisymmetric edge fluxes through a graph-Poisson problem whose matrix is factorized once. Convex limiting of these fluxes gives an update in local flux form that conserves the $\polP_k$ mass and preserves the local bounds whenever the $\polP_k$ lumped masses are positive; for degrees with nonpositive lumped masses, such as $\polP_2$ and $\polP_4$, the mass defect of the limited elements is removed by an element-local exchange within the bounds, exactly whenever the elements have enough room (Section~\ref{Sec:local_exchange}).

The numerical experiments on unstructured meshes show that the limited scheme retains the $L^1$ convergence rate of the high-order baseline for all degrees, i.e., rate $k+1$ up to $\polP_4$ with CIP stabilization, with a loss in $L^\infty$ for $\polP_4$ near the points where the global bounds are attained. They also show robust approximations of discontinuous solutions, and, with the residual-viscosity baseline, convergence to the entropy solution for the KPP problem, which shows that the choice of the high-order baseline matters.

The reconstruction requires one global graph-Laplacian solve per forward-Euler substep of the Runge--Kutta method; the matrix depends only on the mesh and is factorized once, but the cost of the solves grows with the problem size; in our tests the limiter, including these solves, increases the time per step by a factor of about $1.75$ (Table~\ref{tab:cost}). Since the fluxes $\sfQ_i-\sfQ_j$ are conservative for any $\sfQ$, the solves can be inexact, e.g., a few iterations of algebraic multigrid preconditioned conjugate gradients warm-started from the previous stage; the iterated limiting can be stopped when a pass limits no flux, which in the smooth tests is almost always the case after the first pass; and a local reconstruction, discussed next, would replace the global solve by small independent problems. When the solution attains a global bound on a whole region, the nonlocal graph-Poisson fluxes are cut there and the accuracy is reduced; for $\polP_1$, where local antidiffusive fluxes are available, no such loss occurs. A local reconstruction of the fluxes for $k\ge2$, for instance element by element as in \citep{Kuzmin_Quezada_2020}, is therefore a natural remedy; the post-step correction restores the accuracy but is not local. The method does not depend on the spatial dimension, on the polynomial degree, or on the particular high-order method. Its extension to hyperbolic systems, such as the compressible Euler equations, requires applying the graph-Poisson reconstruction to each conserved variable and replacing the scalar bounds by the quasiconcave constraints of the invariant domain, as in \citep{Guermond_Nazarov_Popov_Tomas_2018}. This is the subject of future work.

\section*{Acknowledgements}
The author thanks R\'emi Abgrall (University of Zurich) for helpful discussions on the graph-Poisson flux reconstruction and its relation to residual distribution schemes.

\bibliographystyle{siam}  
\bibliography{ref}

\begin{thebibliography}{10}

\bibitem{Abgrall_2018}
{\sc R.~Abgrall}, {\em Some remarks about conservation for residual
  distribution schemes}, Comput. Methods Appl. Math., 18 (2018), pp.~327--351.

\bibitem{Abgrall_etal_2017}
{\sc R.~Abgrall, Q.~Viville, H.~Beaugendre, and C.~Dobrzynski}, {\em
  Construction of a {$p$}-adaptive continuous residual distribution scheme}, J.
  Sci. Comput., 72 (2017), pp.~1232--1268.

\bibitem{Fahad_2015}
{\sc F.~S.~A. Alrashed}, {\em Parallel multiphase {N}avier-{S}tokes solver},
  ProQuest LLC, Ann Arbor, MI, 2015.
\newblock Thesis (Ph.D.)--Texas A\&M University.

\bibitem{Anderson_etal_2017}
{\sc R.~Anderson, V.~Dobrev, T.~Kolev, D.~Kuzmin, M.~Quezada~de Luna,
  R.~Rieben, and V.~Tomov}, {\em High-order local maximum principle preserving
  ({MPP}) discontinuous {G}alerkin finite element method for the transport
  equation}, J. Comput. Phys., 334 (2017), pp.~102--124.

\bibitem{BarattaEtal2023}
{\sc I.~A. Baratta, J.~P. Dean, J.~S. Dokken, M.~Habera, J.~S. Hale, C.~N.
  Richardson, M.~E. Rognes, M.~W. Scroggs, N.~Sime, and G.~N. Wells}, {\em
  {DOLFINx}: the next generation {FEniCS} problem solving environment}.
\newblock preprint, 2023.

\bibitem{Bollobas_1998}
{\sc B.~Bollob\'as}, {\em Modern graph theory}, vol.~184 of Graduate Texts in
  Mathematics, Springer-Verlag, New York, 1998.

\bibitem{Brooks_Hughes_1982}
{\sc A.~N. Brooks and T.~J.~R. Hughes}, {\em Streamline
  upwind/{P}etrov-{G}alerkin formulations for convection dominated flows with
  particular emphasis on the incompressible {N}avier-{S}tokes equations},
  Comput. Methods Appl. Mech. Engrg., 32 (1982), pp.~199--259.

\bibitem{Burman_Hansbo_2004}
{\sc E.~Burman and P.~Hansbo}, {\em Edge stabilization for {G}alerkin
  approximations of convection-diffusion-reaction problems}, Comput. Methods
  Appl. Mech. Engrg., 193 (2004), pp.~1437--1453.

\bibitem{Chung_1997}
{\sc F.~R.~K. Chung}, {\em Spectral graph theory}, vol.~92 of CBMS Regional
  Conference Series in Mathematics, Conference Board of the Mathematical
  Sciences, Washington, DC; by the American Mathematical Society, Providence,
  RI, 1997.

\bibitem{Cotter_Kuzmin_2016}
{\sc C.~J. Cotter and D.~Kuzmin}, {\em Embedded discontinuous {G}alerkin
  transport schemes with localised limiters}, J. Comput. Phys., 311 (2016),
  pp.~363--373.

\bibitem{Dao_Nazarov_2022}
{\sc T.~A. Dao and M.~Nazarov}, {\em A high-order residual-based viscosity
  finite element method for the ideal {MHD} equations}, J. Sci. Comput., 92
  (2022), p.~77.

\bibitem{Douglas_Dupont_1976}
{\sc J.~Douglas, Jr. and T.~Dupont}, {\em Interior penalty procedures for
  elliptic and parabolic {G}alerkin methods}, in Computing methods in applied
  sciences ({S}econd {I}nternat. {S}ympos., {V}ersailles, 1975), vol.~Vol. 58
  of Lecture Notes in Phys., Springer, Berlin-New York, 1976, pp.~207--216.

\bibitem{Ern_Guermond_2013}
{\sc A.~Ern and J.-L. Guermond}, {\em Weighting the edge stabilization}, SIAM
  J. Numer. Anal., 51 (2013), pp.~1655--1677.

\bibitem{Geuzaine_Remacle_2009}
{\sc C.~Geuzaine and J.-F. Remacle}, {\em Gmsh: {A} 3-{D} finite element mesh
  generator with built-in pre- and post-processing facilities}, Internat. J.
  Numer. Methods Engrg., 79 (2009), pp.~1309--1331.

\bibitem{Guermond_1999}
{\sc J.-L. Guermond}, {\em Stabilization of {G}alerkin approximations of
  transport equations by subgrid modeling}, M2AN Math. Model. Numer. Anal., 33
  (1999), pp.~1293--1316.

\bibitem{Guermond_2026}
{\sc J.-L. Guermond}, {\em Locally conservative redistribution limiting and
  applications to the approximation of conservation equations, {P}art {II}},
  2026.
\newblock arXiv:2606.22839.

\bibitem{Guermond_Nazarov_2014}
{\sc J.-L. Guermond and M.~Nazarov}, {\em A maximum-principle preserving
  {$C^0$} finite element method for scalar conservation equations}, Comput.
  Methods Appl. Mech. Engrg., 272 (2014), pp.~198--213.

\bibitem{Guermond_Nazarov_Popov_2024}
{\sc J.-L. Guermond, M.~Nazarov, and B.~Popov}, {\em Finite element-based
  invariant-domain preserving approximation of hyperbolic systems: beyond
  second-order accuracy in space}, Comput. Methods Appl. Mech. Engrg., 418
  (2024), p.~116470.

\bibitem{Guermond_Nazarov_Popov_Tomas_2018}
{\sc J.-L. Guermond, M.~Nazarov, B.~Popov, and I.~Tomas}, {\em Second-order
  invariant domain preserving approximation of the {E}uler equations using
  convex limiting}, SIAM J. Sci. Comput., 40 (2018), pp.~A3211--A3239.

\bibitem{Guermond_Nazarov_Popov_Yong_2014}
{\sc J.-L. Guermond, M.~Nazarov, B.~Popov, and Y.~Yang}, {\em A second-order
  maximum principle preserving {L}agrange finite element technique for
  nonlinear scalar conservation equations}, SIAM J. Numer. Anal., 52 (2014),
  pp.~2163--2182.

\bibitem{Guermond_pasquetti_popov_JCP_2011}
{\sc J.-L. Guermond, R.~Pasquetti, and B.~Popov}, {\em Entropy viscosity method
  for nonlinear conservation laws}, J. Comput. Phys., 230 (2011),
  pp.~4248--4267.

\bibitem{Guermond_Popov_2016}
{\sc J.-L. Guermond and B.~Popov}, {\em Invariant domains and first-order
  continuous finite element approximation for hyperbolic systems}, SIAM J.
  Numer. Anal., 54 (2016), pp.~2466--2489.

\bibitem{Guermond_Popov_2017}
\leavevmode\vrule height 2pt depth -1.6pt width 23pt, {\em Invariant domains
  and second-order continuous finite element approximation for scalar
  conservation equations}, SIAM J. Numer. Anal., 55 (2017), pp.~3120--3146.

\bibitem{Guermond_Popov_Yang_2017}
{\sc J.-L. Guermond, B.~Popov, and Y.~Yang}, {\em The effect of the consistent
  mass matrix on the maximum-principle for scalar conservation equations}, J.
  Sci. Comput., 70 (2017), pp.~1358--1366.

\bibitem{Guermond_Wang_2025}
{\sc J.-L. Guermond and Z.~Wang}, {\em Mass conservative limiting and
  applications to the approximation of the steady-state radiation transport
  equations}, J. Comput. Phys., 521 (2025), p.~113531.

\bibitem{Hajduk_2021}
{\sc H.~Hajduk}, {\em Monolithic convex limiting in discontinuous {G}alerkin
  discretizations of hyperbolic conservation laws}, Comput. Math. Appl., 87
  (2021), pp.~120--138.

\bibitem{Hajduk_etal_2020a}
{\sc H.~Hajduk, D.~Kuzmin, T.~Kolev, and R.~Abgrall}, {\em Matrix-free subcell
  residual distribution for {B}ernstein finite element discretizations of
  linear advection equations}, Comput. Methods Appl. Mech. Engrg., 359 (2020),
  pp.~112658, 26.

\bibitem{Hajduk_etal_2020b}
{\sc H.~Hajduk, D.~Kuzmin, T.~Kolev, V.~Tomov, I.~Tomas, and J.~N. Shadid},
  {\em Matrix-free subcell residual distribution for {B}ernstein finite
  elements: monolithic limiting}, Comput. \& Fluids, 200 (2020), pp.~104451,
  13.

\bibitem{Kraaijevanger_1991}
{\sc J.~F. B.~M. Kraaijevanger}, {\em Contractivity of {R}unge-{K}utta
  methods}, BIT, 31 (1991), pp.~482--528.

\bibitem{Kronbichler_Maier_Tomas_2025}
{\sc M.~Kronbichler, M.~Maier, and I.~Tomas}, {\em Graph-based methods for
  hyperbolic systems of conservation laws using discontinuous space
  discretizations}, Commun. Comput. Phys., 38 (2025), pp.~74--108.

\bibitem{Kurganov_2007}
{\sc A.~Kurganov, G.~Petrova, and B.~Popov}, {\em Adaptive semidiscrete
  central-upwind schemes for nonconvex hyperbolic conservation laws}, SIAM J.
  Sci. Comput., 29 (2007), pp.~2381--2401.

\bibitem{Kuzmin_2020}
{\sc D.~Kuzmin}, {\em Monolithic convex limiting for continuous finite element
  discretizations of hyperbolic conservation laws}, Comput. Methods Appl. Mech.
  Engrg., 361 (2020), pp.~112804, 28.

\bibitem{Kuzmin_Moller_2005}
{\sc D.~Kuzmin and M.~M\"oller}, {\em Algebraic flux correction. {I}. {S}calar
  conservation laws}, in Flux-corrected transport, Sci. Comput., Springer,
  Berlin, 2005, pp.~155--206.

\bibitem{Kuzmin_Quezada_2020}
{\sc D.~Kuzmin and M.~Quezada~de Luna}, {\em Subcell flux limiting for
  high-order {B}ernstein finite element discretizations of scalar hyperbolic
  conservation laws}, J. Comput. Phys., 411 (2020), pp.~109411, 19.

\bibitem{Kuzmin_etal_2022}
{\sc D.~Kuzmin, M.~Quezada~de Luna, D.~I. Ketcheson, and J.~Gr\"ull}, {\em
  Bound-preserving flux limiting for high-order explicit {R}unge-{K}utta time
  discretizations of hyperbolic conservation laws}, J. Sci. Comput., 91 (2022),
  p.~21.

\bibitem{Kuzmin_2002}
{\sc D.~Kuzmin and S.~Turek}, {\em Flux correction tools for finite elements},
  J. Comput. Phys., 175 (2002), pp.~525--558.

\bibitem{Kuzmin_Vedral_2023}
{\sc D.~Kuzmin and J.~Vedral}, {\em Dissipation-based {WENO} stabilization of
  high-order finite element methods for scalar conservation laws}, J. Comput.
  Phys., 487 (2023), p.~112153.

\bibitem{Liu_Tadmor_1998}
{\sc X.-D. Liu and E.~Tadmor}, {\em Third order nonoscillatory central scheme
  for hyperbolic conservation laws}, Numer. Math., 79 (1998), pp.~397--425.

\bibitem{Lohmann_etal_2017}
{\sc C.~Lohmann, D.~Kuzmin, J.~N. Shadid, and S.~Mabuza}, {\em Flux-corrected
  transport algorithms for continuous {G}alerkin methods based on high order
  {B}ernstein finite elements}, J. Comput. Phys., 344 (2017), pp.~151--186.

\bibitem{Nazarov_2013}
{\sc M.~Nazarov}, {\em Convergence of a residual based artificial viscosity
  finite element method}, Comput. Math. Appl., 65 (2013), pp.~616--626.

\bibitem{Nazarov_Hoffman_2013}
{\sc M.~Nazarov and J.~Hoffman}, {\em Residual-based artificial viscosity for
  simulation of turbulent compressible flow using adaptive finite element
  methods}, Internat. J. Numer. Methods Fluids, 71 (2013), pp.~339--357.

\bibitem{Nazarov_Larcher_2017}
{\sc M.~Nazarov and A.~Larcher}, {\em Numerical investigation of a viscous
  regularization of the {E}uler equations by entropy viscosity}, Comput.
  Methods Appl. Mech. Engrg., 317 (2017), pp.~128--152.

\bibitem{Panzer_2021}
{\sc W.~Pazner}, {\em Sparse invariant domain preserving discontinuous
  {G}alerkin methods with subcell convex limiting}, Comput. Methods Appl. Mech.
  Engrg., 382 (2021), p.~113876.

\bibitem{Quezada_2016}
{\sc M.~Quezada De~Luna}, {\em High-order maximum principle preserving ({MPP})
  techniques for solving conservation laws with applications on multiphase
  flow}, ProQuest LLC, Ann Arbor, MI, 2016.
\newblock Thesis (Ph.D.)--Texas A\&M University.

\bibitem{Sanders_1988}
{\sc R.~Sanders}, {\em A third-order accurate variation nonexpansive difference
  scheme for single nonlinear conservation laws}, Math. Comp., 51 (1988),
  pp.~535--558.

\bibitem{Schar_Smolarkiewicz_1996}
{\sc C.~Sch\"ar and P.~K. Smolarkiewicz}, {\em A synchronous and iterative
  flux-correction formalism for coupled transport equations}, J. Comput. Phys.,
  128 (1996), pp.~101--120.

\bibitem{Striernstrom_2021}
{\sc V.~Stiernstr\"om, L.~Lundgren, M.~Nazarov, and K.~Mattsson}, {\em A
  residual-based artificial viscosity finite difference method for scalar
  conservation laws}, J. Comput. Phys., 430 (2021), p.~110100.

\bibitem{Vilar_Abgrall_2024}
{\sc F.~Vilar and R.~Abgrall}, {\em A posteriori local subcell correction of
  high-order discontinuous {G}alerkin scheme for conservation laws on
  two-dimensional unstructured grids}, SIAM J. Sci. Comput., 46 (2024),
  pp.~A851--A883.

\bibitem{Zalesak_1979}
{\sc S.~T. Zalesak}, {\em Fully multidimensional flux-corrected transport
  algorithms for fluids}, J. Comput. Phys., 31 (1979), pp.~335--362.

\bibitem{Zhang_2010}
{\sc X.~Zhang and C.-W. Shu}, {\em On maximum-principle-satisfying high order
  schemes for scalar conservation laws}, J. Comput. Phys., 229 (2010),
  pp.~3091--3120.

\bibitem{Zhang_Xia_Shu_2012}
{\sc X.~Zhang, Y.~Xia, and C.-W. Shu}, {\em Maximum-principle-satisfying and
  positivity-preserving high order discontinuous {G}alerkin schemes for
  conservation laws on triangular meshes}, J. Sci. Comput., 50 (2012),
  pp.~29--62.

\end{thebibliography}
\end{document}